\documentclass[12pt]{article}

\usepackage[margin=2.3cm]{geometry}
\usepackage{setspace}
\RequirePackage[OT1]{fontenc}

\usepackage{yk, amsthm, amsmath, amsfonts, amssymb, mathrsfs, xcolor, dsfont, color, mathtools, graphicx, float, hyperref, lipsum, accents, xr, cancel, comment, booktabs}
\usepackage{breakurl} 
\usepackage{xurl}
\usepackage[numbers]{natbib}

\usepackage[shortlabels]{enumitem}
\usepackage[normalem]{ulem}
\usepackage[ruled,vlined]{algorithm2e}
\usepackage[ruled,vlined]{algorithm2e}
\usepackage[noend]{algpseudocode}
\usepackage[ruled,vlined]{algorithm2e}
\usepackage{etoolbox}
\newcommand{\s}{\text{sign}}

\renewcommand{\P}{\mathbb{P}}

\newcommand{\indep}{\rotatebox[origin=c]{90}{$\models$}}

\makeatletter
\patchcmd{\@makecaption}
  {\parbox}
  {\advance\@tempdima-\fontdimen2} 
  {}{}
\makeatother

\title{Asymptotic normality of a change plane estimator in fixed dimension with near-optimal rate}

\author{Debarghya Mukherjee$^{\dagger}$, Moulinath Banerjee$^{\dagger}$, \\ Debasri Mukherjee$^*$ and Ya'acov Ritov$^{\dagger}$ \\\\ $^{\dagger}$Department of Statistics, University of Michigan \\ $^*$Department of Economics, Western Michigan University}
\begin{document}

\maketitle
\begin{abstract}
Linear thresholding models postulate that the conditional distribution of a response variable in terms of covariates differs on the two sides of a (typically unknown) hyperplane in the covariate space. A key goal in such models is to learn about this separating hyperplane. Exact likelihood or least squares methods to estimate the thresholding parameter involve an indicator function which make them difficult to optimize and are, therefore, often tackled by using a surrogate loss that uses a smooth approximation to the indicator. In this paper, we demonstrate that the resulting estimator is asymptotically normal with a near optimal rate of convergence: $n^{-1}$ up to a log factor, in both classification and regression thresholding models. This is substantially faster than the currently established convergence rates of smoothed estimators for similar models in the statistics and econometrics literatures. We also present a real-data application of our approach to an environmental data set where $CO_2$ emission is explained in terms of a separating hyperplane defined through per-capita GDP and urban agglomeration. 
\end{abstract}

\section{Introduction}
The simple linear regression model assumes a uniform linear relationship between the covariate and the response, in the sense that the regression parameter $\beta$ is the same over the entire covariate domain. In practice, the situation can be more complicated: for instance, the regression parameter may differ from sub-population to sub-population within a large (super-) population. Some common techniques to account for such heterogeneity include mixed linear models, introducing an interaction effect, or fitting different models among each sub-population which corresponds to a supervised classification setting where the true groups (sub-populations) are \emph{a priori known}. 
	\newline
	\newline
	\indent A more difficult scenario arises when the sub-populations are unknown, in which case regression and classification must happen simultaneously. Consider the scenario where the conditional mean of $Y_i$ given $X_i$ is different for different unknown sub-groups. A  well-studied treatment of this problem -- the so-called change point problem -- considers a simple thresholding model where membership in a sub-group is determined by whether a real-valued observable $X$ falls to the left or right of an unknown parameter $\gamma$. More recently, there has been work for multi-dimensional covariates, namely when the membership is determined by which side a random vector $X$ falls with respect to an hyperplane with unknown normal vector $\theta_0$. A concrete example appears in \cite{wei2014latent} who extend the linear thresholding model due to  \cite{kang2011new} to general dimensions: 
	\begin{eqnarray}\label{eq:weimodel}
	Y=\mu_1\cdot 1_{X^{\top}\theta_0\geq 0}+\mu_2\cdot 1_{X^{\top}\theta_0<0}+\varepsilon\,,
	\end{eqnarray}
	and studied computational algorithms and consistency of the same. This model and others with similar structure, called \emph{change plane models}, are useful in various fields of research, e.g. modeling treatment effect heterogeneity in drug treatment (\cite{imai2013estimating}), modeling sociological data on voting and employment (\cite{imai2013estimating}), or cross country growth regressions in econometrics 
(\cite{seo2007smoothed}).
	\newline
	\newline
\indent Other aspects of this model have also been investigated. \cite{fan2017change} examined the change plane model from the statistical testing point of view, with the null hypothesis being the absence of a separating hyperplane. They proposed a test statistic, studied its asymptotic distribution and provided sample size recommendations for achieving target values of power. \cite{li2018multi} extended the change point detection problem in the multi-dimensional setup by considering the case where $X^{\top}\theta_0$ forms a multiple change point data sequence. 
The key difficultly with change plane type models is the inherent discontinuity in the optimization criteria involved where the parameter of interest appears as an argument to some indicator function, rendering the optimization extremely hard. To alleviate this, one option is to kernel smooth the indicator function, an approach that was adopted by Seo and Linton \cite{seo2007smoothed} in a version of the change-plane problem, motivated by earlier results of Horowitz \cite{horowitz1992smoothed} that dealt with a smoothed version of the maximum score estimator. Their model has an additive structure of the form:
\[Y_t = \beta^{\top}X_t + \delta^{\top} \tilde{X}_t \mathds{1}_{Q_t^{\top} \boldmath \psi > 0} + \epsilon_t \,,\]
where $\psi$ is the (fixed) change-plane parameter, and $t$ can be viewed as a time index. Under a set of assumptions on the model (Assumptions 1 and 2 of their paper), they showed asymptotic normality of their estimator of $\psi$ obtained by minimizing a smoothed least squares criterion
that uses a differentiable distribution function $\mathcal{K}$. The rate of convergence of $\hat{\psi}$ to the truth was shown to be $\sqrt{n/\sigma_n}$ where $\sigma_n$ was the bandwidth parameter used to smooth the least squares function. As noted in their Remark 3, under the special case of i.i.d. observations, their requirement that $\log n/(n \sigma_n^2) \rightarrow 0$ translates to a maximal convergence rate of $n^{3/4}$ up to a logarithmic factor. The work of \cite{li2018multi} who considered multiple parallel change planes (determined by a fixed dimensional normal vector) and high dimensional linear models in the regions between consecutive hyperplanes also builds partly upon the methods of \cite{seo2007smoothed} and obtains the same (almost) $n^{3/4}$ rate for the normal vector (as can be seen by putting Condition 6 in their paper in conjunction with the conclusion of Theorem 3). 
\\\\

While it is established that the condition $n\sigma_n^2 \to \infty$ is sufficient (upto a log factor) for achieving asymptotic normality of the smoothed estimator, there is no result in the existing literature to ascertain whether its necessity. Intuitively speaking, the necessary condition for asymptotic normality ought to be $n \sigma_n \to 0$, as this will ensure a growing number of observations in a $\sigma_n$ neighborhood around the true hyperplane, allowing the central limit theorem to kick in. In this paper we \emph{bridge this gap} by proving that asymptotic normality of the smoothed change point estimator is, in fact, achievable with $n \sigma_n \to \infty$. 
This implies that the best possible rate of convergence of the smoothed estimator can be arbitrarily close to $n^{-1}$, the minimax optimal rate of estimation for this problem. To demonstrate this, we focus on two change plane estimation problems, one with a continuous and another with a binary response. The continuous response model we analyze here is the following: 
\begin{equation}
\label{eq:regression_main_eqn}
Y_i = \beta_0^{\top}X_i + \delta_0^{\top}X_i\mathds{1}_{Q_i^{\top}\psi_0 > 0} + \eps_i \,.
\end{equation}
for i.i.d. observations $\{(X_i, Y_i, Q_i\}_{i=1}^n$, where the zero-mean transitory shocks $\eps_i \indep (X_i, Q_i)$. Our calculation can be easily extended to the case when the covariates on the either side of the change hyperplane are different and $\bbE[\eps \mid X, Q] = 0$ with more tedious bookkeeping. As this generalization adds little of interest, conceptually, to our proof, we posit the simpler model for ease of understanding.
As the parameter $\psi_0$ is only identifiable upto its norm, we assume that the first co-ordinate is $1$ (along the lines of \cite{seo2007smoothed}) which removes one degree of freedom and makes the parameter identifiable. 
\\\\
To illustrate that a similar phenomenon transpires with binary response, we also study a canonical version of such a model which can be briefly described as follows: The covariate $Q \sim P$ where $P$ is distribution on $\mathbb{R}^d$ and the conditional distribution of $Y$ given $Q$ is modeled as follows: 
\begin{equation}
\label{eq:classification_eqn}
P(Y=1|Q) = \alpha_0 \mathds{1}(Q^{\top}\psi_0 \le 0) + \beta_0\mathds{1}(Q^{\top}\psi_0 > 0)
\end{equation}
for some parameters $\alpha_0, \beta_0\in (0,1)$ and $\psi_0\in\mathbb{R}^d$ (with first co-ordinate being one for identifiability issue as for the continuous response model), the latter being of primary interest for estimation. 
This model is identifiable up to a permutation of $(\alpha_0, \beta_0)$, so we further assume $\alpha_0 < \beta_0$. For both models, we show that $\sqrt{n/\sigma_n}(\hat \psi - \psi_0)$ converges to zero-mean normal distribution as long as $n \sigma_n \to \infty$ but the calculations for the binary model are completely relegated to Appendix \ref{sec:supp_classification}. 
\\\\
{\bf Organization of the paper:} The rest of the paper is organized as follows: In Section \ref{sec:theory_regression} we present the methodology, the statement of the asymptotic distributions and a sketch of the proof for the continuous response model \eqref{eq:regression_main_eqn}. In Section \ref{sec:classification_analysis} we briefly describe the binary response model \eqref{eq:classification_eqn} and related assumptions, whilst the details can be found in the supplementary document. In Section \ref{sec:simulation} we present some simulation results, both for the binary and the continuous response models to study the effect of the bandwidth on the quality of the normal approximation in finite samples. In Section \ref{sec:real_data}, we present a real data analysis where we analyze the effect of  income and urbanization on the $CO_2$ emission in different countries. 
\\\\
{\bf Notations: } Before delving into the technical details, we first setup some notations here. We assume from now on, $X \in \reals^p$ and $Q \in \reals^d$. For any vector $v$ we define by $\tilde v$ as the vector with all the co-ordinates expect the first one. We denote by $K$ the kernel function used to smooth the indicator function. For any matrix $A$, we denote by $\|A\|_2$ (or $\|A\|_F$) as its Frobenious norm and $\|A\|_{op}$ as its operator norm. For any vector, $\| \cdot \|_2$ denotes its $\ell_2$ norm.

\renewcommand{\th}{\theta}
\section{Methodology and Theory for Continuous Response Model}
\label{sec:theory_regression}
In this section we present our analysis for the continuous response model. Without smoothing, the original estimating equation is: 
$$
f_{\beta, \delta, \psi}(Y, X, Q) = \left(Y - X^{\top}\beta - X^{\top}\delta\mathds{1}_{Q^{\top}\psi > 0}\right)^2  
$$
and we estimate the parameters as: 
\begin{align}
\label{eq:ls_estimator}
\left(\hat \beta^{LS}, \hat \delta^{LS}, \hat \psi^{LS}\right) & = \argmin_{(\beta, \delta, \psi) \in \Theta} \bbP_n f_{\beta, \delta, \psi} \notag \\
& :=  \argmin_{(\beta, \delta, \psi) \in \Theta}\bbM_n(\beta, \delta, \psi)\,.
\end{align}
where $\bbP_n$ is empirical measure based on i.i.d. observations $\{(X_i, Y_i, Q_i)\}_{i=1}^n$ and $\Theta$ is the parameter space. Henceforth, we assume $\Theta$ is a compact subset of dimension $\reals^{2p+d}$. We also define $\theta = (\beta, \delta, \psi)$, i.e. all the parameters together as a vector and by $\theta_0$ is used to denote the true parameter vector $(\beta_0, \delta_0, \psi_0)$. Some modification of equation \eqref{eq:ls_estimator} leads to the following:  
\begin{align*}
(\hat \beta^{LS}, \hat \delta^{LS}, \hat \psi^{LS}) & = \argmin_{\beta, \delta, \psi}  \sum_{i=1}^n \left(Y_i - X_i^{\top}\beta - X_i^{\top}\delta\mathds{1}_{Q_i^{\top}\psi > 0}\right)^2 \\ 
& = \argmin_{\beta, \delta, \psi}  \sum_{i=1}^n \left[\left(Y_i - X_i^{\top}\beta\right)^2\mathds{1}_{Q_i^{\top}\psi_0 \le 0} \right. \\
&  \hspace{14em} \left. + \left(Y_i - X_i^{\top}\beta - X_i^{\top}\delta\right)^2\mathds{1}_{Q_i^{\top}\psi > 0} \right] \\
& = \argmin_{\beta, \delta, \psi}  \sum_{i=1}^n \left[\left(Y_i - X_i^{\top}\beta\right)^2 + \left\{\left(Y_i - X_i^{\top}\beta - X_i^{\top}\delta\right)^2 \right. \right. \\
& \hspace{17em} \left. \left. - \left(Y_i - X_i^{\top}\beta\right)^2\right\}\mathds{1}_{Q_i^{\top}\psi > 0} \right] 
\end{align*}
Typical empirical process calculations yield under mild conditions: 
$$
\|\hat \beta^{LS} - \beta_0\|^2 + \|\hat \delta^{LS} - \delta_0\|^2 + \|\hat \psi^{LS} - \psi_0 \|_2 = O_p(n^{-1})
$$
but inference is difficult as the limit distribution is unknown, and in any case, would be a highly non-standard distribution. Recall that even in the one-dimensional change point model with fixed jump size, the least squares change point estimator converges at rate $n$ to the truth with a non-standard limit distribution, namely a minimizer of a two-sided compound Poisson process (see \cite{lan2009change} for more details). To obtain a computable estimator with tractable limiting distribution, we resort to a smooth approximation of the indicator function in \eqref{eq:ls_estimator} using a distribution kernel with suitable bandwidth, i.e we replace $\mathds{1}_{Q_i^{\top}\psi > 0}$ by $K(Q_i^{\top}\psi/\sigma_n)$ for some appropriate distribution function $K$ and bandwidth $\sigma_n$, i.e. 
\begin{align*}
(\hat \beta^S, \hat \delta^S, \hat \psi^S) & = \argmin_{\beta, \delta, \psi} \left\{ \frac1n \sum_{i=1}^n \left[\left(Y_i - X_i^{\top}\beta\right)^2 + \left\{\left(Y_i - X_i^{\top}\beta - X_i^{\top}\delta\right)^2 \right. \right. \right. \\
& \hspace{15em} \left. \left. \left. - \left(Y_i - X_i^{\top}\beta\right)^2\right\}K\left(\frac{Q_i^{\top}\psi}{\sigma_n}\right) \right] \right\} \\
& = \argmin_{(\beta, \delta, \psi) \in \Theta} \bbP_n f^s_{(\beta, \delta, \psi)}(X, Y, Q) \\
& := \argmin_{\theta \in \Theta} \bbM^s_n(\theta) \,.
\end{align*}
Define $\bbM$ (resp. $\bbM^s$) to be the population counterpart of $\bbM_n$ and $\bbM_n^s$ respectively which are defined as: 
\begin{align*}
\bbM(\th) & = \bbE\left(Y - X^{\top}\beta\right)^2 + \bbE\left(\left[-2\left(Y_i - X^{\top}\beta\right)X^{\top}\delta + (X^{\top}\delta)^2\right] \mathds{1}_{Q^{\top}\psi > 0}\right) \,, \\
\bbM^s(\theta) & = \bbE\left[(Y - X^{\top}\beta)^2 + \left\{-2(Y-X^{\top}\beta)(X^{\top}\delta) + (X^{\top}\delta)^2\right\}K\left(\frac{Q^{\top}\psi}{\sigma_n}\right)\right] \,.
\end{align*}
As noted in the proof of \textcolor{blue}{Seo and Linton}, the assumption $\log{n}/n\sigma_n^2 \to 0$ was only used to show: 
$$
\frac{\left\|\hat \psi^s - \psi_0\right\|}{\sigma_n} = o_p(1) \,.
$$
In this paper, we show that one can achieve the same conclusion as long as $n\sigma_n \to \infty$.  The rest of the proof for the normality is similar to that of \cite{seo2007smoothed}, we will present it briefly for the ease the readers. The proof is quite long and technical, therefore we break the proof into several lemmas. We, first, list our assumptions: 
\begin{assumption}
\label{eq:assm}
\begin{enumerate}
\item Define $f_\psi(\cdot \mid \tilde Q)$ to be the conditional distribution of $Q^{\top}\psi$ given $\tilde Q$. (In particular we will denote by $f_0(\cdot \mid \tilde q)$ to be conditional distribution of $Q^{\top}\psi_0$ given $\tilde Q$ and $f_s(\cdot \mid \tilde q)$ to be the conditional distribution of $Q^{\top}\psi_0^s$ given $\tilde Q$. Assume that there exists $F_+$ such that $\sup_t f_0(t | \tilde Q) \le F_+$ almost surely on $\tilde Q$ and for all $\psi$ in a neighborhood of $\psi_0$ (in particular for $\psi_0^s$). Further assume that $f_\psi$ is differentiable and the derivative is bounded by $F_+$ for all $\psi$ in a neighborhood of $\psi_0$ (again in particular for $\psi_0^s$).
\vspace{0.1in}
\item Define $g(Q) = \var(X \mid Q)$. There exists $c_-$ and $c_+$ such that $c_- \le \lambda_{\min}(g(Q)) \le \lambda_{\max}(g(Q)) \le c_+$ almost surely. Also assume that $g$ is a Lipschitz with constant $G_+$ with respect to $Q$. 
\vspace{0.1in}
\item There exists $p_+ < \infty$ and $p_- > 0, r > 0$ such that: 
$$
p_- \|\psi - \psi_0\| \le \bbP\left(\s\left(Q^{\top}\psi\right) \neq \s\left(Q^{\top}\psi_0\right)\right) \le p_+ \|\psi - \psi_0\| \,,
$$
for all $\psi$ such that $\|\psi - \psi_0\| \le r$. 
\vspace{0.1in}
\item For all $\psi$ in the parameter space $0 < \bbP\left(Q^{\top}\psi > 0\right) < 1$. 
\vspace{0.1in} 
\item Define $m_2(Q) = \bbE\left[\|X\|^2 \mid Q\right]$ and $m_4(Q) = \bbE\left[\|X\|^4 \mid Q\right]$. Assume $m_2, m_4$ are bounded Lipschitz function of $Q$.  
\end{enumerate}
\end{assumption}

\subsection{Sufficient conditions for above assumptions }
We now demonstrate some sufficient conditions for the above assumptions to hold. The first condition is essentially a condition on the conditional density of the first co-ordinate of $Q$ given all other co-ordinates. If this conditional density is bounded and has bounded derivative, then first assumption is satisfied. This condition is satisfied in fair generality. The second assumption implies that the conditional distribution of X given Q has variance in all the direction over all $Q$. This is also very weak condition, as is satisfied for example if X and Q and independent (with $X$ has non-degenerate covariance matrix) or $(X, Q)$ are jointly normally distributed to name a few. This condition can further be weaken by assuming that the maximum and minimum eigenvalues of $\bbE[g(Q)]$ are bounded away from $\infty$ and $0$ respectively but it requires more tedious book-keeping. The third assumption is satisfied as long as as $Q^{\top}\psi$ has non-zero density near origin, while the fourth assumption merely states that the support of  $Q$ is not confined to one side of the hyperplane for any hyperplane and a simple sufficient condition for this is $Q$ has continuous density with non-zero value at the origin. The last assumption is analogous to the second assumption for the conditional fourth moment which is also satisfied in fair generality. 
\\\\
\noindent
{\bf Kernel function and bandwidth: } We take $K(x) = \Phi(x)$ (distribution of standard normal random variable) for our analysis. For the bandwidth we assume $n\sigma_n^2 \to 0$ and $n \sigma_n \to \infty$ as the other case, (i.e. $n\sigma_n^2 \to \infty$) is already established in \cite{seo2007smoothed}. 
\\\\
\noindent
Based on Assumption \ref{eq:assm} and our choice of kernel and bandwidth we establish the following theorem: 
\begin{theorem}
\label{thm:regression}
Under Assumption \ref{eq:assm} and the above choice of kernel and bandwidth we have: 
$$
\sqrt{n}\begin{pmatrix}\begin{pmatrix} \hat \beta^s \\ \hat \delta^s \end{pmatrix} - \begin{pmatrix} \beta_0 \\ \delta_0 \end{pmatrix} \end{pmatrix} \overset{\mathscr{L}}{\implies} \cN(0, \Sigma_{\beta, \delta})
$$
and 
$$
\sqrt{n/\sigma_n} \left(\hat \psi^s - \psi_0\right) \overset{\mathscr{L}}{\implies} \cN(0, \Sigma_\psi) \,,
$$
for matrices $\Sigma_{\beta, \delta}$ and $\Sigma_\psi$ mentioned explicitly in the proof. Moreover they are asymptotically independent. 
\end{theorem}
The proof of the theorem is relatively long, so we break it into several lemmas. We  provide a roadmap of the proof in this section while the elaborate technical derivations of the supporting lemmas can be found in Appendix. Let $\nabla \bbM_n^s(\theta)$ and $\nabla^2 \bbM_n^s(\theta)$ be the gradient and Hessian of $\bbM_n^s(\theta)$ with respect to $\theta$.  As $\hat \theta^s$ minimizes $\bbM_n^s(\theta)$, we have from the first order condition, $\nabla \bbM_n^s(\hat \theta^s) = 0$. Using one step Taylor expansion we have:
\allowdisplaybreaks 
\begin{align*}
\label{eq:taylor_first}
\nabla \bbM_n^s(\hat \theta^s) = \nabla \bbM_n^s(\theta_0) + \nabla^2 \bbM_n^s(\theta^*)\left(\hat \theta^s - \theta_0\right) = 0
\end{align*}
i.e.
\begin{equation}
\label{eq:main_eq}  
\left(\hat{\theta}^s - \theta_0\right) = -\left(\nabla^2 \bbM_n^s(\theta^*)\right)^{-1} \nabla \bbM_n^s(\theta_0)
\end{equation}
for some intermediate point $\th^*$ between $\hat \theta^s$ and $\theta_0$. 
Following the notation of \cite{seo2007smoothed}, define a diagonal matrix $D_n$ of dimension $2p + d$ with first $2p$ elements being 1 and the last $d$ elements being $\sqrt{\sigma_n}$. 
 we can write: 
\begin{align}
\sqrt{n}D_n^{-1}(\hat \theta^s - \theta_0) & = - \sqrt{n}D_n^{-1}\nabla^2\bbM_n^s(\theta^*)^{-1}\nabla \bbM_n^s(\theta_0)  \notag \\
\label{eq:taylor_main} & = \begin{pmatrix} \nabla^2\bbM_n^{s, \gamma}(\theta^*) & \sqrt{\sigma_n}\nabla^2\bbM_n^{s, \gamma \psi}(\theta^*) \\
\sqrt{\sigma_n}\nabla^2\bbM_n^{s, \gamma \psi}(\theta^*) & \sigma_n\nabla^2\bbM_n^{s, \psi}(\theta^*)\end{pmatrix}^{-1}\begin{pmatrix} \sqrt{n}\nabla \bbM_n^{s, \gamma}(\theta_0) \\ \sqrt{n\sigma_n}\nabla \bbM_n^{s, \psi}(\theta_0)\end{pmatrix}
\end{align}
where $\gamma = (\beta, \delta) \in \reals^{2p}$. The following lemma establishes the asymptotic properties of $\nabla \bbM_n^s(\theta_0)$: 
\begin{lemma}[Asymptotic Normality of $\nabla \bbM_n^s(\theta_0)$]
\label{asymp-normality}
\label{asymp-normality}
Under assumption \ref{eq:assm} we have: 
\begin{align*}
\sqrt{n}\nabla \bbM_n^{s, \gamma}(\theta_0) \implies \cN\left(0, 4V^{\gamma}\right) \,,\\
\sqrt{n\sigma_n}\nabla \bbM_n^{s, \psi}(\theta_0) \implies \cN\left(0, V^{\psi}\right) \,.
\end{align*} 
for some n.n.d. matrices $V^{\gamma}$ and $V^{\psi}$ which is mentioned explicitly in the proof. Further more $\sqrt{n}\nabla \bbM_n^{s, \gamma}(\theta_0)$ and $\sqrt{n\sigma_n}\nabla \bbM_n^{s, \psi}(\theta_0)$ are asymptotically independent. 
\end{lemma}
\noindent
Next, we analyze the convergence of $\nabla^2 \bbM_n^s(\theta^*)$ which is stated in the following lemma: 
\begin{lemma}[Convergence in Probability of $\nabla^s \bbM_n^s(\th^*)$]
\label{conv-prob}
Under Assumption \eqref{eq:assm}, for any random sequence $\breve{\theta} = \left(\breve{\beta}, \breve{\delta}, \breve{\psi}\right)$ such that $\breve{\beta} \overset{p}{\to} \beta_0, \breve{\delta} \overset{p}{\to} \delta_0, \|\breve{\psi} - \psi_0\|/\sigma_n \overset{P} \rightarrow 0$, we have: 
\begin{align*}
\nabla^2_{\gamma} \bbM_n^s(\breve{\theta}) & \overset{p}{\longrightarrow} 2Q^{\gamma} \,, \\
\sqrt{\sigma_n}\nabla^2_{\psi \gamma} \bbM_n^s(\breve{\theta}) & \overset{p}{\longrightarrow} 0 \,, \\
\sigma_n \nabla^2_{\psi} \bbM_n^s(\breve{\theta}) & \overset{p}{\longrightarrow} Q^{\psi} \,.
\end{align*}
for some matrices $Q^{\gamma}, Q^{\psi}$ mentioned explicitly in the proof. This, along with equation \eqref{eq:taylor_main}, establishes: 
\begin{align*}
\sqrt{n}\left(\hat \gamma^s - \gamma_0\right) & \overset{\mathscr{L}}{\implies} \cN\left(0, Q^{\gamma^{-1}}V^{\gamma}Q^{\gamma^{-1}}\right) \,, \\
\sqrt{n/\sigma_n}\left(\hat \psi^s - \psi_0\right) & \overset{\mathscr{L}}{\implies} \cN\left(0, Q^{\psi^{-1}}V^{\psi}Q^{\psi^{-1}}\right) \,.
\end{align*}
where as before $\hat \gamma^s = (\hat \beta^s, \hat \delta^s)$. 
\end{lemma}
It will be shown later that the condition $\|\breve{\psi}_n - \psi_0\|/\sigma_n \overset{P} \rightarrow 0$ needed in Lemma \ref{conv-prob} holds for the (random) sequence $\psi^*$, the intermediate point in the Taylor expansion. Then, combining Lemma \ref{asymp-normality} and Lemma \ref{conv-prob} we conclude the proof of Theorem \ref{thm:regression}. 
Observe that, to show $\left\|\psi^* - \psi_0 \right\| = o_P(\sigma_n)$, it suffices to to prove that $\left\|\hat \psi^s - \psi_0 \right\| = o_P(\sigma_n)$. Towards that direction, we have following lemma: 

\begin{lemma}[Rate of convergence]
\label{lem:rate_smooth}
Under Assumption \ref{eq:assm} and our choice of kernel and bandwidth, 
$$
n^{2/3}\sigma_n^{-1/3} d^2_*\left(\hat \theta^s, \theta_0^s\right) = O_P(1) \,,
$$
where 
\begin{align*}
d_*^2(\th, \th_0^s) & =  \|\beta - \beta_0^s\|^2 + \|\delta - \delta_0^s\|^2 \\
& \qquad \qquad + \frac{\|\psi - \psi_0^s\|^2}{\sigma_n} \mathds{1}_{\|\psi - \psi_0^s\| \le \cK\sigma_n} + \|\psi - \psi_0^s\| \mathds{1}_{\|\psi - \psi_0^s\| > \cK\sigma_n}  \,.
\end{align*}
for some specific constant $\cK$. (This constant will be mentioned precisely in the proof). Hence as $n\sigma_n \to \infty$, we have $n^{2/3}\sigma_n^{-1/3} \gg \sigma_n^{-1}$ which implies $\|\hat \psi^s - \psi_0^s\|/\sigma_n  \overset{P}  \longrightarrow 0 \,.$
\end{lemma}
%
%
%
\noindent
The above lemma establishes $\|\hat \psi^s - \psi_0^s\|/\sigma_n = o_p(1)$ but our goal is to show that $\|\hat \psi^s - \psi_0\|/\sigma_n = o_p(1)$. Therefore, we further need $\|\psi^s_0 - \psi_0\|/\sigma_n \rightarrow 0$ which is demonstrated in the following lemma:

\begin{lemma}[Convergence of population minimizer]
\label{bandwidth}
Under Assumption \ref{eq:assm} and our choice of kernel and bandwidth, we have: $\|\psi^s_0 - \psi_0\|/\sigma_n \rightarrow 0$. 
\end{lemma}

\noindent
Hence the final roadmap is the following: Using Lemma \ref{bandwidth} and Lemma \ref{lem:rate_smooth} we establish that $\|\hat \psi^s - \psi_0\|/\sigma_n = o_p(1)$ if $n\sigma_n \to 0$. This, in turn, enables us to prove Lemma \ref{conv-prob}, i.e. $\sigma_n \nabla^2 \bbM_n^s(\theta^*) \overset{P} \rightarrow Q$,which, along with Lemma \ref{asymp-normality}, establishes the main theorem.

%

%
%
%
\section{Binary response model}
\label{sec:classification_analysis}
Recall our binary response model in equation \eqref{eq:classification_eqn}. To estimate $\psi_0$, we resort to the following  loss (without smoothing): 
\begin{equation}
\label{eq:new_loss}
\bbM(\psi) = \bbE\left((Y - \gamma)\mathds{1}(Q^{\top}\psi \le 0)\right)\end{equation}
with $\gamma \in (\alpha_0, \beta_0)$, which can be viewed as a variant of the square error loss function: 
$$
\bbM(\alpha, \beta, \psi) = \bbE\left(\left(Y - \alpha\mathds{1}(Q^{\top}\psi < 0) - \beta\mathds{1}(Q^{\top}\psi > 0)\right)^2\right)\,.
$$
We establish the connection between these losses in sub-section \ref{loss_func_eq}. It is easy to prove that under fairly mild conditions (discussed later) 
$\psi_0 = \argmin_{\psi \in \Theta}\bbM(\psi)$, uniquely. Under the standard classification paradigm, when we know a priori that 
$\alpha_0 < 1/2 < \beta_0$, we can take $\gamma = 1/2$, and in the absence of this constraint, $\bar{Y}$, which converges to some $\gamma$ between $\alpha_0$ and $\beta_0$, may be substituted in the loss function. In the rest of the paper, we confine ourselves to a known $\gamma$, and for technical simplicity, we take $\gamma = \frac{(\beta_0 + \alpha_0)}{2}$, but this assumption can be removed with more mathematical book-keeping. Thus, $\psi_0$ is estimated by:  
\begin{equation}
\label{non-smooth-score} 
\hat \psi = \argmin_{\psi \in \Theta} \mathbb{M}_n(\psi) = \argmin_{\psi \in \Theta} \frac{1}{n}\sum_{\i=1}^n (Y_i - \gamma)\mathds{1}(Q_i^{\top}\psi \le 0)\,.
\end{equation} We resort to a smooth approximation of the indicator function in 
\eqref{non-smooth-score} using a distribution kernel with suitable bandwidth. The smoothed version of the population score function then becomes: 
\begin{equation}
\label{eq:kernel_smoothed_pop_score}
\bbM^s(\psi) = \bbE\left((Y - \gamma)\left(1-K\left(\frac{Q^{\top}\psi}{\sigma_n}\right)\right)\right)
\end{equation}
where as in the continuous response model, we use $K(x) = \Phi(x)$, and the corresponding empirical version is: 
\begin{equation}
\label{eq:kernel_smoothed_emp_score}
\bbM^s_n(\psi) = \frac{1}{n}\sum_{i=1}^n \left((Y_i - \gamma)\left(1-K\left(\frac{Q_i^{\top}\psi}{\sigma_n}\right)\right)\right)
\end{equation}
Define $\hat{\psi}^s$ and $\psi_0^s$ to be the minimizer of the smoothed version of the empirical (equation \eqref{eq:kernel_smoothed_emp_score}) and population score (equation \eqref{eq:kernel_smoothed_pop_score}) function respectively. Here we only consider the choice of bandwidth $n\sigma_n \to \infty$ and $n\sigma_n^2 \to 0$. Analogous to Theorem \ref{thm:regression} we prove the following result for binary response model: 
\begin{theorem}
\label{thm:binary}
Under Assumptions (\ref{as:distribution} - \ref{as:eigenval_bound}): 
$$
\sqrt{\frac{n}{\sigma_n}}\left(\hat{\psi}_n - \psi_0\right) \Rightarrow N(0, \Gamma) \,,
$$ 
for some non-stochastic matrix $\Gamma$, which will be defined explicitly in the proof. 
\end{theorem}
We have therefore established that in the regime $n\sigma_n \to \infty$ and $n\sigma_n^2 \to 0$, it is possible to attain asymptotic normality using a smoothed estimator for binary response model.

\section{Inferential methods}
\label{sec:inference}
We draw inferences on $(\beta_0, \delta_0, \psi_0)$ by resorting to similar techniques as in \cite{seo2007smoothed}. For the continuous response model, we need consistent estimators of $V^{\gamma}, Q^{\gamma}, V^{\psi}, Q^{\psi}$ (see Lemma \ref{conv-prob} for the definitions) for hypothesis testing. By virtue of the aforementioned Lemma, we can estimate $Q^{\gamma}$ and $Q^{\psi}$ as follows: 
\begin{align*}
\hat Q^{\gamma} & = \nabla^2_{\gamma} \bbM_n^s(\hat \theta) \,, \\ 
\hat Q^{\psi} & = \sigma_n \nabla^2_{\psi} \bbM_n^s(\hat \theta) \,.
\end{align*}
The consistency of the above estimators is established in the proof of Lemma \ref{conv-prob}. For the other two parameters $V^{\gamma}, V^{\psi}$ we use the following estimators: 
\begin{align*}
\hat V^{\psi} & = \frac{1}{n\sigma_n^2}\sum_{i=1}^n\left(\left(Y_i - X_i^{\top}(\hat \beta + \hat \delta)\right)^2 - \left(Y_i- X_i^{\top}\hat \beta\right)^2\right)^2\tilde Q_i \tilde Q_i^{\top}\left(K'\left(\frac{Q_i^{\top}\hat \psi}{\sigma_n}\right)\right)^2 \\
\hat V^{\gamma} & = \hat \sigma^2_\eps \begin{pmatrix} \frac{1}{n}X_iX_i^{\top} & \frac{1}{n}X_iX_i^{\top}\mathds{1}_{Q_i^{\top}\hat \psi > 0} \\ \frac{1}{n}X_iX_i^{\top}\mathds{1}_{Q_i^{\top}\hat \psi > 0} & \frac{1}{n}X_iX_i^{\top}\mathds{1}_{Q_i^{\top}\hat \psi > 0} \end{pmatrix}
\end{align*}
where $\hat \sigma^2_\eps$ can be obtained as $(1/n)(Y_i - X_i^{\top}\hat \beta - X_i^{\top}\hat \delta \mathds{1}(Q_i^{\top}\hat \psi > 0))^2$, i.e. the residual sum of squares. The explicit value of $V_\gamma$ (as derived in equation \eqref{eq:def_v_gamma} in the proof Lemma \ref{asymp-normality}) is: 
$$
V^{\gamma} = \sigma_\eps^2 \begin{pmatrix}\bbE\left[XX^{\top}\right] & \bbE\left[XX^{\top}\mathds{1}_{Q^{\top}\psi_0 > 0}\right] \\
\bbE\left[XX^{\top}\mathds{1}_{Q^{\top}\psi_0 > 0}\right] & \bbE\left[XX^{\top}\mathds{1}_{Q^{\top}\psi_0 > 0}\right] \end{pmatrix} 
$$ 
Therefore, the consistency of $\hat V_\gamma$ is immediate from the law of large numbers. The consistency of $\hat V^{\psi}$ follows via arguments similar to those employed in proving Lemma \ref{conv-prob} but under somewhat more stringent moment conditions: in particular, we need $\bbE[\|X\|^8] < \infty$ and $\bbE[(X^{\top}\delta_0)^k \mid Q]$ to be Lipschitz functions over $Q$ for $1 \le k \le 8$. The inferential techniques for the classification model are similar and hence skipped, to avoid repetition.

\section{Simulation studies}
\label{sec:simulation}
In this section, we present some simulation results to analyse the effect of the choice of $\sigma_n$ on the finite sample approximation of asymptotic normality, i.e. Berry-Essen type bounds. If we choose a smaller sigma, the rate of convergence is accelerated but the normal approximation error at smaller sample sizes will be higher, as we don't have enough observations in the vicinity of the change hyperplane for the CLT to kick in. This problem is alleviated by choosing $\sigma_n$ larger, but this, on the other hand, compromises the convergence rate. Ideally, a Berry-Essen type of bound will quantify this, but this will require a different set of techniques and is left as an open problem. In our simulations, we generate data from following setup: 
\begin{enumerate}
\item Set $N = 50000, p = 3, \alpha_0 = 0.25, \beta = 0.75$ and some $\theta_0 \in \bbR^p$ with first co-ordinate $ = 1$. 
\item Generate $X_1, \dots, X_n \sim \cN(0, I_p)$. 
\item Generate $Y_i \sim \textbf{Bernoulli}\left(\alpha_0\mathds{1}_{X_i^{\top}\theta_0 \le 0} + \beta_0 \mathds{1}_{X_i^{\top}\theta_0 > 0}\right)$. 
\item Estimate $\hat \theta$ by minimizing $\bbM_n(\theta)$ (replacing $\gamma$ by $\bar Y$) based on $\{(X_i, Y_i)\}_{i=1}^n$ for different choices of $\sigma_n$. 
\end{enumerate}
We repeat Step 2 - Step 4 a hundred times to obtain $\hat \theta_1, \dots, \hat \theta_{100}$. Define $s_n$ to be the standard deviation of $\{\hat \theta_i\}_{i=1}^{100}$. Figures ref{fig:co2} and \ref{fig:co3} show the qqplots of $\tilde \theta_i = (\hat \theta_i - \theta_0)/s_n$ against the standard normal for four different choices of $\sigma_n = n^{-0.6}, n^{-0.7}, n^{-0.8}, n^{-0.9}$. 
\begin{figure}
\centering 
\includegraphics[scale=0.4]{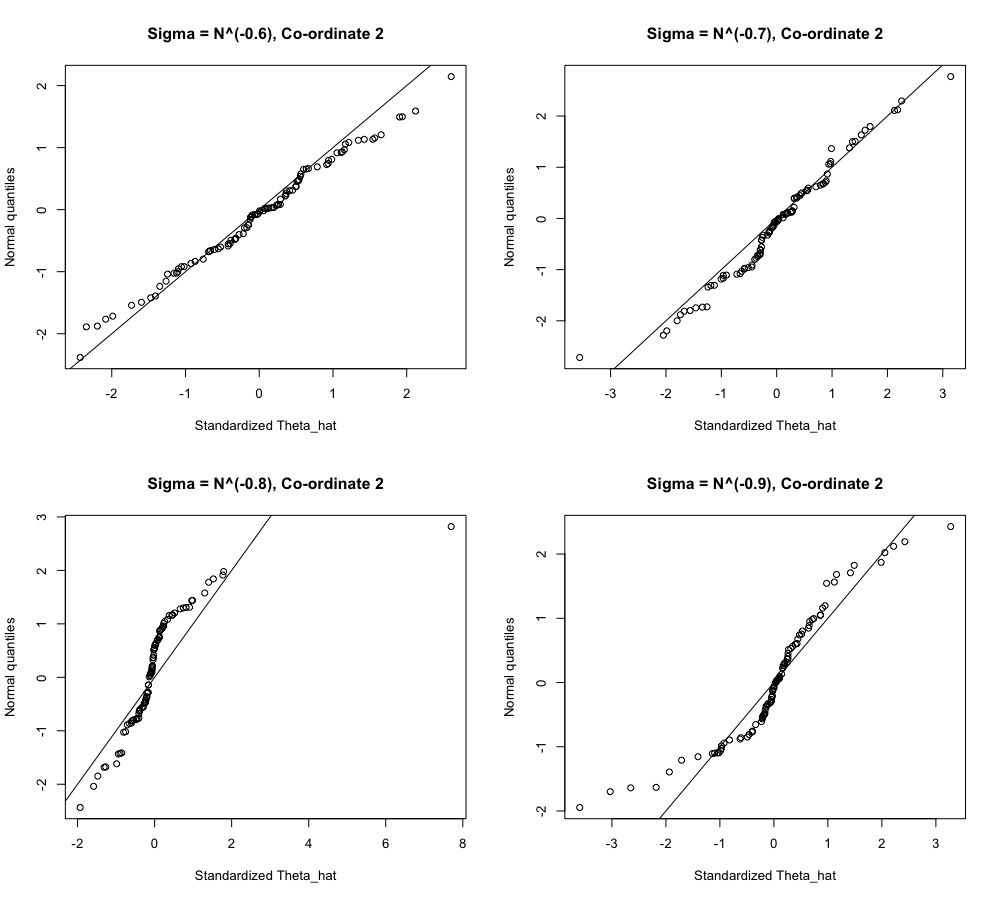}
\caption{In this figure, we present qqplot for estimating second co-ordinate of $\theta_0$ with different choices of $\sigma_n$ mentioned at the top of each plots.}
\label{fig:co2}
\end{figure}
\begin{figure}
\centering 
\includegraphics[scale=0.4]{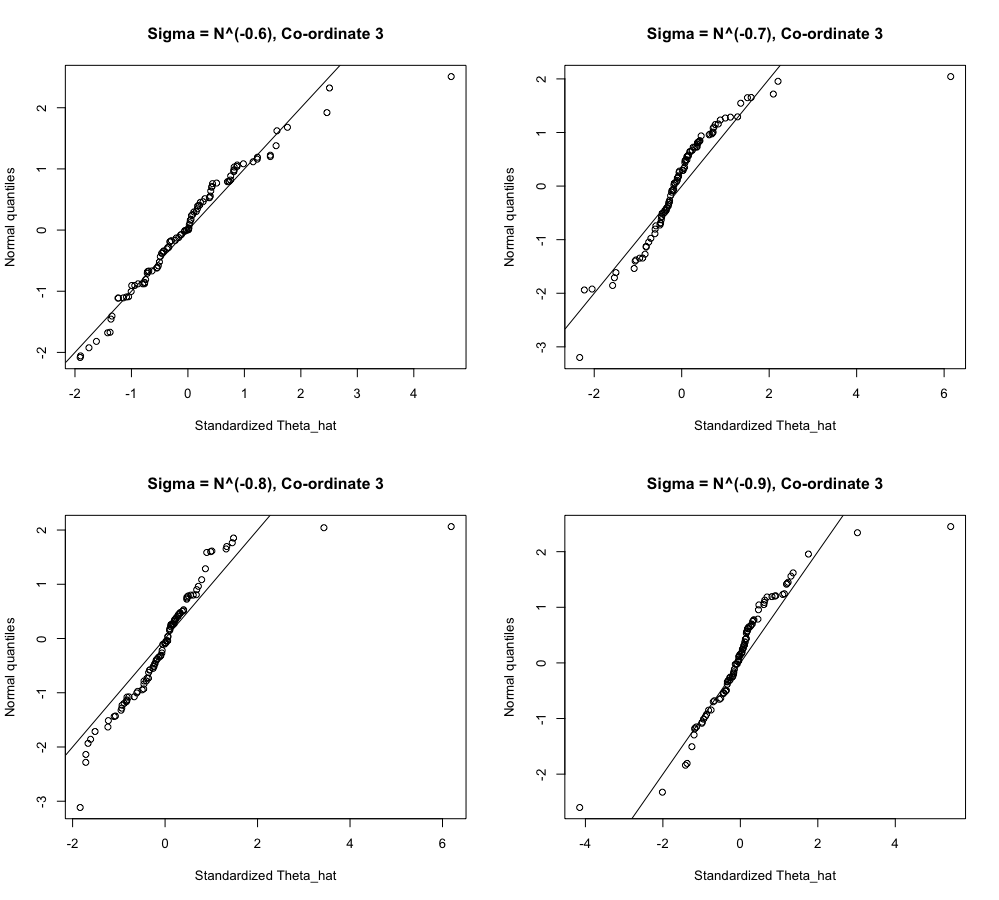}
\caption{In this figure, we present qqplot for estimating third co-ordinate of $\theta_0$ with different choices of $\sigma_n$ mentioned at the top of each plots.}
\label{fig:co3}
\end{figure}
It is evident that smaller value of $\sigma_n$ yield a poor normal approximation. Although our theory shows that asymptotic normality holds as long as $n\sigma_n \to \infty$, in practice we recommend choosing $\sigma_n$ such that $n\sigma_n \ge 30$ for the central limit of theorem to take effect.

\section{Real data analysis}
\label{sec:real_data}
We illustrate our method using cross-country data on pollution (carbon-dioxide), income and urbanization obtained from the World Development Indicators (WDI), World Bank. The Environmental Kuznets Curve hypothesis (EKC henceforth), a popular and ongoing area of research in environmental economics, posits that at an initial stage of economic development pollution increases with economic growth, and then diminishes when society’s priorities change, leading to an inverted U-shaped relation between income (measured via real GDP per capita) and pollution. The hypothesis has led to numerous empirical papers (i) testing the hypothesis (whether the relation is inverted U-shaped for countries/regions of interest in the sample), (ii) exploring the threshold level of income at which pollution starts falling, as well as (iii) examining the countries/regions which belong to the upward rising part versus the downward sloping part of the inverted U-shape, if at all. The studies have been performed using US state level data or cross-country data (e.g. \cite{shafik1992economic}, \cite{millimet2003environmental}, \cite{aldy2005environmental}, \cite{lee2019nonparametric},\cite{boubellouta2021cross}, \cite{list1999environmental}, \cite{grossman1995economic}, \cite{bertinelli2005environmental}, \cite{azomahou2006economic}, \cite{taskin2000searching} to name a few). While some of these papers have found evidence in favor of the EKC hypothesis (inverted U-shaped income-pollution relation), others have found evidence against it (monotonically increasing or other shapes for the relation). The results often depend on countries/regions in the sample, period of analysis, as well as the pollutant studied.
\\\\
\noindent
While income-pollution remains the focal point of most EKC studies, several of them have also included urban agglomeration (UA) or some other measures of urbanization as an important control variable especially while investigating carbon emissions.\footnote {Although income growth is connected to urbanization, countries are heterogenous and follow different growth paths due to their varying geographical structures, population densities, infrastructures, ownerships of resources making a case for using urbanization as another control covariate in the income-pollution study. The income growth paths of oil rich UAE, manufacturing based China, serviced based Singapore, low population density Canada (with vast land) are all different.} (see for example, \cite{shafik1992economic}, \cite{boubellouta2021cross}and \cite{liang2019urbanization}). The theory of ecological economics posits potentially varying effects of increased urbanization on pollution– (i) urbanization leading to more pollution (due to its close links with sanitations, dense transportations, and proximities to polluting manufacturing industries), (ii) urbanization potentially leading to less pollution based on ‘compact city theory’ (see \cite{burton2000compact}, \cite{capello2000beyond}, \cite{sadorsky2014effect}) that explains the potential benefits of increased urbanization in terms of economies of scale (for example, replacing dependence on automobiles with large scale subway systems, using multi-storied buildings instead of single unit houses, keeping more open green space). \cite{liddle2010age}, using 17 developed countries, find a positive and significant effect of urbanization on pollution. On the contrary, using a set of 69 countries \cite{sharma2011determinants} find a negative and significant effect of urbanization on pollution while \cite{du2012economic} find an insignificant effect of urbanization on carbon emission. Using various empirical strategies \cite{sadorsky2014effect} conclude that the positive and negative effects of urbanization on carbon pollution may cancel out depending on the countries involved often leaving insignificant effects on pollution. They also note that many countries are yet to achieve a sizeable level of urbanization which presumably explains why many empirical works using less developed countries find insignificant effect of urbanization.  In summary, based on the existing literature, both the relationship between urbanization and pollution as well as the relationship between income and pollution appear to depend largely on the set of countries considered in the sample. This motivates us to use UA along with income in our change plane model for analyzing carbon-dioxide emission to plausibly separate the countries into two regimes. 
\\\\
\noindent
Following the broad literature we use pollution emission per capita (carbon-dioxide measured in metric tons per capita) as the dependent variable and real GDP per capita (measured in 2010 US dollars), its square (as is done commonly in the EKC literature) and a popular measure of urbanization, namely urban agglomeration  (UA)\footnote{The exact definition can be found in the World Development Indicators database from the World Bank website.} as covariates (in our notation $X$) in our regression. In light of the preceding discussions we fit a change plane model comprising real GDP per capita and UA (in our notation $Q$). To summarize the setup, we use the continuous response model as described in equation \eqref{eq:regression_main_eqn}, i.e 
\begin{align*}
Y_i & = X_i^{\top}\beta_0 + X_i^{\top}\delta_0\mathds{1}_{Q_i^{\top}\psi_0 > 0} + \eps_i \\
& = X_i^{\top}\beta_0\mathds{1}_{Q_i^{\top}\psi_0 \le 0} + X_i^{\top}(\beta_0 + \delta_0)\mathds{1}_{Q_i^{\top}\psi_0 > 0} + \eps_i
\end{align*}
with the per capita $CO_2$ emission in metric ton as $Y$, per capita GDP, square of per capita GDP and UA as $X$ (hence $X \in \bbR^3$) and finally, per capita GDP and UA as $Q$ (hence $Q \in \bbR^2$). Observe that $\beta_0$ represents the regression coefficients corresponding to the countries with $Q_i^{\top}\psi_0 \le 0$ (henceforth denoted by Group 1) and $(\beta_0+ \delta_0)$ represents the regression coefficients corresponding to the countries with $Q_i^{\top}\psi_0 \ge 0$ (henceforth denoted by Group 2). As per our convention, in the interests of identifiability we assume $\psi_{0, 1} = 1$, where $\psi_{0,1}$ is the change plane parameter corresponding to per capita GDP. Therefore the only change plane coefficient to be estimated is $\psi_{0, 2}$, the change plane coefficient for UA. For numerical stability, we divide per capita GDP by $10^{-4}$ (consequently square of per capital GDP is scaled by $10^{-8}$)\footnote{This scaling helps in the numerical stability of the gradient descent algorithm used to optimize the least squares criterion.}.  After some pre-processing (i.e. removing rows consisting of NA and countries with $100\%$ UA) we estimate the coefficients $(\beta_0, \delta_0, \psi_0)$ of our model based on data from 115 countries with $\sigma_n = 0.05$ and test the significance of the various coefficients using the methodologies described in Section \ref{sec:inference}. We present our findings in Table \ref{tab:ekc_coeff}.  
\begin{table}[!h]
    \centering
    \begin{tabular}{|c||c||c|}
    \hline
    Coefficients & Estimated values & p-values \\
    \hline \hline 
      $\beta_{0, 1}$ (\text{RGDPPC for Group 1})  & 6.98555060 & 4.961452e-10   \\
       $\beta_{0, 2}$ (\text{squared RGDPPC for Group 1})  &  -0.43425991 & 7.136484e-02 \\
       $\beta_{0, 3}$ (\text{UA for Group 1}) & -0.02613813  & 1.066065e-01
\\
       $\beta_{0, 1} + \delta_{0, 1}$ (\text{RGDPPC for Group 2}) & 2.0563337 & 0.000000e+00\\
      $\beta_{0, 2} + \delta_{0, 2}$ (\text{squared RGDPPC for Group 2}) &  -0.1866490 & 4.912843e-04   \\
       $\beta_{0, 3} + \delta_{0, 3}$ (\text{UA for Group 2}) & 0.1403171& 1.329788e-05 \\
       $\psi_{0,2 }$  (\text{Change plane coeff for UA}) & -0.07061785 & 0.000000e+00\\
         \hline
    \end{tabular}
    \caption{Table of the estimated regression and change plane coefficients along with their p-values.}
    \label{tab:ekc_coeff}
\end{table}
\\\\
\noindent
From the above analysis, we find that GDP has significantly positive effect on pollution for both groups of countries. The effect of its squared term is negative for both groups; but the effect is significant for Group-2 consisting of mostly high income countries whereas its effect is insignificant (at the 5\% level) for the Group-1 countries (consisting of mostly low or middle income and few high income countries). Thus, not surprisingly, we find evidence in favor of EKC for the developed countries, but not for the mixed group.  Notably, Group-1 consists of a mixed set of countries like Angola, Sudan, Senegal, India, China, Israel, UAE etc., whereas Group-1 consists of rich and developed countries like Canada, USA, UK, France, Germany etc. The urban variable, on the other hand, is seen to have insignificant effect on Group-1 which is in keeping with \cite{du2012economic}, \cite{sadorsky2014effect}. Many of them are yet to achieve substantial urbanization and this is more true for our sample period \footnote{We use 6 years average from 2010-2015 for GDP and pollution measures. Such averaging is in accordance with the cross-sectional empirical literature using cross-country/regional data and helps avoid business cycle fluctuations in GDP. It also minimizes the impacts of outlier events such as the financial crisis or great recession period.  The years that we have chosen are ones for which we could find data for the largest number of countries.}. In contrast, UA has a positive and significant effect on Group-2 (developed) countries which is consistent with the findings of \cite{liddle2010age}, for example. Note that UA plays a crucial role in dividing the countries into different regimes, as the estimated value of $\psi_{0,2}$ is significant. Thus, we are able to partition countries into two regimes: a mostly rich and a mixed group. 
\\\\
\noindent
Note that many underdeveloped countries and poorer regions of emerging countries are still swamped with greenhouse gas emissions from burning coal, cow dung etc., and usage of poor exhaust systems in houses and for transport. This is more true for rural and semi-urban areas of developing countries. So even while being less urbanized compared to developed nations, their overall pollution load is high (due to inefficient energy usage and higher dependence on fossil fuels as pointed out above) and rising with income and they are yet to reach the descending part of the inverted U-shape for the income-pollution relation. On the contrary, for countries in Group-2, the adoption of more efficient energy and exhaust systems are common in households and transportations in general, leading to eventually decreasing pollution with increasing income (supporting EKC). Both the results are in line with the existing EKC literature. Additionally we find that the countries in Group 2 are yet to achieve ‘compact city’ and green urbanization. This is a stylized fact that is confirmed by the positive and significant effect of UA on pollution in our analysis.  
\\\\
\noindent
There are many future potential applications of our method in economics. Similar analyses can be performed for other pollutants (such as sulfur emission, electrical waste/e-waste, nitrogen pollution etc.). While income/GDP remains a common, indeed the most crucial variable in pollution studies, other covariates (including change plane defining variables) may vary, depending on the pollutant of interest. Another potential application can be that of identifying the determinants of family health expenses in household survey data. Families are often asked about their health expenses incurred in the past one year. An interesting case in point may be household surveys collected in India where one finds numerous (large) joint families with several children and old people residing in the same household and most families are uninsured. It is often seen that health expenditure increases with income with a major factor being the costs associated with regularly performed preventative medical examinations which are affordable only once a certain income level is reached. The important covariates here are per capita family income, family wealth, `dependency ratio' (number of children and old to the total number of people in the family) and the binary indicator of any history of major illness/hospitalizations in the family in the past year. Family income per capita and history of major illness are natural candidate covariates for defining the change plane. 

\section{Conclusion}
\label{sec:conclusion}
In this paper we have established that under some mild assumptions the kernel-smoothed change plane estimator is asymptotically normal with near optimal rate $n^{-1}$. To the best of our knowledge, the state of the art result in this genre of problems is due to \cite{seo2007smoothed}, where they demonstrate a best possible rate about $n^{-3/4}$ for i.i.d. data. The main difference between their approach and ours is mainly the proof of Lemma \ref{bandwidth}. Our techniques are based upon modern empirical process theory which allow us to consider much smaller bandwidths $\sigma_n$ compared to those in \cite{seo2007smoothed}, who appear to require larger values to achieve the result, possibly owing to their reliance on the techniques developed in \cite{horowitz1992smoothed}. Although we have established it is possible to have asymptotic normality with really small bandwidths, we believe that the finite sample approximation (e.g. Berry-Essen bound) to normality could be poor, which is also evident from our simulation. 


\appendix


\section{Appendix}
In this section, we present the proof of Lemma \ref{lem:rate_smooth}, which lies at the heart of our refined analysis of the smoothed change plane estimator. Proofs of the other lemmas and our results for the binary response model are available in the Appendix \ref{sec:supp_B}.  
\subsection{Proof of Lemma \ref{lem:rate_smooth}}

\begin{proof}
The proof of Lemma \ref{lem:rate_smooth} is quite long, hence we further break it into few more lemmas. 
\begin{lemma}
\label{lem:pop_curv_nonsmooth}
Under Assumption \eqref{eq:assm}, there exists $u_- , u_+ > 0$ such that: 
$$
u_- d^2(\theta, \theta_0) \le \bbM(\theta) - \bbM(\theta_0) \le u_+ d^2(\theta, \theta_0) \,,
$$
for $\th$ in a (non-srinking) neighborhood of $\th_0$, where: 
$$
d(\theta, \theta_0) := \sqrt{\|\beta - \beta_0\|^2 + \|\delta - \delta_0\|^2 + \|\psi - \psi_0\|} \,.
$$
\end{lemma}

\begin{lemma}
\label{lem:uniform_smooth}
Under Assumption \ref{eq:assm} the smoothed loss function $\bbM^s(\theta)$ is uniformly close to the non-smoothed loss function $\bbM(\th)$: 
$$
\sup_{\th \in \Theta}\left|\bbM^s(\theta) - \bbM(\theta)\right| \le K_1 \sigma_n \,,
$$ 
for some constant $K_1$. 
\end{lemma}


\begin{lemma}
\label{lem:pop_smooth_curvarture}
Under certain assumptions: 
\begin{align*}
\bbM^s(\th) - \bbM^s(\th_0^s) & \gtrsim \|\beta - \beta_0^s\|^2 + \|\delta - \delta_0^s\|^2 \\
& \qquad \qquad + \frac{\|\psi - \psi_0^s\|^2}{\sigma_n} \mathds{1}_{\|\psi - \psi_0^s\| \le \cK\sigma_n} + \|\psi - \psi_0^s\| \mathds{1}_{\|\psi - \psi_0^s\| > \cK\sigma_n} \\\
& := d_*^2(\th, \th_0^s) \,.
\end{align*}
for some constant $\cK$ and for all $\th$ in a neighborhood of $\th_0$, which does not change with $n$. 
\end{lemma}

The proofs of the three lemmas above can be found in Appendix \ref{sec:supp_B}. We next move to the proof of Lemma \ref{lem:rate_smooth}. In Lemma \ref{lem:pop_smooth_curvarture} we have established the curvature of the smooth loss function $\bbM^s(\th)$ around $\th_0^s$. To determine the rate of convergence of $\hat \th^s$ to $\th_0^s$, we further need an upper bound on the modulus of continuity of our loss function. Towards that end, first recall that our loss function is:
$$
f_{\th}(Y, X, Q) =  \left(Y - X^{\top}\beta\right)^2 +  \left[-2\left(Y - X^{\top}\beta\right)X^{\top}\delta + (X^{\top}\delta)^2\right] K\left(\frac{Q^{\top}\psi}{\sigma_n}\right)
$$
The centered loss function can be written as: 
\begin{align}
    & f_{\th}(Y, X, Q) - f_{\th_0^s}(Y, X, Q) \notag \\
    & = \left(Y - X^{\top}\beta\right)^2 +  \left[-2\left(Y - X^{\top}\beta\right)X^{\top}\delta + (X^{\top}\delta)^2\right] K\left(\frac{Q^{\top}\psi}{\sigma_n}\right) \notag \\
    & \qquad \qquad \qquad \qquad - \left(Y - X^{\top}\beta_0^s\right)^2 -  \left[-2\left(Y - X^{\top}\beta_0^s\right)X^{\top}\delta_0^s + (X^{\top}\delta_0^s)^2\right] K\left(\frac{Q^{\top}\psi_0^s}{\sigma_n}\right) \notag \\
    & = \left(Y - X^{\top}\beta\right)^2 +  \left[-2\left(Y - X^{\top}\beta\right)X^{\top}\delta + (X^{\top}\delta)^2\right] K\left(\frac{Q^{\top}\psi}{\sigma_n}\right) \notag \\
    & \qquad \qquad \qquad \qquad - \left(Y - X^{\top}\beta_0^s\right)^2 -  \left[-2\left(Y - X^{\top}\beta_0^s\right)X^{\top}\delta_0^s + (X^{\top}\delta_0^s)^2\right] K\left(\frac{Q^{\top}\psi}{\sigma_n}\right) \notag \\
    & \qquad \qquad \qquad \qquad \qquad \qquad - \left[-2\left(Y - X^{\top}\beta_0^s\right)X^{\top}\delta_0^s + (X^{\top}\delta_0^s)^2\right] \left\{K\left(\frac{Q^{\top}\psi_0^s}{\sigma_n}\right) -  K\left(\frac{Q^{\top}\psi}{\sigma_n}\right)\right\} \notag \\
    & = \underbrace{\left(Y - X^{\top}\beta\right)^2 - \left(Y - X^{\top}\beta_0^s\right)^2}_{M_1}  \notag \\
    & \qquad + \underbrace{\left\{ \left[-2\left(Y - X^{\top}\beta\right)X^{\top}\delta + (X^{\top}\delta)^2\right]  -  \left[-2\left(Y - X^{\top}\beta_0^s\right)X^{\top}\delta_0^s + (X^{\top}\delta_0^s)^2\right]\right\} K\left(\frac{Q^{\top}\psi}{\sigma_n}\right)}_{M_2} \notag \\
    & \qquad \qquad \qquad \qquad \qquad \qquad - \underbrace{\left[-2\left(Y - X^{\top}\beta_0^s\right)X^{\top}\delta_0^s + (X^{\top}\delta_0^s)^2\right] \left\{K\left(\frac{Q^{\top}\psi_0^s}{\sigma_n}\right) -  K\left(\frac{Q^{\top}\psi}{\sigma_n}\right)\right\}}_{M_3} \notag \\
    \label{eq:expand_f} & := M_1 + M_2 + M_3
\end{align}
For the rest of the analysis, fix $\zeta > 0$ and consider the collection of functions $\cF_{\zeta}$ which is defined as: 
$$
\cF_{\zeta} = \left\{f_\th - f_{\th^s}: d_*(\th, \th^s) \le \zeta\right\} \,.
$$
First note that $\cF_\zeta$ has bounded uniform entropy integral (henceforth BUEI) over $\zeta$. To establish this, it is enough to argue that the collection $\cF = \{ f_\th : \th \in \Theta\}$ is BUEI. Note that the functions $X \mapsto X^{\top}\beta$ has VC dimension $p$ and so is the map $X \mapsto X^{\top}(\beta + \delta)$. Therefore the functions $(X, Y) \mapsto (Y - X^{\top}(\beta + \delta))^2 - (Y - X^{\top}\beta)^2$ is also BUEI, as composition with monotone function (here $x^2$) and taking difference keeps this property. Further by the hyperplane $Q \mapsto Q^{\top}\psi$ also has finite dimension (only depends on the dimension of $Q$) and the VC dimension does not change by scaling it with $\sigma_n$. Therefore the functions $Q \mapsto Q^{\top}\psi/sigma_n$ has same VC dimension as $Q \mapsto Q^{\top}\psi$ which is independent of $n$. Again, as composition of monotone function keeps BUEI property, the functions $Q \mapsto K(Q^{\top}\psi/\sigma_n)$ is also BUEI. As the product of two BUEI class is BUEI, we conclude that $\cF$ (and hence $\cF_\zeta$) is BUEI. 
\\\\
\noindent
Now to bound the modulus of continuity we use Lemma 2.14.1 of \cite{vdvw96}: 
\begin{equation*}
\label{eq:moc_bound}
\sqrt{n}\bbE\left[\sup_{\theta: d_*(\theta, \theta_0^s) \le \zeta} \left|\left(\bbP_n - P\right)\left(f_\theta - f_{\theta_0^s}\right)\right|\right] \lesssim \cJ(1, \cF_\zeta) \sqrt{\bbE\left[F_{\zeta}^2(X, Y, Q)\right]}
\end{equation*}
where $F_\zeta$ is some envelope function of $\cF_\zeta$. As the function class $\cF_\zeta$ has bounded entropy integral, $\cJ(1, \cF_\zeta) $ can be bounded above by some constant independent of $n$. We next calculate the order of the envelope function $F_\zeta$. Recall that, by definition of envelope function is: 
$$
F_{\zeta}(X, Y, Q) \ge \sup_{\theta: d_*(\theta, \theta_0^s) \le \zeta} \left| f_{\theta} - f_{\theta_0}\right| \,.
$$
and we can write $f_\th - f_{\th_0^s} = M_1 + M_2 + M_3$ which follows from equation \eqref{eq:expand_f}. Therefore, to find the order of the envelope function, it is enough to find the order of bounds of $M_1, M_2, M_3$ over the set $d_*(\th, \th_0^s) \le \zeta$. We start with $M_1$: 
\begin{align}
    \sup_{d_*(\th, \th_0^s) \le \zeta}|M_1| & = \sup_{d_*(\th, \th_0^s) \le \delta}\left|\left(Y - X^{\top}\beta\right)^2 - \left(Y - X^{\top}\beta_0^s\right)^2\right| \notag \\
    & = \sup_{d_*(\th, \th_0^s) \le \zeta} \left|2YX^{\top}(\beta_0^s - \beta) + (X^{\top}\beta)^2 - (X^{\top}\beta_0^S)^2\right| \notag \\
    & \le \sup_{d_*(\th, \th_0^s) \le \zeta} \|\beta - \beta_0^s\| \left[2|Y|\|X\| + (\|\beta_0^s\| + \zeta)\|X\|^2\right] \notag \\
    \label{eq:env_1} & \le  \zeta\left[2|Y|\|X\| + (\|\beta_0^s\| + \zeta)\|X\|^2\right] := F_{1, \zeta}(X, Y, Q) \hspace{0.1in} [\text{Envelope function of }M_1]
\end{align}
and the second term: 
\allowdisplaybreaks
\begin{align}
    & \sup_{d_*(\th, \th_0^s) \le \zeta} |M_2| \notag \\
    & = \sup_{d_*(\th, \th_0^s) \le \zeta} \left|\left\{\left[-2\left(Y - X^{\top}\beta\right)X^{\top}\delta + (X^{\top}\delta)^2\right] \right. \right. \notag \\
    & \qquad \qquad \qquad \qquad \left. \left. -  \left[-2\left(Y - X^{\top}\beta_0^s\right)X^{\top}\delta_0^s + (X^{\top}\delta_0^s)^2\right]\right\}\right|K\left(\frac{Q^{\top}\psi}{\sigma_n}\right) \notag \\
    & \le \sup_{d_*(\th, \th_0^s) \le \zeta} \left|\left\{\left[-2\left(Y - X^{\top}\beta\right)X^{\top}\delta + (X^{\top}\delta)^2\right] \right. \right. \notag \\
    & \qquad \qquad \qquad \qquad \left. \left. -  \left[-2\left(Y - X^{\top}\beta_0^s\right)X^{\top}\delta_0^s + (X^{\top}\delta_0^s)^2\right]\right\}\right| \notag \\
    & = \sup_{d_*(\th, \th_0^s) \le \zeta} \left|\left\{\left[2Y(X^{\top}\delta_0^s - X^{\top}\delta) + 2[(X^{\top}\beta)(X^{\top}\delta) \right. \right. \right.  \notag \\
    & \qquad \qquad \qquad \qquad \left. \left. \left.  - (X^{\top}\beta_0^s)(X^{\top}\delta_0^s)] + (X^{\top}\delta)^2 - (X^{\top}\delta_0^s)^2\right]\right\}\right| \notag \\
    & \le \sup_{d_*(\th, \th_0^s) \le \zeta} \left\{\|\delta - \delta_0^s\|2|Y|\|X\| + 2\|\beta - \beta_0\|\|X\|\|\delta\| \right. \notag \\
    & \qquad \qquad \qquad \qquad \left. + 2\|\delta - \delta_0^s\|\|X\|\|\beta_0^s\| + 2\|X\|\|\delta + \delta_0^s\|\|\delta - \delta_0^s\|\right\} \notag \\ \notag \\
    & \le \zeta \left[2|Y|\|X\| + 2\|X\|(\|\delta_0^s\| + \|\zeta\|) + 2\|X\|\|\beta_0^s\| + 2\|X\|(\|\delta_0^s\| + \zeta)\right] \notag \\ 
    \label{eq:env_2}& = \zeta \times 2\|X\|\left[2|Y| + 2(\|\delta_0^s\| + \|\zeta\|) + \|\beta_0^s\|\right] := F_{2, \zeta}(X, Y, Q) \hspace{0.1in} [\text{Envelope function of }M_2]
    \end{align}
For the third term, note that: 
\begin{align*}
& \sup_{d_*(\th, \th_0^s) \le \zeta} |M_3| \\
& \le \left|\left[-2\left(Y - X^{\top}\beta_0^s\right)X^{\top}\delta_0^s + (X^{\top}\delta_0^s)^2\right]\right| \times  \sup_{d_*(\th, \th_0^s) \le \zeta} \left|\left\{K\left(\frac{Q^{\top}\psi_0^s}{\sigma_n}\right) -  K\left(\frac{Q^{\top}\psi}{\sigma_n}\right)\right\}\right| \\
& := F_{3, \zeta} (X, Y, Q)
\end{align*}
Henceforth, we define the envelope function to be $F_\zeta = F_{\zeta, 1} + F_{\zeta, 2} + F_{\zeta, 3}$. Hence we have by triangle inequality: 
$$
\sqrt{\bbE\left[F_{\zeta}^2(X, Y, Q)\right]} \le \sum_{i=1}^3 \sqrt{\bbE\left[F_{i, \zeta}^2(X, Y, Q)\right]}
$$
From equation \eqref{eq:env_1} and \eqref{eq:env_2} we have: 
\begin{equation}
\label{eq:moc_bound_2}
\sqrt{\bbE\left[F_{1, \zeta}^2(X, Y, Q)\right]} + \sqrt{\bbE\left[F_{2, \zeta}^2(X, Y, Q)\right]} \lesssim \zeta \,.
\end{equation}
For $F_{3, \zeta}$, first note that: 
\begin{align*}
    & \bbE\left[\left|\left[-2\left(Y - X^{\top}\beta_0^s\right)X^{\top}\delta_0^s + (X^{\top}\delta_0^s)^2\right]\right|^2 \mid Q\right] \\
    & \le 8\bbE\left[\left(Y - X^{\top}\beta_0^s\right)^2(X^{\top}\delta_0)^2 \mid Q\right] + 2\bbE[(X^{\top}\delta_0^s)^4 \mid Q] \\
    & \le \left\{8\|\beta - \beta_0^s\|^2\|\delta_0\|^2 + 8\|\delta_0\|^4 + 2\|\delta_0^s\|^4\right\}m_4(Q) \,.
\end{align*}
where $m_4(Q)$ is defined in Assumption \ref{eq:assm}. In this part, we have to tackle the dichotomous behavior of $\psi$ around $\psi_0^s$ carefully. Henceforth define $d_*^2(\psi, \psi_0^s)$ as: 
\begin{align*}
    d_*^2(\psi, \psi_0^s) =  & \frac{\|\psi - \psi_0^s\|^2}{\sigma_n}\mathds{1}_{\|\psi - \psi_0^s\| \le \cK\sigma_n} + \|\psi - \psi_0^s\|\mathds{1}_{\|\psi - \psi_0^s\| > \cK\sigma_n} 
\end{align*}
This is a slight abuse of notation, but the reader should think of it as the part of $\psi$ in $d_*^2(\th, \th_0^s)$. Define $B_{\zeta}(\psi_0^s)$ to be set of all $\psi$'s such that $d^2_*(\psi, \psi_0^s) \le \zeta^2$. We can decompose $B_{\zeta}(\psi_0^s)$ as a disjoint union of two sets: 
\begin{align*}
    B_{\zeta, 1}(\psi_0^s) & = \left\{\psi: d^2_*(\psi, \psi_0^s) \le \zeta^2, \|\psi - \psi_0^s\| \le \cK\sigma_n\right\} \\
    & = \left\{\psi:\frac{\|\psi - \psi_0^s\|^2}{\sigma_n} \le \zeta^2, \|\psi - \psi_0^s\| \le \cK\sigma_n\right\}  \\
    & = \left\{\psi:\|\psi - \psi_0^s\| \le \zeta \sqrt{\sigma_n}, \|\psi - \psi_0^s\| \le \cK\sigma_n\right\}  \\\\
    B_{\zeta, 2}(\psi_0^s) & = \left\{\psi: d^2_*(\psi, \psi_0^s) \le \zeta^2, \|\psi - \psi_0^s\| > \cK\sigma_n\right\} \\
    & = \left\{\psi: \|\psi - \psi_0^s\| \le \zeta^2, \|\psi - \psi_0^s\| > \cK\sigma_n\right\} 
\end{align*}
Assume $\cK > 1$. The case where $\cK < 1$ follows from similar calculations and hence skipped for brevity. Consider the following two cases: 
\\\\
\noindent
{\bf Case 1: }Suppose $\zeta \le \sqrt{\cK\sigma_n}$. Then $B_{\zeta, 2} = \phi$. Also as $\cK > 1$, we have: $\zeta\sqrt{\sigma_n} \le \cK\sigma_n$. Hence we have: 
$$
\sup_{d_*^2(\psi, \psi_0^s) \le \zeta^2}\|\psi - \psi_0^s\| = \sup_{B_{\zeta, 1}}\|\psi - \psi_0^s\| = \zeta\sqrt{\sigma_n} \,.
$$
This implies: 
\begin{align*}
    & \sup_{d_*(\th, \th_0^s) \le \zeta} \left|\left\{K\left(\frac{Q^{\top}\psi_0^s}{\sigma_n}\right) -  K\left(\frac{Q^{\top}\psi}{\sigma_n}\right)\right\}\right|^2 \\
    & \le \max\left\{\left|\left\{K\left(\frac{Q^{\top}\psi_0^s}{\sigma_n}\right) -  K\left(\frac{Q^{\top}\psi_0^s}{\sigma_n} + \|\tilde Q\|\frac{\zeta}{\sqrt{\sigma_n}}\right)\right\}\right|^2, \right. \\
    & \qquad \qquad \qquad \left. \left|\left\{K\left(\frac{Q^{\top}\psi_0^s}{\sigma_n}\right) -  K\left(\frac{Q^{\top}\psi_0^s}{\sigma_n} - \|\tilde Q\|\frac{\zeta}{\sqrt{\sigma_n}}\right)\right\}\right|^2\right\} \\
    & := \max\{T_1, T_2\} \,.
\end{align*}
Therefore we have: 
$$
\bbE\left[F^2_{3, \zeta}(X, Y, Q)\right] \le   \bbE[m_4(Q) T_1] +   \bbE[m_4(Q) T_2] \,.
$$
Now: 
\begin{align}
    & \bbE[m_4(Q) T_1] \notag \\
    & = \bbE\left[m_4(Q) \left|\left\{K\left(\frac{Q^{\top}\psi_0^s}{\sigma_n}\right) -  K\left(\frac{Q^{\top}\psi_0^s}{\sigma_n} + \|\tilde Q\|\frac{\zeta}{\sqrt{\sigma_n}}\right)\right\}\right|^2\right] \notag \\
    & = \sigma_n \int_{\bbR^{p-1}}\int_{-\infty}^{\infty} m_4(\sigma_nt - \tilde q^{\top}\tilde \psi_0^s, \tilde q)  \left|K\left(t\right) -  K\left(t + \|\tilde q\|\frac{\zeta}{\sqrt{\sigma_n}}\right)\right|^2 \ f_s(\sigma_nt\mid \tilde q) \ dt \ f(\tilde q) \ d\tilde q \notag\\
    & \le \sigma_n \int_{\bbR^{p-1}}\int_{-\infty}^{\infty} m_4(\sigma_nt - \tilde q^{\top}\tilde \psi_0^s, \tilde q) \left|K\left(t\right) -  K\left(t + \|\tilde q\|\frac{\zeta}{\sqrt{\sigma_n}}\right)\right| \ f_s(\sigma_nt\mid \tilde q) \ dt \ f(\tilde q) \ d\tilde q \notag \\
    & = \sigma_n \int_{\bbR^{p-1}}\int_{-\infty}^{\infty}m_4(\sigma_nt - \tilde q^{\top}\tilde \psi_0^s, \tilde q) \int_{t}^{t + \|\tilde q\|\frac{\zeta}{\sqrt{\sigma_n}}} K'(s) \ ds \ f_s(\sigma_nt\mid \tilde q) \ dt \ f(\tilde q) \ d\tilde q \notag \\
    & = \sigma_n \int_{\bbR^{p-1}}\int_{-\infty}^{\infty}K'(s) \int_{s- \|\tilde q\|\frac{\zeta}{\sqrt{\sigma_n}}}^s m_4(\sigma_nt - \tilde q^{\top}\tilde \psi_0^s, \tilde q) f_s(\sigma_nt\mid \tilde q) \ dt \ ds 
    \ f(\tilde q) \ d\tilde q \notag \\
    & = \zeta \sqrt{\sigma_n} \bbE[\|\tilde Q\|m_4(-\tilde Q^{\top}\psi_0^s, \tilde Q)f_s(0 \mid \tilde Q)] + R \notag 
\end{align}
where as before we split $R$ into three parts $R = R_1 + R_2 + R_3$. 
\begin{align}
    \left|R_1\right| & = \left|\sigma_n \int_{\bbR^{p-1}}\int_{-\infty}^{\infty}K'(s) \int_{s- \|\tilde q\|\frac{\zeta}{\sqrt{\sigma_n}}}^s m_4(- \tilde q^{\top}\tilde \psi_0^s, \tilde q) (f_s(\sigma_nt\mid \tilde q) - f_s(0 \mid \tilde q)) \ dt \ ds \ f(\tilde q) \ d\tilde q\right| \notag \\
   \label{eq:r1_env_1} & \le \sigma_n^2 \int_{\bbR^{p-1}}m_4(- \tilde q^{\top}\tilde \psi_0^s, \tilde q)\dot f_s(\tilde q) \int_{-\infty}^{\infty}K'(s) \int_{s- \|\tilde q\|\frac{\zeta}{\sqrt{\sigma_n}}}^s |t| dt\ ds \ f(\tilde q) \ d\tilde q 
    \end{align}
We next calculate the inner integral (involving $(s,t)$) of equation \eqref{eq:r1_env_1}: 
\begin{align*}
& \int_{-\infty}^{\infty}K'(s) \int_{s- \|\tilde q\|\frac{\zeta}{\sqrt{\sigma_n}}}^s |t| dt\ ds \\
& =\left(\int_{-\infty}^0 + \int_0^{\|\tilde q\|\frac{\zeta}{\sqrt{\sigma_n}}} + \int_{\|\tilde q\|\frac{\zeta}{\sqrt{\sigma_n}}}^{\infty}\right)K'(s)  \int_{s- \|\tilde q\|\frac{\zeta}{\sqrt{\sigma_n}}}^s |t| dt\ ds \\
& = \frac12\int_{-\infty}^0 K'(s)\left[\left(s- \|\tilde q\|\frac{\zeta}{\sqrt{\sigma_n}}\right)^2 - s^2\right] \ ds  +  \frac12\int_0^{\|\tilde q\|\frac{\zeta}{\sqrt{\sigma_n}}} K'(s)\left[\left(s- \|\tilde q\|\frac{\zeta}{\sqrt{\sigma_n}}\right)^2 + s^2\right] \ ds \\
& \qquad \qquad \qquad \qquad + \frac12 \int_{\|\tilde q\|\frac{\zeta}{\sqrt{\sigma_n}}}^{\infty}K'(s) \left[s^2 - \left(s- \|\tilde q\|\frac{\zeta}{\sqrt{\sigma_n}}\right)^2\right] \ ds\\
& = -\|\tilde q\|\frac{\zeta}{\sqrt{\sigma_n}}  \int_{-\infty}^0 K'(s) s  \ ds + \|\tilde q\|^2\frac{\zeta^2}{2\sigma_n}    \int_{-\infty}^0 K'(s)  \ ds + \int_0^{\|\tilde q\|\frac{\zeta}{\sqrt{\sigma_n}}} s^2K'(s) \ ds \\ 
& \qquad \qquad  -\|\tilde q\|\frac{\zeta}{\sqrt{\sigma_n}} \int_0^{\|\tilde q\|\frac{\zeta}{\sqrt{\sigma_n}}} sK'(s) \ ds +  \|\tilde q\|^2\frac{\zeta^2}{2\sigma_n} \int_0^{\|\tilde q\|\frac{\zeta}{\sqrt{\sigma_n}}} K'(s) \ ds \\
& \qquad \qquad \qquad + \|\tilde q\|\frac{\zeta}{\sqrt{\sigma_n}} \int_{\|\tilde q\|\frac{\zeta}{\sqrt{\sigma_n}}}^{\infty} sK'(s) \ ds -  \|\tilde q\|^2\frac{\zeta^2}{2\sigma_n}  \int_{\|\tilde q\|\frac{\zeta}{\sqrt{\sigma_n}}}^{\infty} K'(s) \ ds \\
&  =   \|\tilde q\|^2\frac{\zeta^2}{2\sigma_n}\left[2K\left(\|\tilde q\|\frac{\zeta}{\sqrt{\sigma_n}}\right) - 1\right] + \|\tilde q\|\frac{\zeta}{\sqrt{\sigma_n}} \left[ -\int_{-\infty}^0 K'(s) s  \ ds -  \right. \\
& \qquad \qquad \left. \int_0^{\|\tilde q\|\frac{\zeta}{\sqrt{\sigma_n}}} K'(s)s \ ds + \int_{\|\tilde q\|\frac{\zeta}{\sqrt{\sigma_n}}}^{\infty} sK'(s) \ ds\right]  + \int_0^{\|\tilde q\|\frac{\zeta}{\sqrt{\sigma_n}}} s^2K'(s) \ ds \\
&  =   \|\tilde q\|^2\frac{\zeta^2}{\sigma_n}\left[K\left(\|\tilde q\|\frac{\zeta}{\sqrt{\sigma_n}}\right) - K(0)\right] + \|\tilde q\|\frac{\zeta}{\sqrt{\sigma_n}} \left[ -\int_{-\infty}^{-\|\tilde q\|\frac{\zeta}{\sqrt{\sigma_n}}} K'(s) s  \ ds + \int_{\|\tilde q\|\frac{\zeta}{\sqrt{\sigma_n}}}^{\infty} sK'(s) \ ds\right] \\
& \qquad \qquad + \int_0^{\|\tilde q\|\frac{\zeta}{\sqrt{\sigma_n}}} s^2K'(s) \ ds \\
&  =   \|\tilde q\|^2\frac{\zeta^2}{\sigma_n}\left[K\left(\|\tilde q\|\frac{\zeta}{\sqrt{\sigma_n}}\right) - K(0)\right] + \|\tilde q\|\frac{\zeta}{\sqrt{\sigma_n}}\int_{-\infty}^{\infty} K'(s)|s|\mathds{1}_{|s| \ge \|\tilde q\|\frac{\zeta}{\sqrt{\sigma_n}}}  \ ds + \int_0^{\|\tilde q\|\frac{\zeta}{\sqrt{\sigma_n}}} s^2K'(s) \ ds \\
& \le \dot{K}_+ \|\tilde q\|^3\frac{\zeta^3}{\sigma^{3/2}_n} + \|\tilde q\|\frac{\zeta}{\sqrt{\sigma_n}} \int_{-\infty}^{\infty} K'(s)|s|  \ ds + \|\tilde q\|^2\frac{\zeta^2}{\sigma_n}\left(K\left(\|\tilde q\|\frac{\zeta}{\sqrt{\sigma_n}}\right) - K(0)\right) \\
& \lesssim \|\tilde q\|^3\frac{\zeta^3}{\sigma^{3/2}_n} + \|\tilde q\|\frac{\zeta}{\sqrt{\sigma_n}}  
\end{align*}
Putting this bound in equation \eqref{eq:r1_env_1} we obtain: 
    \begin{align*}
    |R_1| & \le \frac{\sigma_n^2}{2}  \int_{\bbR^{p-1}}m_4(- \tilde q^{\top}\tilde \psi_0^s, \tilde q)\dot f_s(\tilde q) \left(\|\tilde q\|^3\frac{\zeta^3}{\sigma^{3/2}_n} + \|\tilde q\|\frac{\zeta}{\sqrt{\sigma_n}}\right) \ f(\tilde q) \ d\tilde q \\
    & \le \frac{\zeta^3}{2\sqrt{\sigma_n}} \bbE\left[m_4(- \tilde Q^{\top}\tilde \psi_0^s, \tilde Q)\dot f_s(\tilde Q)\|\tilde Q\|^3\right] + \frac{\zeta \sqrt{\sigma_n}}{2} \bbE\left[m_4(- \tilde Q^{\top}\tilde \psi_0^s, \tilde Q)\dot f_s(\tilde Q)\|\tilde Q\|\right] 
\end{align*}
and 
\begin{align*}
    & \left|R_2\right| \\
    & = \left|\sigma_n \int_{\bbR^{p-1}}\int_{-\infty}^{\infty}K'(s) \int_{s- \|\tilde q\|\frac{\zeta}{\sqrt{\sigma_n}}}^s \left(m_4(\sigma_n t - \tilde q^{\top}\tilde \psi_0^s, \tilde q) - m_4( - \tilde q^{\top}\tilde \psi_0^s, \tilde q)\right)f_s(0 \mid \tilde q) \ dt \ ds \ f(\tilde q) \ d\tilde q\right| \\
    & \le \sigma_n^2 \int_{\bbR^{p-1}}\dot m_4( \tilde q)f_s(0 \mid \tilde q) \int_{-\infty}^{\infty}K'(s) \int_{s- \|\tilde q\|\frac{\zeta}{\sqrt{\sigma_n}}}^s |t| dt\ ds \ f(\tilde q) \ d\tilde q \\
    & \le \sigma_n^2 \int_{\bbR^{p-1}}\dot m_4( \tilde q)f_s(0 \mid \tilde q) \left(\|\tilde q\|^3\frac{\zeta^3}{\sigma^{3/2}_n} + \|\tilde q\|\frac{\zeta}{\sqrt{\sigma_n}}\right) \ f(\tilde q) \ d\tilde q \\
    & = \zeta \sigma_n^{3/2} \bbE\left[\dot m_4( \tilde Q)f_s(0 \mid \tilde Q)\|\tilde Q\|\right] + \zeta^3 \sqrt{\sigma_n} \bbE\left[\dot m_4( \tilde Q)f_s(0 \mid \tilde Q)\|\tilde Q\|^3\right]
\end{align*}
The third residual $R_3$ is even higher order term and hence skipped. It is immediate that the order of the remainders are equal to or smaller than $\zeta \sqrt{\sigma_n}$ which implies: 
$$
\bbE[m_4(Q)T_1] \lesssim \zeta\sqrt{\sigma_n} \,.
$$
The calculation for $T_2$ is similar and hence skipped for brevity. Combining conclusions for $T_1$ and $T_2$ we conclude when $\zeta \le \sqrt{\cK \sigma_n}$: 
\begin{align}
& \bbE\left[F^2_{3, \zeta}(X, Y, Q)\right] \notag \\
    & \bbE\left[\left|\left[-2\left(Y - X^{\top}\beta_0^s\right)X^{\top}\delta_0^s + (X^{\top}\delta_0^s)^2\right]\right|^2 \times  \sup_{d_*(\th, \th_0^s) \le \zeta} \left|\left\{K\left(\frac{Q^{\top}\psi_0^s}{\sigma_n}\right) -  K\left(\frac{Q^{\top}\psi}{\sigma_n}\right)\right\}\right|^2\right] \notag \\
    & \lesssim \bbE\left[m_4(Q)\sup_{d_*(\th, \th_0^s) \le \zeta} \left|\left\{K\left(\frac{Q^{\top}\psi_0^s}{\sigma_n}\right) -  K\left(\frac{Q^{\top}\psi}{\sigma_n}\right)\right\}\right|^2\right] \notag \\
    \label{eq:env_3} & \lesssim \zeta \sqrt{\sigma_n} \,.
\end{align}
\\
\noindent
{\bf Case 2: } Now consider $\zeta > \sqrt{\cK \sigma_n}$. Then it is immediate that: 
$$
\sup_{d_*^2(\psi, \psi^s_0) \le \zeta^2} \|\psi - \psi^s_0\| = \zeta^2 \,.
$$
Using this we have: 
\begin{align}
    & \bbE[m_4(Q) T_1] \notag \\
    & = \bbE\left[m_4(Q)\left|\left\{K\left(\frac{Q^{\top}\psi_0^s}{\sigma_n}\right) -  K\left(\frac{Q^{\top}\psi_0^s}{\sigma_n} + \|\tilde Q\|\frac{\zeta^2}{\sqrt{\sigma_n}}\right)\right\}\right|^2\right] \notag \\
    & = \sigma_n \int_{\bbR^{p-1}}\int_{-\infty}^{\infty} m_4(\sigma_nt - \tilde q^{\top}\tilde \psi_0^s, \tilde q)  \left|K\left(t\right) -  K\left(t + \|\tilde q\|\frac{\zeta^2}{\sigma_n}\right)\right|^2 \ f_s(\sigma_nt\mid \tilde q) \ dt \ f(\tilde q) \ d\tilde q \notag \\
    & \le \sigma_n \int_{\bbR^{p-1}}\int_{-\infty}^{\infty} m_4(\sigma_nt - \tilde q^{\top}\tilde \psi_0^s, \tilde q)  \left|K\left(t\right) -  K\left(t + \|\tilde q\|\frac{\zeta^2}{\sigma_n}\right)\right| \ f_s(\sigma_nt\mid \tilde q) \ dt \ f(\tilde q) \ d\tilde q \notag \\
     & \le \sigma_n \int_{\bbR^{p-1}}\int_{-\infty}^{\infty} m_4(\sigma_nt - \tilde q^{\top}\tilde \psi_0^s, \tilde q)\|\tilde q\|\frac{\zeta^2}{\sigma_n} \ f_s(\sigma_nt\mid \tilde q) \ dt \ f(\tilde q) \ d\tilde q \notag \\
      & = \zeta^2 \int_{\bbR^{p-1}}m_4(- \tilde q^{\top}\tilde \psi_0^s, \tilde q)  f_s(0 \mid \tilde q)\|\tilde q\| \ f(\tilde q) \ d\tilde q + R \notag\\
    & \le \zeta^2  \bbE\left[\|\tilde Q\|m_4\left(- \tilde Q^{\top}\tilde \psi_0^s, \tilde Q\right)  f_s(0 \mid \tilde Q)\right]  + R \notag 
\end{align}
The analysis of the remainder term is similar and if is of higher order. This concludes when $\zeta > \sqrt{K\sigma_n}$: 
\begin{align}
& \bbE\left[F^2_{3, \zeta}(X, Y, Q)\right] \notag \\
    & \bbE\left[\left|\left[-2\left(Y - X^{\top}\beta_0^s\right)X^{\top}\delta_0^s + (X^{\top}\delta_0^s)^2\right]\right|^2 \times  \sup_{d_*(\th, \th_0^s) \le \zeta} \left|\left\{K\left(\frac{Q^{\top}\psi_0^s}{\sigma_n}\right) -  K\left(\frac{Q^{\top}\psi}{\sigma_n}\right)\right\}\right|^2\right] \notag \\
    & \lesssim \bbE\left[m_4(Q)\sup_{d_*(\th, \th_0^s) \le \zeta} \left|\left\{K\left(\frac{Q^{\top}\psi_0^s}{\sigma_n}\right) -  K\left(\frac{Q^{\top}\psi}{\sigma_n}\right)\right\}\right|^2\right] \notag \\
    \label{eq:env_4}  & \lesssim \zeta^2
\end{align}
Combining \eqref{eq:env_3}, \eqref{eq:env_4} with equation \eqref{eq:moc_bound_2} we have:
\begin{align*}
\sqrt{n}\bbE\left[\sup_{\theta: d_*(\theta, \theta_0^s) \le \zeta} \left|\left(\bbP_n - P\right)\left(f_\theta - f_{\theta_0^s}\right)\right|\right] & \lesssim  \sqrt{\zeta}\sigma_n^{1/4}\mathds{1}_{\zeta \le \sqrt{\cK\sigma_n}} + \zeta \mathds{1}_{\zeta > \sqrt{\cK \sigma_n}} \\
& := \phi_n(\zeta) \,.
\end{align*}
Hence to obtain rate we have to solve $r_n^2 \phi_n(1/r_n) \le \sqrt{n}$, i.e. (ignoring $\cK$ as this does not affect the rate)
$$
r_n^{3/2}\sigma_n^{1/4}\mathds{1}_{r_n \ge \sigma_n^{-1/2}} + r_n \mathds{1}_{r_n \le \sigma_n^{-1/2}} \le \sqrt{n} \,.
$$
Now if $r_n \le \sigma_n^{-1/2}$ then $r_n = \sqrt{n}$ which implies $\sqrt{n} \le \sigma_n^{-1/2}$ i.e. $n\sigma_n \to 0$ and hence contradiction. On the other hand, if $r_n \ge \sigma_n^{-1/2}$ then $r_n = n^{1/3}\sigma_n^{-1/6}$. This implies $n^{1/3}\sigma_n^{-1/6} \ge \sigma_n^{-1/2}$, i.e. $n^{1/3} \ge \sigma_n^{-1/3}$, i.e. $n\sigma_n \to \infty$ which is okay. This implies: 
$$
n^{2/3}\sigma_n^{-1/3}d^2(\hat \th^s, \th_0^s) = O_p(1) \,.
$$
Now as $n^{2/3}\sigma_n^{-1/3} \gg \sigma_n^{-1}$, we have: 
$$
\frac{1}{\sigma_n}d^2(\hat \th^s, \th_0^s) = o_p(1) \,.
$$
which further indicates $\|\hat \psi^s - \psi_0^s\|/\sigma_n = o_p(1)$. This, along with the fact that $\|\psi_0^s - \psi_0\|/\sigma_n = o(1)$ (from Lemma \ref{bandwidth}), establishes that $\|\hat \psi_0^s - \psi_0\|/\sigma_n = o_p(1)$. This completes the proof. 
\end{proof}

\section{Supplementary Lemmas for the proof of Theorem \ref{thm:regression}}
\label{sec:supp_B}
\subsection{Proof of Lemma \ref{bandwidth}}
\begin{proof}
First we establish the fact that $\theta_0^s \to \theta_0$. Note that for all $n$, we have: 
$$
\bbM^s(\theta_0^s) \le \bbM^s(\theta_0) 
$$
Taking $\limsup$ on the both side we have: 
$$
\limsup_{n \to \infty} \bbM^s(\theta_0^s) \le \bbM(\theta_0) \,.
$$
Now using Lemme \ref{lem:uniform_smooth} we have: 
$$
\limsup_{n \to \infty} \bbM^s(\theta_0^s) = \limsup_{n \to \infty} \left[\bbM^s(\theta_0^s) - \bbM(\theta_0^s) + \bbM(\theta_0^s)\right] = \limsup_{n \to \infty} \bbM(\theta_0^s) \,.
$$
which implies $\limsup_{n \to \infty} \bbM(\theta_0^s) \le \bbM(\theta_0)$ and from the continuity of $\bbM(\theta)$ and $\theta_0$ being its unique minimizer, we conclude the proof. Now, using Lemma \ref{lem:pop_curv_nonsmooth} and Lemma \ref{lem:uniform_smooth} we further obtain: 
\begin{align}
   u_- d^2(\theta_0^s, \theta_0) & \le \bbM(\th_0^s) - \bbM(\theta_0) \notag \\
    & = \bbM(\th_0^s) - \bbM^s(\th^s_0) + \underset{\le 0}{\underline{\bbM^s(\th_0^s) - \bbM^s(\theta_0)}} + \bbM^s(\th_0) - \bbM(\theta_0) \notag \\
    \label{eq:est_dist_bound} & \le \sup_{\theta \in \Theta}\left|\bbM^s(\theta) - \bbM(\theta)\right| \le K_1 \sigma_n \,. 
\end{align}
Note that we neeed consistency of $\th_0^s$ here as the lower bound in Lemma \ref{lem:pop_curv_nonsmooth} is only valid in a neighborhood around $\th_0$. As $\th_0^s$ is the minimizer of $\bbM^s(\th)$, from the first order condition we have: 
\begin{align}
    \label{eq:beta_grad}\nabla_{\beta}\bbM^s_n(\theta_0^s) & = -2\bbE\left[X(Y - X^{\top}\beta_0^s)\right] + 2\bbE \left\{\left[X_iX_i^{\top}\delta_0^s\right] K\left(\frac{Q_i^{\top}\psi_0^s}{\sigma_n}\right)\right\}  = 0 \\
    \label{eq:delta_grad}\nabla_{\delta}\bbM^s_n(\th_0^s) & = \bbE \left\{\left[-2X_i\left(Y_i - X_i^{\top}\beta_0^s\right) + 2X_iX_i^{\top}\delta_0^s\right] K\left(\frac{Q_i^{\top}\psi_0^s}{\sigma_n}\right)\right\} = 0\\
    \label{eq:psi_grad}\nabla_{\psi}\bbM^s_n(\th_0^s) & = \frac{1}{\sigma_n}\bbE \left\{\left[-2\left(Y_i - X_i^{\top}\beta_0^s\right)X_i^{\top}\delta_0^s + (X_i^{\top}\delta_0^s)^2\right]\tilde Q_i K'\left(\frac{Q_i^{\top}\psi_0^s}{\sigma_n}\right)\right\} = 0
\end{align}
We first show that $(\tilde \psi^s_0 - \tilde \psi_0)/\sigma_n \to 0$ by \emph{reductio ab absurdum}. From equation \eqref{eq:est_dist_bound}, we know $\|\psi_0^s - \psi_0\|/\sigma_n = O(1)$. Hence it has a convergent subsequent $\psi^s_{0, n_k}$, where $(\tilde \psi^s_{0, n_k} - \tilde \psi_0)/\sigma_n \to h$. If we can prove that $h = 0$, then we establish every subsequence of $\|\psi_0^s - \psi_0\|/\sigma_n$ has a further subsequence which converges to $0$ which further implies $\|\psi_0^s - \psi_0\|/\sigma_n$ converges to $0$. To save some notations, we prove that if $(\psi_0^s - \psi_0)/\sigma_n \to h$ then $h = 0$. We start with equation \eqref{eq:psi_grad}. Define $\tilde \eta = (\tilde \psi^s_0 - \tilde \psi_0)/\sigma_n = (\psi_0^s - \psi_0)/\sigma_n$ where $\tilde \psi$ is all the co-ordinates of $\psi$ except the first one, as the first co-ordinate of $\psi$ is always assumed to be $1$ for identifiability purpose. 
\allowdisplaybreaks
\begin{align}
    0 & =  \frac{1}{\sigma_n}\bbE \left\{\left[-2\left(Y_i - X_i^{\top}\beta_0^s\right)X_i^{\top}\delta_0^s + (X_i^{\top}\delta_0^s)^2\right]\tilde Q_i K'\left(\frac{Q_i^{\top}\psi_0^s}{\sigma_n}\right)\right\}  \notag \\
    & =  \frac{1}{\sigma_n}\bbE\left[\left( -2\delta_0^s XX^{\top}(\beta_0 - \beta^s_0) -2\delta_0^s XX^{\top}\delta_0\mathds{1}_{Q^{\top}\delta_0 > 0} + (X_i^{\top}\delta_0^s)^2\right)\tilde QK'\left(\frac{Q^{\top}\psi^s_0}{\sigma_n}\right)\right] \notag \\
     & =  \frac{1}{\sigma_n}\bbE\left[\left( -2\delta_0^s XX^{\top}(\beta_0 - \beta^s_0) -2\delta_0^s XX^{\top}(\delta_0 - \delta_0^s)
     \mathds{1}_{Q^{\top}\delta_0 > 0} \right. \right. \notag \\
     & \hspace{10em} \left. \left. + (X_i^{\top}\delta_0^s)^2\left(1 - 2\mathds{1}_{Q^{\top}\delta_0 > 0}\right)\right)\tilde QK'\left(\frac{Q^{\top}\psi^s_0}{\sigma_n}\right)\right] \notag \\
    & =  \frac{-2}{\sigma_n}\bbE\left[\left(\delta_0^{s^{\top}} g(Q)(\beta_0 - \beta^s_0)\right)\tilde QK'\left(\frac{Q^{\top}\psi^s_0}{\sigma_n}\right)\right] \notag \\
    & \qquad \qquad \qquad -  \frac{2}{\sigma_n} \bbE\left[\left(\delta_0^{s^{\top}}g(Q)(\delta_0 - \delta^s_0)\right)\tilde QK'\left(\frac{Q^{\top}\psi^s_0}{\sigma_n}\right)\mathds{1}_{Q^{\top}\delta_0 > 0}\right] \notag \\
    & \hspace{15em} +   \frac{1}{\sigma_n}\bbE\left[\left(\delta_0^{s^{\top}}g(Q)\delta^s_0\right)\tilde QK'\left(\frac{Q^{\top}\psi^s_0}{\sigma_n}\right)\left(1 - 2\mathds{1}_{Q^{\top}\delta_0 > 0}\right)\right] \notag \\
     & = -\underbrace{\frac{2}{\sigma_n}\bbE\left[\left(\delta_0^{s^{\top}} g(Q)(\beta_0 - \beta^s_0)\right)\tilde QK'\left(\frac{Q^{\top}\psi^s_0}{\sigma_n}\right)\right]}_{T_1} \notag \\
     & \qquad \qquad -\underbrace{\frac{2}{\sigma_n} \bbE\left[\left(\delta_0^{s^{\top}}g(Q)(\delta_0 - \delta^s_0)\right)\tilde QK'\left(\frac{Q^{\top}\psi^s_0}{\sigma_n}\right)\mathds{1}_{Q^{\top}\delta_0 > 0}\right]}_{T_2} \notag \\
     \label{eq:pop_est_conv_1} & \qquad \qquad \qquad  + \underbrace{\frac{1}{\sigma_n}\bbE\left[\left(\delta_0{\top}g(Q)\delta_0\right)\tilde QK'\left(\frac{Q^{\top}\psi^s_0}{\sigma_n}\right)\left(1 - 2\mathds{1}_{Q^{\top}\delta_0 > 0}\right)\right]}_{T_3} \notag \\
     & \qquad \qquad \qquad \qquad  + \underbrace{\frac{2}{\sigma_n}\bbE\left[\left((\delta_0 - \delta_0^s)^{\top}g(Q)\delta_0\right)\tilde QK'\left(\frac{Q^{\top}\psi^s_0}{\sigma_n}\right)\left(1 - 2\mathds{1}_{Q^{\top}\delta_0 > 0}\right)\right]}_{T_4} \notag \\
 & = T_1 + T_2 + T_3 + T_4
 \end{align}
As mentioned earlier, there is a bijection between $(Q_1, \tilde Q)$ and $(Q^{\top}\psi_0, \tilde Q)$. The map of one side is obvious. The other side is also trivial as the first coordinate of $\psi_0$ is 1, which makes $Q^{\top}\psi_0 = Q_1 + \tilde Q^{\top}\tilde \psi_0$:  
$$
(Q^{\top}\psi_0, \tilde Q) \mapsto (Q^{\top}\psi_0 - \tilde Q^{\top}\tilde \psi_0, \tilde Q) \,.
$$
We first show that $T_1, T_2$ and $T_4$ are $o(1)$. Towards that end first note that: 
\begin{align*}
|T_1| & \le \frac{2}{\sigma_n}\bbE\left[\|g(Q)\|_{op} \ \|\tilde Q\| \ \left|K'\left(\frac{Q^{\top}\psi^s_0}{\sigma_n}\right)\right|\right]\|\delta_0^s\|\|\beta_0 - \beta_0^s\| \\
|T_2| & \le \frac{2}{\sigma_n} \bbE\left[\|g(Q)\|_{op} \ \|\tilde Q\| \ \left|K'\left(\frac{Q^{\top}\psi^s_0}{\sigma_n}\right)\right|\right]\|\delta_0^s\|\|\delta_0 - \delta_0^s\| \\
|T_4| & \le \frac{2}{\sigma_n} \bbE\left[\|g(Q)\|_{op} \ \|\tilde Q\| \ \left|K'\left(\frac{Q^{\top}\psi^s_0}{\sigma_n}\right)\right|\right]\|\delta_0^s\|\|\delta_0 - \delta_0^s\|
\end{align*}
From the above bounds, it is immediate that to show that above terms are $o(1)$ all we need to show is: 
$$
 \frac{1}{\sigma_n}\bbE\left[\|g(Q)\|_{op} \ \|\tilde Q\| \ \left|K'\left(\frac{Q^{\top}\psi^s_0}{\sigma_n}\right)\right|\right] = O(1) \,.
$$
Towards that direction, define $\eta = (\tilde \psi_0^s - \tilde \psi_0)/\sigma_n$: 
\begin{align*}
& \frac{1}{\sigma_n}\bbE\left[\|g(Q)\|_{op} \ \|\tilde Q\| \ \left|K'\left(\frac{Q^{\top}\psi^s_0}{\sigma_n}\right)\right|\right] \\
& \le c_+  \frac{1}{\sigma_n}\bbE\left[\|\tilde Q\| \ \left|K'\left(\frac{Q^{\top}\psi^s_0}{\sigma_n}\right)\right|\right]  \\
& = c_+ \frac{1}{\sigma_n}\int \int \|\tilde q\| \left|K'\left(\frac{t}{\sigma_n} + \tilde q^{\top}\eta \right)\right| f_0\left(t \mid \tilde q\right) f(\tilde q) \ dt \ d\tilde q \\
& = c_+ \int \int \|\tilde q\| \left|K'\left(t + \tilde q^{\top}\eta \right)\right| f_0\left(\sigma_n t \mid \tilde q\right) f(\tilde q) \ dt \ d\tilde q \\
& =  c_+ \int  \|\tilde q\| f_0\left(0 \mid \tilde q\right) \int \left|K'\left(t + \tilde q^{\top}\eta \right)\right| \ dt  \ f(\tilde q) \ d\tilde q + R_1 \\
& = c_+ \int \left|K'\left(t\right)\right| dt \ \bbE\left[\|\tilde Q\| f_0(0 \mid \tilde Q)\right] + R_1 = O(1) + R_1 \,.
\end{align*}
Therefore, all it remains to show is $R_1$ is also $O(1)$ (or of smaller order): 
\begin{align*}
|R_1| & = \left|c_+ \int \int \|\tilde q\| \left|K'\left(t + \tilde q^{\top}\eta \right)\right| \left(f_0\left(\sigma_n t \mid \tilde q\right) - f_0(0 \mid \tilde q) \right)f(\tilde q) \ dt \ d\tilde q\right| \\
& \le c_+ F_+ \sigma_n \int \|\tilde q\| \int_{-\infty}^{\infty} |t|\left|K'\left(t + \tilde q^{\top}\eta \right)\right| \ dt \ f(\tilde q) \ d\tilde q \\
& = c_+ F_+ \sigma_n \int \|\tilde q\| \int_{-\infty}^{\infty} |t - q^{\top}\eta|\left|K'\left(t\right)\right| \ dt \ f(\tilde q) \ d\tilde q \\
& \le c_+ F_+ \sigma_n \left[\int \|\tilde q\| \int_{-\infty}^{\infty} |t|\left|K'\left(t\right)\right| \ dt \ f(\tilde q) \ d\tilde q  + \int \|\tilde q\|^2\|\eta\| \int_{-\infty}^{\infty}\left|K'\left(t\right)\right| \ dt \ f(\tilde q) \ d\tilde q\right] \\
& = c_+ F_+ \sigma_n \left[\left(\int_{-\infty}^{\infty} |t|\left|K'\left(t\right)\right| \ dt\right) \times \bbE[\|\tilde Q\|] +  \left(\int_{-\infty}^{\infty}\left|K'\left(t\right)\right| \ dt\right) \times \|\eta\| \ \bbE[\|\tilde Q\|^2]\right] \\
& = O(\sigma_n) = o(1) \,.
\end{align*}
This completes the proof.  For $T_3$, the limit is non-degenerate which can be calculated as follows: 
\begin{align*}
T_3 &= \frac{1}{\sigma_n}\bbE\left[\left(\delta_0{\top}g(Q)\delta_0\right)\tilde QK'\left(\frac{Q^{\top}\psi^s_0}{\sigma_n}\right)\left(1 - 2\mathds{1}_{Q^{\top}\delta_0 > 0}\right)\right] \\
& =  \frac{1}{\sigma_n} \int \int \left(\delta_0{\top}g(t - \tilde q^{\top}\tilde \psi_0, \tilde q)\delta_0\right)\tilde q K'\left(\frac{t}{\sigma_n} + \tilde q^{\top} \eta\right)\left(1 - 2\mathds{1}_{t > 0}\right) \ f_0(t \mid \tilde q) \ f(\tilde q) \ dt \ d\tilde q \\
& = \int \int \left(\delta_0{\top}g(\sigma_n t - \tilde q^{\top}\tilde \psi_0, \tilde q)\delta_0\right)\tilde q K'\left(t + \tilde q^{\top} \eta\right)\left(1 - 2\mathds{1}_{t > 0}\right) \ f_0(\sigma_n t \mid \tilde q) \ f(\tilde q) \ dt \ d\tilde q \\
& = \int \int \left(\delta_0{\top}g(- \tilde q^{\top}\tilde \psi_0, \tilde q)\delta_0\right)\tilde q K'\left(t + \tilde q^{\top} \eta\right)\left(1 - 2\mathds{1}_{t > 0}\right) \ f_0(0 \mid \tilde q) \ f(\tilde q) \ dt \ d\tilde q + R \\
& = \int  \left(\delta_0{\top}g(- \tilde q^{\top}\tilde \psi_0, \tilde q)\delta_0\right)\tilde q f_0(0 \mid \tilde q) \left[\int_{-\infty}^0 K'\left(t + \tilde q^{\top} \eta\right) \ dt - \int_0^\infty K'\left(t + \tilde q^{\top}\tilde \eta\right) \ dt \right] \ f(\tilde q) \ d\tilde q + R \\
&=  \int  \left(\delta_0{\top}g(- \tilde q^{\top}\tilde \psi_0, \tilde q)\delta_0\right)\tilde q f_0(0 \mid \tilde q)\left(2K\left(\tilde q^{\top}\eta\right) - 1\right) \ f(\tilde q) \ d\tilde q + R  \\
& =   \bbE\left[\tilde Q f(0 \mid \tilde Q) \left(\delta_0^{\top}g(- \tilde Q^{\top}\tilde \psi_0, \tilde Q)\delta_0\right)\left(2K(\tilde Q^{\top} \eta)- 1\right)\right] + R 
\end{align*} 
That the remainder $R$ is $o(1)$ again follows by similar calculation as before and hence skipped. Therefore we have when $\eta = (\tilde \psi_0^s - \psi_0)/\sigma_n \to h$: 
$$
T_3 \overset{n \to \infty}{\longrightarrow}  \bbE\left[\tilde Q f(0 \mid \tilde Q) \left(\delta_0^{\top}g(- \tilde Q^{\top}\tilde \psi_0, \tilde Q)\delta_0\right)\left(2K(\tilde Q^{\top}h)- 1\right)\right]  \,,
$$
which along with equation \eqref{eq:pop_est_conv_1} implies: 
$$
\bbE\left[\tilde Q f(0 \mid \tilde Q) \left(\delta_0^{\top}g(- \tilde Q^{\top}\tilde \psi_0, \tilde Q)\delta_0\right)\left(2K(\tilde Q^{\top}h)- 1\right)\right]  = 0 \,.
$$
Taking inner product with respect to $h$ on both side of the above equation we obtain: 
$$
\bbE\left[\tilde Q^{\top}h f(0 \mid \tilde Q) \left(\delta_0^{\top}g(- \tilde Q^{\top}\tilde \psi_0, \tilde Q)\delta_0\right)\left(2K(\tilde Q^{\top}h)- 1\right)\right]  = 0
$$
Now from the symmetry of our Kernel $K$ we have $\left(\delta_0^{\top}g(- \tilde Q^{\top}\tilde \psi_0, \tilde Q)
\delta_0\right)\tilde Q^{\top}h f(0 \mid \tilde Q) (2K(\tilde Q^{\top}\tilde h) - 1) \ge 0$ almost surely. As the expectation is $0$, we further deduce that $\tilde Q^{\top}h f(0 \mid \tilde Q) (2K(\tilde Q^{\top}\tilde h)-1) = 0$ almost surely, which further implies $h = 0$. 
\\\\
\noindent
We next prove that $(\beta_0 - \beta^s_0)/\sqrt{\sigma_n} \to 0$ and $(\delta_0 - \delta^s_0)/\sqrt{\sigma_n} \to 0$ using equations\eqref{eq:beta_grad} and \eqref{eq:delta_grad}. We start with equation \eqref{eq:beta_grad}: 
\begin{align}
    0 & = -\bbE\left[X(Y - X^{\top}\beta_0^s)\right] + \bbE \left\{\left[X_iX_i^{\top}\delta_0^s\right] K\left(\frac{Q_i^{\top}\psi_0^s}{\sigma_n}\right)\right\}  \notag \\
    & = -\bbE\left[XX^{\top}(\beta_0 - \beta_0^s)\right] - \bbE[XX^{\top}\delta_0\mathds{1}_{Q^{\top}\psi_0 > 0}] + \bbE \left[ g(Q)K\left(\frac{Q_i^{\top}\psi_0^s}{\sigma_n}\right)\right]\delta_0^s \notag \\
    & = -\Sigma_X(\beta_0 - \beta_0^s) -\bbE\left[g(Q)\mathds{1}_{Q^{\top}\psi_0 > 0}\right](\delta_0 - \delta_0^s) + \bbE \left[g(Q)\left\{K\left(\frac{Q_i^{\top}\psi_0^s}{\sigma_n}\right) - \mathds{1}_{Q^{\top}\psi_0 > 0}\right\}\right]\delta_0^s \notag \\
     \label{eq:deriv1} & = \Sigma_X\frac{(\beta_0^2 - \beta_0)}{\sigma_n}  + \bbE\left[g(Q)\mathds{1}_{Q^{\top}\psi_0 > 0}\right]\frac{(\delta_0^2 - \delta_0)}{\sigma_n} + \frac{1}{\sigma_n}\bbE \left[g(Q)\left\{K\left(\frac{Q_i^{\top}\psi_0^s}{\sigma_n}\right) - \mathds{1}_{Q^{\top}\psi_0 > 0}\right\}\right]\delta_0^s \notag \\  
     & = \left( \bbE\left[g(Q)\mathds{1}_{Q^{\top}\psi_0 > 0}\right]\right)^{-1}\Sigma_X \frac{(\beta_0^2 - \beta_0)}{\sigma_n}  + \frac{\delta_0^s - \delta_0}{\sigma_n}  \notag \\
     & \qquad \qquad \qquad \qquad +  \left( \bbE\left[g(Q)\mathds{1}_{Q^{\top}\psi_0 > 0}\right]\right)^{-1} \frac{1}{\sigma_n}\bbE \left[g(Q)\left\{K\left(\frac{Q_i^{\top}\psi_0^s}{\sigma_n}\right) - \mathds{1}_{Q^{\top}\psi_0 > 0}\right\}\right]\delta_0^s 
\end{align}
From equation \eqref{eq:delta_grad} we have:
\begin{align}
    0 & = \bbE \left\{\left[-X\left(Y - X^{\top}\beta_0^s\right) + XX^{\top}\delta_0^s\right] K\left(\frac{Q^{\top}\psi_0^s}{\sigma_n}\right)\right\} \notag \\
    & = -\bbE\left[g(Q)K\left(\frac{Q^{\top}\psi_0^s}{\sigma_n}\right)\right](\beta_0 - \beta_0^s) -  \bbE\left[g(Q)K\left(\frac{Q^{\top}\psi_0^s}{\sigma_n}\right)\mathds{1}_{Q^{\top}\psi_0 > 0}\right]\delta_0 \notag \\
    & \hspace{20em}+  \bbE\left[g(Q)K\left(\frac{Q^{\top}\psi_0^s}{\sigma_n}\right)\right]\delta_0^s \notag \\
     & = -\bbE\left[g(Q)K\left(\frac{Q^{\top}\psi_0^s}{\sigma_n}\right)\right](\beta_0 - \beta_0^s) -  \bbE\left[g(Q)K\left(\frac{Q^{\top}\psi_0^s}{\sigma_n}\right)\mathds{1}_{Q^{\top}\psi_0 > 0}\right](\delta_0 - \delta_0^s) \notag \\
     & \hspace{20em} +   \bbE\left[g(Q)K\left(\frac{Q^{\top}\psi_0^s}{\sigma_n}\right)\left(1 - \mathds{1}_{Q^{\top}\psi_0 > 0}\right)\right]\delta_0^s  \notag \\
        & = \bbE\left[g(Q)K\left(\frac{Q^{\top}\psi_0^s}{\sigma_n}\right)\right]\frac{(\beta_0^s - \beta_0)}{\sigma_n} + \bbE\left[g(Q)K\left(\frac{Q^{\top}\psi_0^s}{\sigma_n}\right)\mathds{1}_{Q^{\top}\psi_0 > 0}\right]\frac{(\delta^s_0 - \delta_0)}{\sigma_n} \notag \\
     \label{eq:deriv2} & \hspace{20em} +   \frac{1}{\sigma_n}\bbE\left[g(Q)K\left(\frac{Q^{\top}\psi_0^s}{\sigma_n}\right)\left(1 - \mathds{1}_{Q^{\top}\psi_0 > 0}\right)\right]\delta_0^s  \notag \\ 
     & = \left( \bbE\left[g(Q)K\left(\frac{Q^{\top}\psi_0^s}{\sigma_n}\right)\mathds{1}_{Q^{\top}\psi_0 > 0}\right]\right)^{-1}\bbE\left[g(Q)K\left(\frac{Q^{\top}\psi_0^s}{\sigma_n}\right)\right]\frac{(\beta_0^s - \beta_0)}{\sigma_n} + \frac{(\delta^s_0 - \delta_0)}{\sigma_n} \notag \\
     & \qquad \qquad \qquad + \left( \bbE\left[g(Q)K\left(\frac{Q^{\top}\psi_0^s}{\sigma_n}\right)\mathds{1}_{Q^{\top}\psi_0 > 0}\right]\right)^{-1} \frac{1}{\sigma_n}\bbE\left[g(Q)K\left(\frac{Q^{\top}\psi_0^s}{\sigma_n}\right)\left(1 - \mathds{1}_{Q^{\top}\psi_0 > 0}\right)\right]\delta_0^s 
     \end{align}
Subtracting equation \eqref{eq:deriv2} from \eqref{eq:deriv1} we obtain: 
$$
0 = A_n \frac{(\beta_0^s - \beta_0)}{\sigma_n} + b_n \,,
$$
i.e. 
$$
\lim_{n \to \infty}  \frac{(\beta_0^s - \beta_0)}{\sigma_n} = \lim_{n \to \infty} -A_n^{-1}b_n \,.
$$
where: 
\begin{align*}
A_n & =  \left( \bbE\left[g(Q)\mathds{1}_{Q^{\top}\psi_0 > 0}\right]\right)^{-1}\Sigma_X   \\
& \qquad \qquad  - \left( \bbE\left[g(Q)K\left(\frac{Q^{\top}\psi_0^s}{\sigma_n}\right)\mathds{1}_{Q^{\top}\psi_0 > 0}\right]\right)^{-1}\bbE\left[g(Q)K\left(\frac{Q^{\top}\psi_0^s}{\sigma_n}\right)\right] \\
b_n & =  \left( \bbE\left[g(Q)\mathds{1}_{Q^{\top}\psi_0 > 0}\right]\right)^{-1} \frac{1}{\sigma_n}\bbE \left[g(Q)\left\{K\left(\frac{Q_i^{\top}\psi_0^s}{\sigma_n}\right) - \mathds{1}_{Q^{\top}\psi_0 > 0}\right\}\right]\delta_0^s  \\
& \qquad  - \left( \bbE\left[g(Q)K\left(\frac{Q^{\top}\psi_0^s}{\sigma_n}\right)\mathds{1}_{Q^{\top}\psi_0 > 0}\right]\right)^{-1} \frac{1}{\sigma_n}\bbE\left[g(Q)K\left(\frac{Q^{\top}\psi_0^s}{\sigma_n}\right)\left(1 - \mathds{1}_{Q^{\top}\psi_0 > 0}\right)\right]\delta_0^s 
\end{align*}
It is immediate via DCT that as $n \to \infty$: 
\begin{align}
     \label{eq:limit_3} \bbE\left[g(Q)K\left(\frac{Q^{\top}\psi_0^s}{\sigma_n}\right)\right] & \longrightarrow \bbE\left[g(Q)\mathds{1}_{Q^{\top}\psi_0 > 0}\right] \,. \\
     \label{eq:limit_4} \bbE\left[g(Q)K\left(\frac{Q^{\top}\psi_0^s}{\sigma_n}\right)\mathds{1}_{Q^{\top}\psi_0 > 0}\right] & \longrightarrow \bbE\left[g(Q)\mathds{1}_{Q^{\top}\psi_0 > 0}\right] \,.
\end{align}
From equation \eqref{eq:limit_3} and \eqref{eq:limit_4} it is immediate that: 
\begin{align*}
\lim_{n \to \infty} A_n  & =  \left( \bbE\left[g(Q)\mathds{1}_{Q^{\top}\psi_0 > 0}\right]\right)^{-1}\Sigma_X - I \\
& = \left( \bbE\left[g(Q)\mathds{1}_{Q^{\top}\psi_0 > 0}\right]\right)^{-1}\left( \bbE\left[g(Q)\mathds{1}_{Q^{\top}\psi_0 \le 0}\right]\right) := A\,.
\end{align*}
Next observe that: 
\begin{align}
    & \frac{1}{\sigma_n} \bbE\left[g(Q)\left\{K\left(\frac{Q^{\top}\psi_0^s}{\sigma_n}\right) - \mathds{1}_{Q^{\top}\psi_0 > 0}\right\}\right] \notag \\
    & = \frac{1}{\sigma_n} \bbE\left[g(Q)\left\{K\left(\frac{Q^{\top}\psi_0}{\sigma_n} + \tilde Q^{\top}\tilde \eta\right) - \mathds{1}_{Q^{\top}\psi_0 > 0}\right\}\right] \notag \\
    & = \int_{\bbR^{p-1}} \int_{-\infty}^{\infty} g(\sigma_nt - \tilde q^{\top}\tilde \psi_0, \tilde q)\left[K\left(t + \tilde q^{\top}\tilde \eta\right) - \mathds{1}_{t > 0}\right] f(\sigma_n t \mid \tilde q) \ dt \ f(\tilde q) \ d\tilde q  \notag \\
    \label{eq:limit_1} & \longrightarrow \bbE\left[g(-\tilde Q^{\top}\tilde \psi_0, \tilde Q)f(0 \mid \tilde Q)\right] \cancelto{0}{\int_{-\infty}^{\infty} \left[K\left(t\right) - \mathds{1}_{t > 0}\right] \ dt} \,. 
\end{align}
Similar calculation yields: 
\begin{align}
    \label{eq:limit_2} & \frac{1}{\sigma_n} \bbE\left[g(Q)K\left(\frac{Q^{\top}\psi_0^s}{\sigma_n}\right)\left(1 - \mathds{1}_{Q^{\top}\psi_0 > 0}\right)\right] \notag \\
    & \longrightarrow \bbE[g(-\tilde Q^{\top}\tilde \psi_0, \tilde Q)f_0(0 \mid \tilde Q)]\int_{-\infty}^{\infty} \left[K\left(t\right)\mathds{1}_{t \le 0}\right] \ dt \,.
\end{align}
Combining equation \eqref{eq:limit_1} and \eqref{eq:limit_2} we conclude: 
\begin{align*}
\lim_{n \to \infty} b_n &= \left( \bbE\left[g(Q)\mathds{1}_{Q^{\top}\psi_0 > 0}\right]\right)^{-1} \bbE[g(-\tilde Q^{\top}\tilde \psi_0, \tilde Q)\delta_0f_0(0 \mid \tilde Q)]\int_{-\infty}^{\infty} \left[K\left(t\right)\mathds{1}_{t \le 0}\right] \ dt  \\
& := b \,.
\end{align*}
which further implies, 
$$
\lim_{n \to \infty}  \frac{(\beta_0^s - \beta_0)}{\sigma_n} = -A^{-1}b \implies (\beta_0^s - \beta_0) = o(\sqrt{\sigma_n})\,,
$$
and by similar calculations: 
$$
(\delta_0^s - \delta_0) = o(\sqrt{\sigma_n}) \,.
$$
This completes the proof. 
\end{proof}

\subsection{Proof of Lemma \ref{lem:pop_curv_nonsmooth}}
\begin{proof}
From the definition of $M(\theta)$ it is immediate that $\bbM(\theta_0) = \bbE[\eps^2] = \sigma^2$. For any general $\th$: 
\begin{align*}
    \bbM(\th) & = \bbE\left[\left(Y - X^{\top}\left(\beta + \delta\mathds{1}_{Q^{\top}\psi > 0}\right)\right)^2\right] \\
    & = \sigma^2 + \bbE\left[\left( X^{\top}\left(\beta + \delta\mathds{1}_{Q^{\top}\psi > 0} - \beta_0 - \delta_0\mathds{1}_{Q^{\top}\psi_0 > 0}\right)\right)^2\right] \\
    & \ge \sigma^2 + c_- \bbE_Q\left[\left\|\beta - \beta_0 + \delta\mathds{1}_{Q^{\top}\psi > 0}- \delta_0\mathds{1}_{Q^{\top}\psi_0 > 0} \right\|^2\right]
\end{align*}
This immediately implies: 
$$
\bbM(\th) - \bbM(\th_0) \ge c_- \bbE\left[\left\|\beta - \beta_0 + \delta\mathds{1}_{Q^{\top}\psi > 0}- \delta_0\mathds{1}_{Q^{\top}\psi_0 > 0} \right\|^2\right] \,.
$$

\noindent
For notational simplicity, define $p_{\psi} = \bbP(Q^{\top}\psi > 0)$. Expanding the RHS we have: 
\begin{align}
    & \bbE\left[\left\|\beta - \beta_0 + \delta\mathds{1}_{Q^{\top}\psi > 0}- \delta_0\mathds{1}_{Q^{\top}\psi_0 > 0} \right\|^2\right] \notag \\
    & = \|\beta - \beta_0\|^2 +  2(\beta - \beta_0)^{\top}\bbE\left[\delta\mathds{1}_{Q^{\top}\psi > 0}- \delta_0\mathds{1}_{Q^{\top}\psi_0 > 0}\right] + \bbE\left[\left\|\delta\mathds{1}_{Q^{\top}\psi > 0}- \delta_0\mathds{1}_{Q^{\top}\psi_0 > 0}\right\|^2\right] \notag  \\
     & = \|\beta - \beta_0\|^2 +  2(\beta - \beta_0)^{\top}\bbE\left[\delta\mathds{1}_{Q^{\top}\psi > 0}-\delta\mathds{1}_{Q^{\top}\psi_0 > 0} + \delta\mathds{1}_{Q^{\top}\psi_0 > 0} - \delta_0\mathds{1}_{Q^{\top}\psi_0 > 0}\right] \notag  \\
     & \qquad \qquad \qquad \qquad \qquad+ \bbE\left[\left\|\delta\mathds{1}_{Q^{\top}\psi > 0}-\delta\mathds{1}_{Q^{\top}\psi_0 > 0} + \delta\mathds{1}_{Q^{\top}\psi_0 > 0} - \delta_0\mathds{1}_{Q^{\top}\psi_0 > 0}\right\|^2\right] \notag  \\
    & = \|\beta - \beta_0\|^2 + 2(\beta - \beta_0)^{\top}(\delta - \delta_0)p_{\psi_0} +  \|\delta - \delta_0\|^2 p_{\psi_0} \notag  \\
    & \qquad \qquad \qquad +  2(\beta - \beta_0)^{\top}\delta\left(p_{\psi} - p_{\psi_0}\right) + \|\delta\|^2 \bbP\left(\s(Q^{\top}\psi) \neq \s(Q^{\top}\psi_0)\right)  \notag  \\
    \label{eq:nsb1} & \qquad \qquad \qquad \qquad \qquad - 2\delta^{\top}(\delta - \delta_0)\bbP\left(Q^{\top}\psi_0 > 0, Q^{\top}\psi < 0\right) 
    \end{align}
Using the fact that $2ab \ge (a^2/c) + cb^2$ for any constant $c$ we have: 
\begin{align*}
& \|\beta - \beta_0\|^2 + 2(\beta - \beta_0)^{\top}(\delta - \delta_0)p_{\psi_0} +  \|\delta - \delta_0\|^2 p_{\psi_0} \\
& \ge \|\beta - \beta_0\|^2  + \|\delta - \delta_0\|^2 p_{\psi_0}  -  \frac{\|\beta - \beta_0\|^2 p_{\psi_0}}{c} - c \|\delta - \delta_0\|^2 p_{\psi_0} \\
& = \|\beta - \beta_0\|^2\left(1 - \frac{p_{\psi_0}}{c}\right) +  \|\delta - \delta_0\|^2  p_{\psi_0} (1 - c) \,.
\end{align*}
for any $c$. To make the RHS non-negative we pick $p_{\psi_0} < c < 1$ and concludes that: 
\begin{equation}
\label{eq:nsb2}
 \|\beta - \beta_0\|^2 + 2(\beta - \beta_0)^{\top}(\delta - \delta_0)p_{\psi_0} +  \|\delta - \delta_0\|^2 p_{\psi_0} \gtrsim \left( \|\beta - \beta_0\|^2 + \|\delta - \delta_0\|^2\right) \,.
\end{equation}
For the last 3 summands of RHS of equation \eqref{eq:nsb1}:     
\begin{align}
& 2(\beta - \beta_0)^{\top}\delta\left(p_{\psi} - p_{\psi_0}\right) + \|\delta\|^2 \bbP\left(\s(Q^{\top}\psi) \neq \s(Q^{\top}\psi_0)\right)  \notag  \\
 & \qquad \qquad - 2\delta^{\top}(\delta - \delta_0)\bbP\left(Q^{\top}\psi_0 > 0, Q^{\top}\psi < 0\right) \notag \\
 & = 2(\beta - \beta_0)^{\top}\delta \bbP\left(Q^{\top}\psi > 0, Q^{\top}\psi_0 < 0\right) - 2(\beta - \beta_0)^{\top}\delta \bbP\left(Q^{\top}\psi < 0, Q^{\top}\psi_0 > 0\right) \notag \\
  & \qquad \qquad + |\delta\|^2 \bbP\left(\s(Q^{\top}\psi) \neq \s(Q^{\top}\psi_0)\right) - 2\delta^{\top}(\delta - \delta_0)\bbP\left(Q^{\top}\psi_0 > 0, Q^{\top}\psi < 0\right) \notag \\
  & = \left[\|\delta\|^2 - 2(\beta - \beta_0)^{\top}\delta - 2\delta^{\top}(\delta - \delta_0)\right]\bbP\left(Q^{\top}\psi_0 > 0, Q^{\top}\psi < 0\right) \notag \\
  & \qquad \qquad \qquad \qquad \qquad \qquad + \left[\|\delta\|^2 + 2(\beta - \beta_0)^{\top}\delta\right]\bbP\left(Q^{\top}\psi > 0, Q^{\top}\psi_0 < 0\right) \notag \\
    & = \left[\|\delta_0\|^2 - 2(\beta - \beta_0)^{\top}(\delta - \delta_0) -  2(\beta - \beta_0)^{\top}\delta_0 - \|\delta - \delta_0\|^2\right]\bbP\left(Q^{\top}\psi_0 > 0, Q^{\top}\psi < 0\right) \notag \\
  & \qquad   + \left[\|\delta_0\|^2 + \|\delta - \delta_0\|^2 + 2(\delta - \delta_0)^{\top}\delta_0 + 2(\beta - \beta_0)^{\top}(\delta - \delta_0) +  2(\beta - \beta_0)^{\top}\delta_0\right]\bbP\left(Q^{\top}\psi > 0, Q^{\top}\psi_0 < 0\right) \notag \\
  & \ge  \left[\|\delta_0\|^2 - 2\|\beta - \beta_0\|\|\delta - \delta_0\| -  2\|\beta - \beta_0\|\|\delta_0\| - \|\delta - \delta_0\|^2\right]\bbP\left(Q^{\top}\psi_0 > 0, Q^{\top}\psi < 0\right) \notag \\
  & \qquad   + \left[\|\delta_0\|^2 + \|\delta - \delta_0\|^2 + 2\|\delta - \delta_0\|\|\delta_0\| + 2\|\beta - \beta_0\|\|\delta - \delta_0\| +  2\|\beta - \beta_0\|\|\delta_0\|\right]\bbP\left(Q^{\top}\psi > 0, Q^{\top}\psi_0 < 0\right) \notag \\
\label{eq:nsb3} & \gtrsim \|\delta_0\|^2 \bbP\left(\s(Q^{\top}\psi) \neq \s(Q^{\top}\psi_0)\right)   \gtrsim \|\psi - \psi_0\| \hspace{0.2in} [\text{By Assumption }\ref{eq:assm}]\,.
\end{align}
Combining equation \eqref{eq:nsb2} and \eqref{eq:nsb3} we complete the proof of lower bound. The upper bound is relatively easier: note that by our previous calculation: 
\begin{align*}
    \bbM(\theta) - \bbM(\th_0) & =  \bbE\left[\left( X^{\top}\left(\beta + \delta\mathds{1}_{Q^{\top}\psi > 0} - \beta_0 - \delta_0\mathds{1}_{Q^{\top}\psi_0 > 0}\right)\right)^2\right] \\
    & \le  c_+\bbE\left[\left\|\beta - \beta_0 + \delta\mathds{1}_{Q^{\top}\psi > 0}- \delta_0\mathds{1}_{Q^{\top}\psi_0 > 0} \right\|^2\right] \\
    & = c_+\bbE\left[\left\|\beta - \beta_0 + \delta\mathds{1}_{Q^{\top}\psi > 0}- \delta\mathds{1}_{Q^{\top}\psi_0 > 0} + \delta\mathds{1}_{Q^{\top}\psi_0 > 0} - \delta_0\mathds{1}_{Q^{\top}\psi_0 > 0} \right\|^2\right] \\
    & \lesssim \left[\|\beta - \beta_0\|^2 + \|\delta - \delta_0\|^2 + \bbP\left(\s(Q^{\top}\psi) \neq \s(Q^{\top}\psi_0)\right)\right] \\
    & \lesssim \left[\|\beta - \beta_0\|^2 + \|\delta - \delta_0\|^2 + \|\psi - \psi_0\|\right] \,.
\end{align*}
This completes the entire proof.
\end{proof}

\subsection{Proof of Lemma \ref{lem:uniform_smooth}}
\begin{proof}
The difference of the two losses: 
\begin{align*}
\left|\bbM^s(\theta) - \bbM(\theta)\right| & = \left|\bbE\left[\left\{-2\left(Y_i - X^{\top}\beta\right)X^{\top}\delta + (X^{\top}\delta)^2\right\}\left(K\left(\frac{Q^{\top}\psi}{\sigma_n}\right) - \mathds{1}_{Q^{\top}\psi > 0}\right)\right]\right| \\
& \le \bbE\left[\left|-2\left(Y_i - X^{\top}\beta\right)X^{\top}\delta + (X^{\top}\delta)^2\right|\left|K\left(\frac{Q^{\top}\psi}{\sigma_n}\right) - \mathds{1}_{Q^{\top}\psi > 0}\right|\right] \\
& := \bbE\left[m(Q)\left|K\left(\frac{Q^{\top}\psi}{\sigma_n}\right) - \mathds{1}_{Q^{\top}\psi > 0}\right|\right] 
\end{align*}
where $m(Q) = \bbE\left[\left|-2\left(Y_i - X^{\top}\beta\right)X^{\top}\delta + (X^{\top}\delta)^2\right| \mid Q\right]$. This function can be bounded as follows: 
\begin{align*}
m(Q) & = \bbE\left[\left|-2\left(Y_i - X^{\top}\beta\right)X^{\top}\delta + (X^{\top}\delta)^2\right| \mid Q\right] \\
& \le \bbE[ (X^{\top}\delta)^2 \mid Q] + 2\bbE\left[\left|(\beta - \beta_0)^{\top}XX^{\top}\delta\right|\right] + 2\bbE\left[\left|\delta_0^{\top}XX^{\top}\delta\right|\right] \\
& \le c_+\left(\|\delta\|^2 + 2\|\beta - \beta_0\|\|\delta\| + 2\|\delta\|\|\delta_0\|\right) \lesssim 1 \,,
\end{align*}
as our parameter space is compact. For the rest of the calculation define $\eta = (\tilde \psi - \tilde \psi_0)/\sigma_n$. The definition of $\eta$ may be changed from proof to proof, but it will be clear from the context. Therefore we have: 
\begin{align*}
\left|\bbM^s(\theta) - \bbM(\theta)\right| & \lesssim \bbE\left[\left|K\left(\frac{Q^{\top}\psi}{\sigma_n}\right) - \mathds{1}_{Q^{\top}\psi > 0}\right|\right]  \\
& =  \bbE\left[\left| \mathds{1}\left(\frac{Q^{\top}\psi_0}{\sigma_n} + \eta^{\top}\tilde{Q} \ge 0\right) - K\left(\frac{Q^{\top}\psi_0}{\sigma_n} + \eta^{\top}\tilde{Q}\right)\right|\right]  \\
& = \sigma_n \int_{\mathbb{R}^{p-1}} \int_{-\infty}^{\infty}  \left | \mathds{1}\left(t  \ge 0\right) - K\left(t \right)\right | f_0(\sigma_n (t-\eta^{\top}\tilde{q}) | \tilde{q}) \ dt \ dP(\tilde{q}) \\
& \le f_+  \sigma_n  \int_{-\infty}^{\infty}  \left | \mathds{1}\left(t  \ge 0\right) - K\left(t \right)\right | \ dt \lesssim \sigma_n \,.
\end{align*}
where the integral over $t$ is finite follows from the definition of the kernel. This completes the proof. 
\end{proof}

\subsection{Proof of Lemma \ref{lem:pop_smooth_curvarture}}
\begin{proof}
First note that we can write: 
\begin{align}
    & \bbM^s(\th) - \bbM^s(\th_0^s) \notag \\
    & = \underbrace{\bbM^s(\th) - \bbM(\th)}_{\ge -K_1\sigma_n} + \underbrace{\bbM(\th) - \bbM(\th_0)}_{\underbrace{\ge u_- d^2(\theta, \theta_0)}_{\ge \frac{u_-}{2} d^2(\th, \th_0^s) - u_-\sigma_n }} + \underbrace{\bbM(\th_0) - \bbM(\th_0^s)}_{\ge - u_+ d^2(\theta_0, \th_0^s) \ge - u_+\sigma_n} + \underbrace{\bbM(\th_0^s) - \bbM^s(\th_0^s)}_{\ge - K_1 \sigma_n} \notag \\
    & \ge \frac{u_-}{2}d^2(\th, \th_0^s) - (2K_1 + \xi)\sigma_n \notag \\
    & \ge \frac{u_-}{2}\left[\|\beta - \beta^s_0\|^2 + \|\delta - \delta^s_0\|^2 + \|\psi - \psi^s_0\|\right] - (2K_1 + \xi)\sigma_n \notag \\ 
    & \ge \left[\frac{u_-}{2}\left(\|\beta - \beta^s_0\|^2 + \|\delta - \delta^s_0\|^2\right) + \frac{u_-}{4}\|\psi - \psi^s_0\|\right]\mathds{1}_{\|\psi - \psi^s_0\| > \frac{4(2K_1 + \xi)}{u_-}\sigma_n} \notag \\
    \label{eq:lower_curv_smooth} & \gtrsim  \left[\|\beta - \beta^s_0\|^2 + \|\delta - \delta^s_0\|^2 + \|\psi - \psi^s_0\|\right]\mathds{1}_{\|\psi - \psi^s_0\| > \frac{4(2K_1 + \xi)}{u_-}\sigma_n}
\end{align}
where $\xi$ can be taken as close to $0$ as possible. Henceforth we set $\cK = 4(2K_1 + \xi)/u_-$. For the other part of the curvature (i.e. when $\|\psi - \psi_0^s\| \le \cK \sigma_n$) we start with a two step Taylor expansion of the smoothed loss function: 
\begin{align*}
    \bbM^s(\th) - \bbM^s(\th_0^s) = \frac12 (\th_0 - \th^0_s)^{\top}\nabla^2 \bbM^s(\th^*)(\th_0 - \th^0_s) 
\end{align*}
Recall the definition of $\bbM^s(\th)$: 
$$
\bbM^s_n(\theta) = \bbE\left(Y - X^{\top}\beta\right)^2 + \bbE \left\{\left[-2\left(Y_i - X_i^{\top}\beta\right)X_i^{\top}\delta + (X_i^{\top}\delta)^2\right] K\left(\frac{Q_i^{\top}\psi}{\sigma_n}\right)\right\}
$$
%
The partial derivates of $\bbM^s(\th)$ with respect to $(\beta, \delta, \psi)$ was derived in equation \eqref{eq:beta_grad} - \eqref{eq:psi_grad}. From there, we calculate the hessian of $\bbM^s(\th)$: 
\begin{align*}
    \nabla_{\beta\beta}\bbM^s(\th) & = 2\Sigma_X \\
    \nabla_{\delta\delta}\bbM^s(\th) & = 2 \bbE\left[g(Q)K\left(\frac{Q_i^{\top}\psi}{\sigma_n}\right)\right] = 2 \bbE\left[g(Q)K\left(\frac{Q_i^{\top}\psi_0 }{\sigma_n} + \tilde Q^{\top}\tilde \eta\right)\right] \\
    \nabla_{\psi\psi} \bbM^s(\th) & = \frac{1}{\sigma_n^2}\bbE \left\{\left[-2\left(Y_i - X_i^{\top}\beta\right)X_i^{\top}\delta + (X_i^{\top}\delta)^2\right]\tilde Q_i\tilde Q_i^{\top} K''\left(\frac{Q_i^{\top}\psi_0 }{\sigma_n} + \tilde Q^{\top}\tilde \eta\right)\right\} \\
    \nabla_{\beta \delta}\bbM^s(\th) & = 2 \bbE\left[g(Q)K\left(\frac{Q_i^{\top}\psi}{\sigma_n}\right)\right] = 2 \bbE\left[g(Q)K\left(\frac{Q_i^{\top}\psi_0 }{\sigma_n} + \tilde Q^{\top}\tilde \eta\right)\right] \\
    \nabla_{\beta \psi}\bbM^s(\th) & = \frac{2}{\sigma_n}\bbE\left(g(Q)\delta\tilde Q^{\top}K'\left(\frac{Q_i^{\top}\psi_0 }{\sigma_n} + \tilde Q^{\top}\tilde \eta\right)\right)  \\
    \nabla_{\delta \psi} \bbM^s(\th) & = \frac{2}{\sigma_n}\bbE \left\{\left[-X_i\left(Y_i - X_i^{\top}\beta\right) + X_iX_i^{\top}\delta\right]\tilde Q_i^{\top} K'\left(\frac{Q_i^{\top}\psi_0 }{\sigma_n} + \tilde Q^{\top}\tilde \eta\right)\right\} \,.
\end{align*}
where we use $\tilde \eta$ for a generic notation for $(\tilde \psi - \tilde \psi_0)/\sigma_n$. For notational simplicity, we define $\gamma = (\beta, \delta)$ and $\nabla^2\bbM^{s, \gamma}(\th)$, $\nabla^2\bbM^{s, \gamma \psi}(\th), \nabla^2\bbM^{s, \psi \psi}(\th)$ to be corresponding blocks of the hessian matrix. We have: 
\begin{align}
      \bbM^s(\th) - \bbM^s(\th_0^s) & = \frac12 (\th - \th^0_s)^{\top}\nabla^2 \bbM^s(\th^*)(\th - \th^0_s)  \notag \\
      & = \frac12 (\gamma - \gamma_0^s)^{\top}\nabla^2 \bbM^{s, \gamma}(\th^*)(\gamma - \gamma^0_s) + (\gamma - \gamma_0^s)^{\top}\nabla^2 \bbM^{s, \gamma \psi}(\th^*)(\psi - \psi^0_s) \notag \\
      & \qquad \qquad \qquad \qquad + \frac12(\psi - \psi_0^s)^{\top}\nabla^2 \bbM^{s, \psi \psi}(\th^*)(\psi - \psi^0_s)  \notag \\
      \label{eq:hessian_1} & := \frac12 \left(T_1 + 2T_2 + T_3\right)
\end{align}
Note that we can write: 
\begin{align*}
    T_1 & = (\gamma - \gamma_0^s)^{\top}\nabla^2 \bbM^{s, \gamma}(\tilde \th)(\gamma - \gamma^0_s) \\
    & = (\gamma - \gamma_0^s)^{\top}\nabla^2 \bbM^{s, \gamma}(\th_0)(\gamma - \gamma^0_s)  + (\gamma - \gamma_0^s)^{\top}\left[\nabla^2 \bbM^{s, \gamma}(\tilde \th) - \nabla^2 \bbM^{s, \gamma}(\th_0)\right](\gamma - \gamma^0_s) 
\end{align*}
The operator norm of the difference of two hessians can be bounded as: 
$$
\left\|\nabla^2 \bbM^{s, \gamma}(\th^*) - \nabla^2 \bbM^{s, \gamma}(\th_0)\right\|_{op}  = O(\sigma_n) \,.
$$
for any $\th^*$ in a neighborhood of $\theta_0^s$ with $\|\psi - \psi_0^s\| \le \cK \sigma_n$. To prove this note that for any $\theta$: 
$$
\nabla^2 \bbM^{s, \gamma}(\th^*) - \nabla^2 \bbM^{s, \gamma}(\th_0) = 2\begin{pmatrix}0 & A \\
A & A\end{pmatrix} = \begin{pmatrix}0 & 1 \\ 1 & 1\end{pmatrix} \otimes A 
$$
where: 
$$
A = \bbE\left[g(Q)K\left(\frac{Q_i^{\top}\psi}{\sigma_n}\right)\right] - \bbE\left[g(Q)K\left(\frac{Q_i^{\top}\psi_0 }{\sigma_n}\right)\right]
$$
Therefore it is enough to show $\|A\|_{op} = O(\sigma_n)$. Towards that direction: 
\begin{align*}
A & = \bbE\left[g(Q)K\left(\frac{Q_i^{\top}\psi}{\sigma_n}\right)\right] - \bbE\left[g(Q)K\left(\frac{Q_i^{\top}\psi_0 }{\sigma_n}\right)\right] \\
& = \sigma_n \int \int g(\sigma_n t - \tilde q^{\top}\tilde \psi_0)\left(K(t + \tilde q^{\top}\eta) - K(t) \right) f_0(\sigma_n t \mid \tilde q) \ f(\tilde q) \ dt  \ d\tilde q \\
& = \sigma_n \left[\int \int g(- \tilde q^{\top}\tilde \psi_0)\left(K(t + \tilde q^{\top}\eta) - K(t) \right) f_0(0 \mid \tilde q) \ f(\tilde q) \ dt  \ d\tilde q + R \right] \\
& = \sigma_n \left[\int \int g(- \tilde q^{\top}\tilde \psi_0)f_0(0 \mid \tilde q) \int_t^{t + \tilde q^{\top}\eta}K'(s) \ ds \ f(\tilde q) \ dt  \ d\tilde q + R \right] \\
& = \sigma_n \left[\int g(- \tilde q^{\top}\tilde \psi_0)f_0(0 \mid \tilde q) \int_{-\infty}^{\infty}K'(s) \int_{s-\tilde q^{\top}\eta}^s \ dt \ ds \ f(\tilde q)\ d\tilde q + R \right] \\
& = \sigma_n \left[\int g(- \tilde q^{\top}\tilde \psi_0)f_0(0 \mid \tilde q)\tilde q^{\top}\eta  \ f(\tilde q)\ d\tilde q + R \right] \\
& = \sigma_n \left[\bbE\left[g(- \tilde Q^{\top}\tilde \psi_0, \tilde Q)f_0(0 \mid \tilde Q)\tilde Q^{\top}\eta\right] + R \right]
\end{align*}
using the fact that $\left\|\bbE\left[g(- \tilde Q^{\top}\tilde \psi_0, \tilde Q)f_0(0 \mid \tilde Q)\tilde Q^{\top}\eta\right]\right\|_{op} = O(1)$ and $\|R\|_{op} = O(\sigma_n)$ we conclude the claim. From the above claim we conclude: 
\begin{equation}
    \label{eq:hessian_gamma}
    T_1 = (\gamma - \gamma_0^s)^{\top}\nabla^2 \bbM^{s, \gamma}(\th^*)(\gamma - \gamma^s_0) \ge \|\gamma - \gamma^s_0\|^2(1 - O(\sigma_n)) \ge \frac12 \|\gamma - \gamma_0^s\|^2
\end{equation}
for all large $n$. 
\\\\
\noindent 
We next deal with the cross term $T_2$ in equation \eqref{eq:hessian_1}. Towards that end first note that: 
\begin{align*}
    & \frac{1}{\sigma_n}\bbE\left((g(Q)\delta)\tilde Q^{\top}K'\left(\frac{Q_i^{\top}\psi_0 }{\sigma_n} + \tilde Q^{\top}\eta^*\right)\right) \\
    & = \int_{\bbR^{(p-1)}}\left[ \int_{-\infty}^{\infty} \left(g\left(\sigma_nt - \tilde q^{\top}\tilde \psi_0, \tilde q\right)\delta\right) K'\left(t + \tilde q^{\top}\eta^*\right) f_0(\sigma_n t \mid \tilde q) \ dt\right] \tilde q^{\top} \ f(\tilde q) \ d\tilde q  \\
     & = \int_{\bbR^{(p-1)}}\left[ \int_{-\infty}^{\infty} \left(g\left(- \tilde q^{\top}\tilde \psi_0, \tilde q\right)\delta\right) K'\left(t + \tilde q^{\top}\eta^*\right) f_0(0 \mid \tilde q) \ dt\right] \tilde q^{\top} \ f(\tilde q) \ d\tilde q  + R_1\\
    & = \bbE\left[\left(g\left( - \tilde Q^{\top}\tilde \psi_0, \tilde Q\right)\delta\right)\tilde Q^{\top}f_0(0 \mid \tilde Q)\right] + R_1
\end{align*}
where the remainder term $R_1$ can be further decomposed $R_1 = R_{11} + R_{12} + R_{13}$ with: 
\begin{align*}
    \left\|R_{11}\right\|  & = \left\|\int_{\bbR^{(p-1)}}\left[ \int_{-\infty}^{\infty} \left(g\left(- \tilde q^{\top}\tilde \psi_0, \tilde q\right)\delta\right) K'\left(t + \tilde q^{\top}\eta^*\right) (f_0(\sigma_nt\mid \tilde q) -  f_0(0 \mid \tilde q)) \ dt\right] \tilde q^{\top} \ f(\tilde q) \ d\tilde q\right\| \\
    & \le \left\|\int_{\bbR^{(p-1)}}\left[ \int_{-\infty}^{\infty} \left\|g\left(- \tilde q^{\top}\tilde \psi_0, \tilde q\right)\right\|_{op}\|\delta\| \left|K'\left(t + \tilde q^{\top}\eta^*\right)\right| \left|f_0(\sigma_nt\mid \tilde q) -  f_0(0 \mid \tilde q)\right| \ dt\right] \left|\tilde q\right| \ f(\tilde q) \ d\tilde q\right\| \\
    & \le \sigma_n  \dot{f}^+  c_+ \|\delta\| \int_{\bbR^{(p-1)}} \|\tilde q\|  \int_{-\infty}^{\infty} |t| \left|K'\left(t + \tilde q^{\top}\eta^*\right)\right| \ dt \ f(\tilde q) \ d\tilde q \\
    & \le \sigma_n  \dot{f}^+  c_+ \|\delta\| \int_{\bbR^{(p-1)}} \|\tilde q\|  \int_{-\infty}^{\infty} |t -  \tilde q^{\top}\eta^*| \left|K'\left(t\right)\right| \ dt \ f(\tilde q) \ d\tilde q \\
    & \le \sigma_n  \dot{f}^+  c_+ \|\delta\| \left[\int_{\bbR^{(p-1)}} \|\tilde q\|  \int_{-\infty}^{\infty} |t| \left|K'\left(t\right)\right| \ dt \ f(\tilde q) \ d\tilde q \right. \\
    & \qquad \qquad \qquad \left. + \int_{\bbR^{(p-1)}} \|\tilde q\|^2 \|\eta^*\|  \int_{-\infty}^{\infty} |K'(t)| \ dt \ f(\tilde q) \ d\tilde q\right] \\
     & \le \sigma_n  \dot{f}^+  c_+ \|\delta\| \left[\int_{\bbR^{(p-1)}} \|\tilde q\|  \int_{-\infty}^{\infty} |t| \left|K'\left(t\right)\right| \ dt \ f(\tilde q) \ d\tilde q + \cK\int_{\bbR^{(p-1)}} \|\tilde q\|^2 \int_{-\infty}^{\infty} |K'(t)| \ dt \ f(\tilde q) \ d\tilde q\right] \\
        & \lesssim \sigma_n \,.
\end{align*}
where the last bound follows from our assumptions using the fact that: 
\begin{align*}
    & \|R_{12}\| \\
    &= \left\|\int_{\bbR^{(p-1)}}\left[ \int_{-\infty}^{\infty} \left(\left(g\left(\sigma_n t- \tilde q^{\top}\tilde \psi_0, \tilde q\right) - g\left(- \tilde q^{\top}\tilde \psi_0, \tilde q\right)\right)\delta\right) K'\left(t + \tilde q^{\top} \eta^*\right)  f_0(0 \mid \tilde q) \ dt\right] \tilde q^{\top} \ f(\tilde q) \ d\tilde q\right\| \\
    & \le \int \|\tilde q\|\|\delta\|f_0(0 \mid \tilde q) \int_{-\infty}^{\infty} \left\|g\left(\sigma_n t- \tilde q^{\top}\tilde \psi_0, \tilde q\right) - g\left(- \tilde q^{\top}\tilde \psi_0, \tilde q\right) \right\|_{op}\left|K'\left(t + \tilde q^{\top} \eta^*\right)\right| \ dt \ f(\tilde q) \ d\tilde q \\
    & \le \dot{c}_+ \sigma_n \int \|\tilde q\|\|\delta\|f_0(0 \mid \tilde q)\dot \int_{-\infty}^{\infty} |t| \left|K'\left(t + \tilde q^{\top}\tilde \eta\right)\right| \ dt \ f(\tilde q) \ d\tilde q \hspace{0.2in} [\text{Assumption }\ref{eq:assm}]\\
    &  \lesssim \sigma_n \,.
\end{align*}
The other remainder term $R_{13}$ is the higher order term and can be shown to be $O(\sigma_n^2)$ using same techniques. This implies for all large $n$: 
\begin{align*}
    \left\|\nabla_{\beta \psi}\bbM^s(\th)\right\|_{op} & = O(1) \,.
\end{align*}
and similar calculation yields $ \left\|\nabla_{\delta \psi}\bbM^s(\th)\right\|_{op} = O(1)$. Using this we have: 
\begin{align}
    T_2 & = (\gamma - \gamma_0^s)^{\top}\nabla^2 \bbM^{s, \gamma \psi}(\tilde \th)(\psi - \psi^0_s) \notag \\
    & = (\beta - \beta_0^s)^{\top}\nabla_{\beta \psi}^2 \bbM^{s}(\tilde \th)(\psi - \psi^0_s) + (\delta - \delta_0^s)^{\top}\nabla_{\delta \psi}^2 \bbM^{s}(\tilde \th)(\psi - \psi^0_s) \notag \\
    & \ge - C\left[\|\beta - \beta_0^s\| + \|\delta - \delta_0^s\| \right]\|\psi - \psi^0_s\| \notag \\
    & \ge -C \sqrt{\sigma_n}\left[\|\beta - \beta_0^s\| + \|\delta - \delta_0^s\| \right]\frac{\|\psi - \psi^0_s\| }{\sqrt{\sigma_n}}  \notag \\
    \label{eq:hessian_cross} & \gtrsim - \sqrt{\sigma_n}\left(\|\beta - \beta_0^s\|^2 + \|\delta - \delta_0^s\|^2 +\frac{\|\psi - \psi^0_s\|^2 }{\sigma_n} \right)
\end{align}
Now for $T_3$ note that: 
\allowdisplaybreaks
\begin{align*}
& \sigma_n \nabla_{\psi\psi} \bbM^s_n(\th) \\
& = \frac{1}{\sigma_n}\bbE \left\{\left[-2\left(Y_i - X_i^{\top}\beta\right)X_i^{\top}\delta + (X_i^{\top}\delta)^2\right]\tilde Q_i\tilde Q_i^{\top} K''\left(\frac{Q_i^{\top}\psi_0 }{\sigma_n} + \tilde Q^{\top}\tilde \eta\right)\right\} \\
& = \frac{1}{\sigma_n}\bbE \left\{\left[-2\left(Y_i - X_i^{\top}\beta\right)X_i^{\top}\delta \right]\tilde Q_i\tilde Q_i^{\top} K''\left(\frac{Q_i^{\top}\psi_0 }{\sigma_n} + \tilde Q^{\top}\tilde \eta\right)\right\}  \\
& \qquad \qquad \qquad + \frac{1}{\sigma_n}\bbE \left\{(\delta^{\top}g(Q) \delta)\tilde Q_i\tilde Q_i^{\top} K''\left(\frac{Q_i^{\top}\psi_0 }{\sigma_n} + \tilde Q^{\top}\tilde \eta\right)\right\}  \\
& = \frac{1}{\sigma_n}\bbE \left\{\left[-2 X_i^{\top}\left(\beta_0 -\beta\right)X_i^{\top}\delta - 2(X_i^{\top}\delta_0)(X_i^{\top}\delta)\mathds{1}_{Q_i^{\top}\psi_0 > 0}\right]\tilde Q_i\tilde Q_i^{\top} K''\left(\frac{Q_i^{\top}\psi_0 }{\sigma_n} + \tilde Q^{\top}\tilde \eta\right)\right\}  \\
& \qquad \qquad \qquad + \frac{1}{\sigma_n}\bbE \left\{(\delta^{\top}g(Q) \delta)\tilde Q_i\tilde Q_i^{\top} K''\left(\frac{Q_i^{\top}\psi_0 }{\sigma_n} + \tilde Q^{\top}\tilde \eta\right)\right\}  \\
& = \frac{-2}{\sigma_n}\bbE \left\{((\beta_0 - \beta)^{\top}g(Q) \delta)\tilde Q_i\tilde Q_i^{\top} K''\left(\frac{Q_i^{\top}\psi_0 }{\sigma_n} + \tilde Q^{\top}\tilde \eta\right)\right\} \\
& \qquad \qquad \qquad + \frac{-2}{\sigma_n}\bbE \left\{(\delta_0^{\top}g(Q) \delta)\tilde Q_i\tilde Q_i^{\top} K''\left(\frac{Q_i^{\top}\psi_0 }{\sigma_n} + \tilde Q^{\top}\tilde \eta\right)\mathds{1}_{Q_i^{\top}\psi_0 > 0}\right\} \\
& \qquad \qquad \qquad \qquad \qquad \qquad + \frac{1}{\sigma_n}\bbE \left\{(\delta^{\top}g(Q) \delta)\tilde Q_i\tilde Q_i^{\top} K''\left(\frac{Q_i^{\top}\psi_0 }{\sigma_n} + \tilde Q^{\top}\tilde \eta\right)\right\}  \\
& = \underbrace{\frac{-2}{\sigma_n}\bbE \left\{((\beta_0 - \beta)^{\top}g(Q)\delta)\tilde Q_i\tilde Q_i^{\top} K''\left(\frac{Q_i^{\top}\psi_0 }{\sigma_n} + \tilde Q^{\top}\tilde \eta\right)\right\}}_{M_1} \\
& \qquad \qquad \qquad + \underbrace{\frac{-2}{\sigma_n}\bbE \left\{(\delta_0^{\top}g(Q) \delta_0)\tilde Q_i\tilde Q_i^{\top} K''\left(\frac{Q_i^{\top}\psi_0 }{\sigma_n} + \tilde Q^{\top}\tilde \eta\right)\mathds{1}_{Q_i^{\top}\psi_0 > 0}\right\}}_{M_2} \\
& \qquad \qquad \qquad \qquad \qquad \qquad +
\underbrace{\frac{-2}{\sigma_n}\bbE \left\{(\delta_0^{\top} g(Q) (\delta - \delta_0))\tilde Q_i\tilde Q_i^{\top} K''\left(\frac{Q_i^{\top}\psi_0 }{\sigma_n} + \tilde Q^{\top}\tilde \eta\right)\mathds{1}_{Q_i^{\top}\psi_0 > 0}\right\}}_{M_3} \\
& \qquad \qquad \qquad \qquad \qquad \qquad \qquad \qquad \qquad + \underbrace{\frac{1}{\sigma_n}\bbE \left\{(\delta^{\top}g(Q) \delta)\tilde Q_i\tilde Q_i^{\top} K''\left(\frac{Q_i^{\top}\psi_0 }{\sigma_n} + \tilde Q^{\top}\tilde \eta\right)\right\}}_{M_4}  \\
& := M_1 + M_2 + M_3 + M_4
\end{align*}
We next show that $M_1$ and $M_4$ are $O(\sigma_n)$. Towards that end note that for any two vectors $v_1, v_2$: 
\begin{align*}
    & \frac{1}{\sigma_n}\bbE \left\{(v_1^{\top}g(Q)v_2)\tilde Q_i\tilde Q_i^{\top} K''\left(\frac{Q_i^{\top}\psi_0 }{\sigma_n} + \tilde Q^{\top}\tilde \eta\right)\right\} \\
    & = \int \tilde q \tilde q^{\top} \int_{-\infty}^{\infty}(v_1^{\top}g(\sigma_nt - \tilde q^{\top}\tilde \eta, \tilde q)v_2) K''(t + \tilde q^{\top}\tilde \eta) f(\sigma_nt \mid \tilde q) \ dt \ f(\tilde q) \ d\tilde q \\
    & = \int \tilde q \tilde q^{\top} (v_1^{\top}g( - \tilde q^{\top}\tilde \eta, \tilde q)v_2)f(0 \mid \tilde q) f(\tilde q) \ d\tilde q \cancelto{0}{\int_{-\infty}^{\infty} K''(t) \ dt}  + R = R
\end{align*}
as $\int K''(t) \ dt = 0$ follows from our choice of kernel $K(x) = \Phi(x)$. Similar calculation as in the case of analyzing the remainder of $T_2$ yields $\|R\|_{op} = O(\sigma_n)$.
\noindent
This immediately implies $\|M_1\|_{op} = O(\sigma_n)$ and $\|M_4\|_{op} = O(\sigma_n)$. Now for $M_2$: 
\begin{align}
M_2 & = \frac{-2}{\sigma_n}\bbE \left\{(\delta_0^{\top}g(Q) \delta_0)\tilde Q_i\tilde Q_i^{\top} K''\left(\frac{Q_i^{\top}\psi_0 }{\sigma_n} + \tilde Q^{\top}\tilde \eta\right)\mathds{1}_{Q_i^{\top}\psi_0 > 0}\right\} \notag \\
& = -2\int \int (\delta_0^{\top}g(\sigma_n t - \tilde q^{\top}\tilde \psi_0) \delta_0)\tilde q\tilde q^{\top} K''\left(t + \tilde q^{\top}\eta^*\right)\mathds{1}_{t > 0} f_0(\sigma_n t \mid \tilde q) \ dt f(\tilde q) \ d\tilde q \notag \\
& = -2\int (\delta_0^{\top}g(- \tilde q^{\top}\tilde \psi_0) \delta_0)\tilde q\tilde q^{\top}  f_0(0 \mid \tilde q)  \int_{-\infty}^{\infty} K''\left(t + \tilde q^{\top}\eta^*\right)\mathds{1}_{t > 0} \ dt f(\tilde q) \ d\tilde q  + R \notag \\
\label{eq:M_2_double_deriv} & = 2\bbE\left[(\delta_0^{\top}g(- \tilde Q^{\top}\tilde \psi_0) \delta_0)\tilde 
Q\tilde Q^{\top}  f_0(0 \mid \tilde Q) K'(\tilde Q^{\top}\eta^*)\right] + R 
\end{align}
where the remainder term R is $O_p(\sigma_n)$ can be established as follows: 
\begin{align*}
R & = -2\left[\int \int (\delta_0^{\top}g(\sigma_n t - \tilde q^{\top}\tilde \psi_0) \delta_0)\tilde q\tilde q^{\top} K''\left(t + \tilde q^{\top}\eta^*\right)\mathds{1}_{t > 0} f_0(\sigma_n t \mid \tilde q) \ dt f(\tilde q) \ d\tilde q \right. \\
& \qquad \qquad - \left. \int (\delta_0^{\top}g(- \tilde q^{\top}\tilde \psi_0) \delta_0)\tilde q\tilde q^{\top}  f_0(0 \mid \tilde q)  \int_{-\infty}^{\infty} K''\left(t + \tilde q^{\top}\eta^*\right)\mathds{1}_{t > 0} \ dt f(\tilde q) \ d\tilde q \right] \\
& = -2\left\{\left[\int \int (\delta_0^{\top}g(\sigma_n t - \tilde q^{\top}\tilde \psi_0) \delta_0)\tilde q\tilde q^{\top} K''\left(t + \tilde q^{\top}\eta^*\right)\mathds{1}_{t > 0} f_0(\sigma_n t \mid \tilde q) \ dt f(\tilde q) \ d\tilde q \right. \right. \\
& \qquad \qquad - \left. \left. \int (\delta_0^{\top}g(\sigma_n t - \tilde q^{\top}\tilde \psi_0) \delta_0) \tilde q\tilde q^{\top}  f_0(0 \mid \tilde q)  \int_{-\infty}^{\infty} K''\left(t + \tilde q^{\top}\eta^*\right)\mathds{1}_{t > 0} \ dt f(\tilde q) \ d\tilde q \right] \right. \\
& \left. +  \left[\int (\delta_0^{\top}g(\sigma_n t - \tilde q^{\top}\tilde \psi_0) \delta_0)\tilde q\tilde q^{\top}  f_0(0 \mid \tilde q)  \int_{-\infty}^{\infty} K''\left(t + \tilde q^{\top}\eta^*\right)\mathds{1}_{t > 0} \ dt f(\tilde q) \ d\tilde q \right. \right. \\
& \qquad \qquad \left. \left. -\int (\delta_0^{\top}g(- \tilde q^{\top}\tilde \psi_0) \delta_0)\tilde q\tilde q^{\top}  f_0(0 \mid \tilde q)  \int_{-\infty}^{\infty} K''\left(t + \tilde q^{\top}\eta^*\right)\mathds{1}_{t > 0} \ dt f(\tilde q) \ d\tilde q \right]\right\} \\
& = -2(R_1 + R_2) \,.
\end{align*}
For $R_1$: 
\begin{align*}
\left\|R_1\right\|_{op} & =  \left\|\left[\int \int (\delta_0^{\top}g(\sigma_n t - \tilde q^{\top}\tilde \psi_0) \delta_0)\tilde q\tilde q^{\top} K''\left(t + \tilde q^{\top}\eta^*\right)\mathds{1}_{t > 0} f_0(\sigma_n t \mid \tilde q) \ dt f(\tilde q) \ d\tilde q \right. \right. \,.\\
& \qquad \qquad - \left. \left. \int (\delta_0^{\top}g(\sigma_n t - \tilde q^{\top}\tilde \psi_0) \delta_0)\tilde q\tilde q^{\top}  f_0(0 \mid \tilde q)  \int_{-\infty}^{\infty} K''\left(t + \tilde q^{\top}\eta^*\right)\mathds{1}_{t > 0} \ dt f(\tilde q) \ d\tilde q \right] \right\|_{op} \\
& \le c_+ \int \int \|\tilde q\|^2 |K''\left(t + \tilde q^{\top}\eta^*\right)| |f_0(\sigma_n t \mid \tilde q) -f_0(0\mid \tilde q)| \ dt \ f(\tilde q) \ d\tilde q  \\
& \le c_+ F_+\sigma_n  \int \|\tilde q\|^2 \int |t| |K''\left(t + \tilde q^{\top}\eta^*\right)| \ dt \ f(\tilde q) \ d\tilde q \\
& = c_+ F_+\sigma_n  \int \|\tilde q \|^2 \int |t -  \tilde q^{\top}\eta^*| |K''\left(t\right)| \ dt \ f(\tilde q) \ d\tilde q \\
& \le c_+ F_+ \sigma_n \left[\bbE[\|\tilde Q\|^2]\int |t||K''(t)| \ dt + \|\eta^*\|\bbE[\|\tilde Q\|^3]\int |K''(t)| \ dt\right] = O(\sigma_n) \,.
\end{align*}
and similarly for $R_2$: 
\begin{align*}
\|R_2\|_{op} & = \left\|\left[\int (\delta_0^{\top}g(\sigma_n t - \tilde q^{\top}\tilde \psi_0) \delta_0)\tilde q\tilde q^{\top}  f_0(0 \mid \tilde q)  \int_{-\infty}^{\infty} K''\left(t + \tilde q^{\top}\eta^*\right)\mathds{1}_{t > 0} \ dt f(\tilde q) \ d\tilde q \right. \right. \\
& \qquad \qquad \left. \left. -\int (\delta_0^{\top}g(- \tilde q^{\top}\tilde \psi_0) \delta_0)\tilde q\tilde q^{\top}  f_0(0 \mid \tilde q)  \int_{-\infty}^{\infty} K''\left(t + \tilde q^{\top}\eta^*\right)\mathds{1}_{t > 0} \ dt f(\tilde q) \ d\tilde q \right]\right\|_{op} \\
& \le F_+ \|\delta_0\|^2 \int \left\|g(\sigma_n t - \tilde q^{\top}\tilde \psi_0)  - g( - \tilde q^{\top}\tilde \psi_0) \right\|_{op} \|\tilde q\|^2 \int_{-\infty}^{\infty} |K''\left(t + \tilde q^{\top}\eta^*\right)| \ dt \\
& \le G_+ F_+ \sigma_n \int \|\tilde q\|^2 \int_{-\infty}^{\infty} |t||K''\left(t + \tilde q^{\top}\eta^*\right)| \ dt = O(\sigma_n) \,.
\end{align*}
Therefore from \eqref{eq:M_2_double_deriv} we conclude: 
\begin{equation}
M_2 = 2\bbE\left[(\delta_0^{\top}g(- \tilde Q^{\top}\tilde \psi_0) \delta_0)\tilde 
Q\tilde Q^{\top}  f_0(0 \mid \tilde Q) K'(\tilde Q^{\top}\eta^*)\right] + O(\sigma_n) \,.
\end{equation}
Similar calculation for $M_3$ yields: 
\begin{equation*}
M_3 = 2\bbE\left[(\delta_0^{\top}g(- \tilde Q^{\top}\tilde \psi_0)(\delta - \delta_0))\tilde 
Q\tilde Q^{\top}  f_0(0 \mid \tilde Q) K'(\tilde Q^{\top}\eta^*)\right] + O(\sigma_n) \,.
\end{equation*}
i.e. 
\begin{equation}
\|M_3\|_{op} \le c_+ \bbE\left[\|\tilde Q\|^2f_0(0 \mid \tilde Q) K'(\tilde Q^{\top}\eta^*)\right]\|\delta_0\| \|\delta - \delta_0\| \,.
\end{equation}
Now we claim that for any $\cK < \infty$, $\lambda_{\min} (M_2) > 0$ for all $\|\eta^*\| \le \cK$. Towards that end, define a function $\lambda:B_{\bbR^{2d}}(1) \times B_{\bbR^{2d}}(\cK) \to \bbR_+$ as: 
$$
\lambda: (v, \eta) \mapsto 2\bbE\left[(\delta_0^{\top}g(- \tilde Q^{\top}\tilde \psi_0) \delta_0) 
\left(v^{\top}\tilde Q\right) ^2 f_0(0 \mid \tilde Q) K'(\tilde Q^{\top}\eta)\right]
$$
Clearly $\lambda \ge 0$ and is continuous on a compact set. Hence its infimum must be attained. Suppose the infimum is $0$, i.e. there exists $(v^*, \eta^*)$ such that: 
$$
\bbE\left[(\delta_0^{\top}g(- \tilde Q^{\top}\tilde \psi_0) \delta_0) 
\left(v^{*^{\top}}\tilde Q\right) ^2 f_0(0 \mid \tilde Q) K'(\tilde Q^{\top}\eta^*)\right] = 0 \,.
$$
as $\lambda_{\min}(g(\dot)) \ge c_+$, we must have $\left(v^{*^{\top}}\tilde Q\right) ^2 f_0(0 \mid \tilde Q) K'(\tilde Q^{\top}\eta^*) = 0$ almost surely. But from our assumption, $\left(v^{*^{\top}}\tilde Q\right) ^2 > 0$ and $K'(\tilde Q^{\top}\eta^*) > 0$ almost surely, which implies $f_0(0 \mid \tilde q) = 0$ almost surely, which is a contradiction. Hence there exists $\lambda_-$ such that: 
$$
\lambda_{\min} (M_2)  \ge \lambda_- > 0 \ \ \forall \ \ \|\psi - \psi_0^s\| \le \cK \,.
$$
Hence we have: 
$$
\lambda_{\min}\left(\sigma_n \nabla_{\psi \psi}\bbM^2(\theta)\right) \ge \frac{\lambda_-}{2}(1 - O(\sigma_n)) 
$$
for all theta such that $d_*(\theta, \theta_0^s) \le \eps \,.$ 
\begin{align}
\label{eq:hessian_psi}
    & \frac{1}{\sigma_n}(\psi - \psi_0^s)^{\top}\sigma_n \nabla^{\psi \psi}\bbM^s(\tilde \th) (\psi - \psi^0) \gtrsim \frac{\|\psi - \psi^s_0\|^2}{\sigma_n} \left(1- O(\sigma_n)\right)  
\end{align}
From equation \eqref{eq:hessian_gamma}, \eqref{eq:hessian_cross} and \eqref{eq:hessian_psi} we have: 
\begin{align*}
& \frac12 (\th_0 - \th^0_s)^{\top}\nabla^2 \bbM^s(\th^*)(\th_0 - \th^0_s) \\
& \qquad \qquad \gtrsim \left[\|\beta - \beta^s_0\|^2 + \|\gamma - \gamma^s_0\|^2 + \frac{\|\psi - \psi^s_0\|^2}{\sigma_n}\right]\mathds{1}_{\|\psi - \psi_0^s\| \le \cK \sigma_n}  \,.
\end{align*}
This, along with equation \eqref{eq:lower_curv_smooth} concludes the proof.  
\end{proof}

\subsection{Proof of Lemma \ref{asymp-normality}}
We start by proving that analogues of Lemma 2 of \cite{seo2007smoothed}: we show that: 
\begin{align*}
\lim_{n \to \infty} \bbE\left[ \sqrt{n\sigma_n}\nabla \bbM_n^{s, \psi}(\theta_0)\right] & = 0 \\
\lim_{n \to \infty} \var\left[ \sqrt{n\sigma_n}\nabla \bbM_n^{s, \psi}(\theta_0)\right] & = V^{\psi}
\end{align*}
for some matrix $V^{\psi}$ which will be specified later in the proof. To prove the limit of the expectation: 
\begin{align*}
&  \bbE\left[ \sqrt{n\sigma_n}\nabla \bbM_n^{s, \psi}(\theta_0)\right] \\
& =  \sqrt{\frac{n}{\sigma_n}}\bbE\left[\left\{(Y - X^{\top}(\beta_0 + \delta_0))^2 - (Y - X^{\top}\beta_0)^2\right\}\tilde Q K'\left(\frac{Q^{\top}\psi_0}{\sigma_n}\right)\right] \\
& =  \sqrt{\frac{n}{\sigma_n}}\bbE\left[\left(\delta_0^{\top}g(Q)\delta_0\right)\left(1 - 2\mathds{1}_{Q^{\top}\psi_0 > 0}\right)\tilde Q K'\left(\frac{Q^{\top}\psi_0}{\sigma_n}\right)\right] \\
& = \sqrt{\frac{n}{\sigma_n}} \times \sigma_n \int \int \left(\delta_0^{\top}g(\sigma_nt - \tilde q^{\top}\tilde \psi_0, \tilde q)\delta_0\right)\left(1 - 2\mathds{1}_{t > 0}\right)\tilde q K'\left(t\right) \ f_0(\sigma_n t \mid \tilde q) f (\tilde q) \ dt \ d\tilde q \\
& = \sqrt{n\sigma_n} \left[\int \tilde q \left(\delta_0^{\top}g(- \tilde q^{\top}\tilde \psi_0, \tilde q)\delta_0\right)f_0(0 \mid \tilde q) \cancelto{0}{\left(\int_{-\infty}^{\infty} \left(1 - 2\mathds{1}_{t > 0}\right)K'\left(t\right)  \ dt\right)} f (\tilde q) d\tilde q  + O(\sigma_n)\right] \\
& = O(\sqrt{n\sigma_n^3}) = o(1) \,.
\end{align*}
For the variance part: 
\begin{align*}
& \var\left[ \sqrt{n\sigma_n}\nabla \bbM_n^{s, \psi}(\theta_0)\right] \\
& = \frac{1}{\sigma_n}\var\left(\left\{(Y - X^{\top}(\beta_0 + \delta_0))^2 - (Y - X^{\top}\beta_0)^2\right\}\tilde Q K'\left(\frac{Q^{\top}\psi_0}{\sigma_n}\right)\right) \\
& = \frac{1}{\sigma_n}\bbE\left(\left\{(Y - X^{\top}(\beta_0 + \delta_0))^2 - (Y - X^{\top}\beta_0)^2\right\}^2 \tilde Q\tilde Q^{\top} \left(K'\left(\frac{Q^{\top}\psi_0}{\sigma_n}\right)\right)^2\right) \\
& \qquad \qquad + \frac{1}{\sigma_n}\bbE^{\otimes 2}\left[\left\{(Y - X^{\top}(\beta_0 + \delta_0))^2 - (Y - X^{\top}\beta_0)^2\right\}\tilde Q K'\left(\frac{Q^{\top}\psi_0}{\sigma_n}\right)\right]
\end{align*}
The outer product of the expectation (the second term of the above summand) is $o(1)$ which follows from our previous analysis of the expectation term. For the second moment: 
\begin{align*}
& \frac{1}{\sigma_n}\bbE\left(\left\{(Y - X^{\top}(\beta_0 + \delta_0))^2 - (Y - X^{\top}\beta_0)^2\right\}^2 \tilde Q\tilde Q^{\top} \left(K'\left(\frac{Q^{\top}\psi_0}{\sigma_n}\right)\right)^2\right) \\
& = \frac{1}{\sigma_n}\bbE\left(\left\{(X^{\top}\delta_0)^2(1 - 2\mathds{1}_{Q^{\top}\psi_0 > 0}) -2\eps (X^{\top}\delta_0)\right\}^2 \tilde Q\tilde Q^{\top} \left(K'\left(\frac{Q^{\top}\psi_0}{\sigma_n}\right)\right)^2\right) \\
& = \frac{1}{\sigma_n}\left[\bbE\left((X^{\top}\delta_0)^4 \tilde Q\tilde Q^{\top} \left(K'\left(\frac{Q^{\top}\psi_0}{\sigma_n}\right)\right)^2\right) + 4\sigma_\eps^2\bbE\left((X^{\top}\delta_0)^2 \tilde Q\tilde Q^{\top} \left(K'\left(\frac{Q^{\top}\psi_0}{\sigma_n}\right)\right)^2\right)  \right] \\
& \longrightarrow \left(\int_{-\infty}^{\infty}(K'(t))^2 \ dt\right)\left[\bbE\left(g_{4, \delta_0}(-\tilde Q^{\top}\tilde \psi_0, \tilde Q)\tilde Q\tilde Q^{\top}f_0(0 \mid \tilde Q)\right) \right. \\
& \hspace{10em}+ \left. 4\sigma_\eps^2\bbE\left(\delta_0^{\top}g(-\tilde Q^{\top}\tilde \psi_0, \tilde Q)\delta_0 \tilde Q\tilde Q^{\top}f_0(0 \mid \tilde Q)\right)\right] \\
& := 2V^{\psi} \,.
\end{align*}
Finally using Lemma 6 of \cite{horowitz1992smoothed} we conclude that $ \sqrt{n\sigma_n}\nabla \bbM_n^{s, \psi}(\theta_0) \implies \cN(0, V^{\psi})$. 
\\\\
\noindent
We next prove that $ \sqrt{n}\nabla \bbM_n^{s, \gamma}(\theta_0)$ to normal distribution. This is a simple application of CLT along with bounding some remainder terms which are asymptotically negligible. The gradients are: 
\begin{align*}
\sqrt{n}\begin{pmatrix} \nabla_{\beta}\bbM^s_n(\theta_0^s) \\ \nabla_{\delta}\bbM^s_n(\theta_0^s) \end{pmatrix} & = 2\sqrt{n}\begin{pmatrix}\frac1n \sum_i X_i(X_i^{\top}\beta_0 - Y_i)+ \frac1n \sum_i X_iX_i^{\top}\delta_0 K\left(\frac{Q_i^{\top}\psi_0}{\sigma_n}\right)  \\ 
\frac1n \sum_i \left[X_i(X_i^{\top}\beta_0 + X_i^{\top}\delta_0 - Y_i)\right] K\left(\frac{Q_i^{\top}\psi_0^s}{\sigma_n}\right) \end{pmatrix} \\
& = 2\begin{pmatrix}  -\frac{1}{\sqrt{n}} \sum_i X_i \eps_i + \frac{1}{\sqrt{n}} \sum_i X_iX_i^{\top}\delta_0 \left(K\left(\frac{Q_i^{\top}\psi_0}{\sigma_n}\right) - \mathds{1}_{Q_i^{\top}\psi_0 > 0}\right)  \\  -\frac{1}{\sqrt{n}} \sum_i X_i \eps_iK\left(\frac{Q_i^{\top}\psi_0}{\sigma_n}\right)  + \frac{1}{\sqrt{n}} \sum_i X_iX_i^{\top}\delta_0K\left(\frac{Q_i^{\top}\psi_0}{\sigma_n}\right)\mathds{1}_{Q_i^{\top}\psi_0 \le 0} 
\end{pmatrix}\\
& = 2\begin{pmatrix}  -\frac{1}{\sqrt{n}} \sum_i X_i \eps_i + R_1  \\  -\frac{1 }{\sqrt{n}} \sum_i X_i \eps_i\mathbf{1}_{Q_i^{\top}\psi_0 > 0}  +R_2 
\end{pmatrix}
\end{align*}
That $(1/\sqrt{n})\sum_i X_i \eps_i$ converges to normal distribution follows from a simple application of CLT. Therefore, once we prove that $R_1$ and $R_2$ are $o_p(1)$ we have: 
$$
\sqrt{n} \nabla_{\gamma}\bbM^s_n(\theta_0^s) \overset{\mathscr{L}}{\implies} \cN\left(0, 4V^{\gamma}\right)
$$
where: 
\begin{equation}
\label{eq:def_v_gamma}
V^{\gamma} = \sigma_\eps^2 \begin{pmatrix}\bbE\left[XX^{\top}\right] & \bbE\left[XX^{\top}\mathds{1}_{Q^{\top}\psi_0 > 0}\right] \\
\bbE\left[XX^{\top}\mathds{1}_{Q^{\top}\psi_0 > 0}\right] & \bbE\left[XX^{\top}\mathds{1}_{Q^{\top}\psi_0 > 0}\right] \end{pmatrix} \,.
\end{equation}
To complete the proof we now show that $R_1$ and $R_2$ are $o_p(1)$. For $R_1$, we show that $\bbE[R_1] \to 0$ and $\var(R_1) \to 0$.  For the expectation part: 
\begin{align*}
 & \bbE[R_1] \\
 & = \sqrt{n}\bbE\left[XX^{\top}\delta_0 \left(K\left(\frac{Q^{\top}\psi_0}{\sigma_n}\right) - \mathds{1}_{Q^{\top}\psi_0 > 0}\right)\right] \\
 & = \sqrt{n}\delta_0^{\top}\bbE\left[g(Q) \left(K\left(\frac{Q^{\top}\psi_0}{\sigma_n}\right) - \mathds{1}_{Q^{\top}\psi_0 > 0}\right)\right] \\
 & = \sqrt{n}\int_{\bbR^{p-1}} \int_{-\infty}^{\infty} \delta_0^{\top}g\left(t-\tilde q^{\top}\tilde \psi_0, \tilde q\right)\left(\mathds{1}_{t > 0} - K\left(\frac{t}{\sigma_n}\right)\right)f_0(t \mid \tilde q) f(\tilde q) \ dt \ d\tilde q \\
  & = \sqrt{n}\sigma_n \int_{\bbR^{p-1}} \int_{-\infty}^{\infty} \delta_0^{\top}g\left(\sigma_n z-\tilde q^{\top}\tilde \psi_0, \tilde q\right)\left(\mathds{1}_{z > 0} - K\left(z\right)\right)f_0(\sigma_n z \mid \tilde q) f(\tilde q) \ dz \ d\tilde q \\
  & =  \sqrt{n}\sigma_n \left[\int_{\bbR^{p-1}}\delta_0^{\top}g\left(-\tilde q^{\top}\tilde \psi_0, \tilde q\right) f_0(0 \mid \tilde q) f(\tilde q)  \ d\tilde q \cancelto{0}{\left[\int_{-\infty}^{\infty}  \left(\mathds{1}_{z > 0} - K\left(z\right)\right)\ dz\right]}  + O(\sigma_n) \right] \\
  & = O(\sqrt{n}\sigma_n^2) = o(1) \,.
\end{align*}
For the variance part: 
\begin{align*}
& \var(R_1) \\
& = \var\left(XX^{\top}\delta_0 \left(K\left(\frac{Q^{\top}\psi_0}{\sigma_n}\right) - \mathds{1}_{Q^{\top}\psi_0 > 0}\right)\right) \\
& \le \bbE\left[\|X\|^2 \delta_0^{\top}XX^{\top}\delta_0 \left(K\left(\frac{Q^{\top}\psi_0}{\sigma_n}\right) - \mathds{1}_{Q^{\top}\psi_0 > 0}\right)^2\right] \\
& = O(\sigma_n ) = o(1) \,.
\end{align*}
This shows that $\var(R_1) = o(1)$ and this establishes $R_1 = o_p(1)$. The proof for $R_2$ is similar and hence skipped for brevity. 
\\\\
Our next step is to prove that $\sqrt{n\sigma_n}\nabla_{\psi}\bbM^s_n(\theta_0^s)$ and $\sqrt{n}\nabla \bbM^{s, \gamma}_n(\theta_0^s)$ are asymptotically uncorrelated. Towards that end, first note that: 
\begin{align*}
& \bbE\left[X(X^{\top}\beta_0 - Y) + XX^{\top}\delta_0 K\left(\frac{Q^{\top}\psi_0}{\sigma_n}\right) \right] \\
& = \bbE\left[XX^{\top}\delta_0\left(K\left(\frac{Q^{\top}\psi_0}{\sigma_n}\right) - \mathds{1}_{Q^{\top}\psi_0 > 0}\right)\right] \\
& = \bbE\left[g(Q)\delta_0\left(K\left(\frac{Q^{\top}\psi_0}{\sigma_n}\right) - \mathds{1}_{Q^{\top}\psi_0 > 0}\right)\right] \\
& = \sigma_n \int \int g(\sigma_n t - \tilde q^{\top}\tilde \psi_0, \tilde q)(K(t) - \mathds{1}_{t>0})f_0(\sigma_n t \mid \tilde q) f(\tilde q) \ dt \ d\tilde q  \\
& = \sigma_n \int g(- \tilde q^{\top}\tilde \psi_0, \tilde q)\cancelto{0}{\int_{-\infty}^{\infty} (K(t) - \mathds{1}_{t>0}) \ dt} \ f_0(0 \mid \tilde q) f(\tilde q) \ dt \ d\tilde q + O(\sigma_n^2) \\
& = O(\sigma_n^2) \,.
\end{align*}
Also, it follows from the proof of $\bbE\left[\sqrt{n\sigma_n}\nabla_\psi \bbM_n^s(\theta_0)\right] \to 0$ we have: 
$$
\bbE\left[\left\{(Y - X^{\top}(\beta_0 + \delta_0))^2 - (Y - X^{\top}\beta_0)^2\right\}\tilde Q K'\left(\frac{Q^{\top}\psi_0}{\sigma_n}\right)\right] = O(\sigma_n^2) \,.
$$
Finally note that: 
\begin{align*}
& \bbE\left[\left(\left\{(Y - X^{\top}(\beta_0 + \delta_0))^2 - (Y - X^{\top}\beta_0)^2\right\}\tilde Q K'\left(\frac{Q^{\top}\psi_0}{\sigma_n}\right)\right) \times \right. \\
& \qquad \qquad \qquad \qquad \qquad \left. \left(X(X^{\top}\beta_0 - Y) + XX^{\top}\delta_0 K\left(\frac{Q^{\top}\psi_0}{\sigma_n}\right)\right)^{\top}\right] \\
& = \bbE\left[\left(\left\{(X^{\top}\delta_0)^2(1 - 2\mathds{1}_{Q^{\top}\psi_0 > 0}) - 2\eps X^{\top}\delta_0\right\}\tilde QK'\left(\frac{Q^{\top}\psi_0}{\sigma_n}\right)\right) \right. \\
& \qquad \qquad \qquad \qquad \qquad \left. \times \left\{XX^{\top}\delta_0\left(K\left(\frac{Q^{\top}\psi_0}{\sigma_n}\right) - \mathds{1}_{Q^{\top}\psi_0 > 0}\right) - X\eps \right\}\right] \\
& = \bbE\left[\left((X^{\top}\delta_0)^2(1 - 2\mathds{1}_{Q^{\top}\psi_0 > 0})\tilde QK'\left(\frac{Q^{\top}\psi_0}{\sigma_n}\right)\right) \right. \\
& \qquad \qquad \qquad  \left. \times \left(XX^{\top}\delta_0\left(K\left(\frac{Q^{\top}\psi_0}{\sigma_n}\right) - \mathds{1}_{Q^{\top}\psi_0 > 0}\right)\right)^{\top}\right] \\
& \qquad \qquad + 2\sigma^2_\eps \bbE\left[XX^{\top}\delta_0\tilde Q^{\top}K'\left(\frac{Q^{\top}\psi_0}{\sigma_n}\right)\right] \\
&= O(\sigma_n ) \,.
\end{align*}
Now getting back to the covariance: 
\begin{align*}
& \bbE\left[\left(\sqrt{n\sigma_n}\nabla_{\psi}\bbM^s_n(\theta_0)\right)\left(\sqrt{n}\nabla_\beta \bbM^s_n(\theta_0)\right)^{\top}\right] \\
& = \frac{1}{\sqrt{\sigma_n}}\bbE\left[\left(\left\{(Y - X^{\top}(\beta_0 + \delta_0))^2 - (Y - X^{\top}\beta_0)^2\right\}\tilde Q K'\left(\frac{Q^{\top}\psi_0}{\sigma_n}\right)\right) \times \right. \\
& \qquad \qquad \qquad \qquad \qquad \left. \left(X(X^{\top}\beta_0 - Y) + XX^{\top}\delta_0 K\left(\frac{Q^{\top}\psi_0}{\sigma_n}\right)\right)^{\top}\right] \\
& \qquad \qquad + \frac{n-1}{\sqrt{\sigma_n}}\left[\bbE\left[\left\{(Y - X^{\top}(\beta_0 + \delta_0))^2 - (Y - X^{\top}\beta_0)^2\right\}\tilde Q K'\left(\frac{Q^{\top}\psi_0}{\sigma_n}\right)\right] \right. \\
& \qquad \qquad \qquad \qquad \times \left. \left(\bbE\left[X(X^{\top}\beta_0 - Y) + XX^{\top}\delta_0 K\left(\frac{Q^{\top}\psi_0}{\sigma_n}\right) \right]\right)^{\top}\right] \\
& = \frac{1}{\sqrt{\sigma_n}} \times O(\sigma_n) + \frac{n-1}{\sqrt{\sigma_n}} \times O(\sigma_n^4) = o(1) \,.
\end{align*}
The proof for $\bbE\left[\left(\sqrt{n\sigma_n}\nabla_{\psi}\bbM^s_n(\theta_0)\right)\left(\sqrt{n}\nabla_\delta \bbM^s_n(\theta_0)\right)^{\top}\right]$ is similar and hence skipped. This completes the proof.

\subsection{Proof of Lemma \ref{conv-prob}}
To prove first note that by simple application of law of large number (and using the fact that $\|\psi^* - \psi_0\|/\sigma_n = o_p(1)$ we have: 
\begin{align*}
\nabla^2 \bbM_n^{s, \gamma}(\theta^*) & = 2\begin{pmatrix}\frac{1}{n}\sum_i X_i X_i^{\top} & \frac{1}{n}\sum_i X_i X_i^{\top}K\left(\frac{Q_i^{\top}\psi^*}{\sigma_n}\right) \\ \frac{1}{n}\sum_i X_i X_i^{\top}K\left(\frac{Q_i^{\top}\psi^*}{\sigma_n}\right) & \frac{1}{n}\sum_i X_i X_i^{\top}K\left(\frac{Q_i^{\top}\psi^*}{\sigma_n}\right)
\end{pmatrix} \\
& \overset{p}{\longrightarrow} 2 \begin{pmatrix}\bbE\left[XX^{\top}\right] & \bbE\left[XX^{\top}\mathds{1}_{Q^{\top}\psi_0 > 0}\right] \\ \bbE\left[XX^{\top}\mathds{1}_{Q^{\top}\psi_0 > 0}\right] & \bbE\left[XX^{\top}\mathds{1}_{Q^{\top}\psi_0 > 0}\right] \end{pmatrix} := 2Q^{\gamma}
\end{align*}
The proof of the fact that $\sqrt{\sigma_n}\nabla^2_{\psi \gamma}\bbM_n^s(\theta^*) = o_p(1)$ is same as the proof of Lemma 5 of \cite{seo2007smoothed} and hence skipped. Finally the proof of the fact that 
$$
\sigma_n \nabla^2_{\psi \psi}\bbM_n^s(\theta^*) \overset{p}{\longrightarrow} 2Q^{\psi}\,.
$$
for some non-negative definite matrix $Q$. The proof is similar to that of Lemma 6 of \cite{seo2007smoothed}, using which we conclude the proof with:  
$$
Q^{\psi} = \left(\int_{-\infty}^{\infty} -\text{sign}(t) K''(t) \ dt\right) \times \bbE\left[\delta_0^{\top} g\left(-\tilde Q^{\top}\tilde \psi_0, \tilde Q\right)\delta_0 \tilde Q \tilde Q^{\top} f_0(0 \mid \tilde Q)\right]  \,.
$$
This completes the proof. So we have established: 
\begin{align*}
\sqrt{n}\left(\hat \gamma^s  - \gamma_0\right) & \overset{\mathscr{L}}{\implies} \cN\left(0, \left(Q^\gamma\right)^{-1}V^\gamma \left(Q^\gamma\right)^{-1}\right) \,, \\
\sqrt{\frac{n}{\sigma_n}}\left(\hat \psi^s  - \psi_0\right) & \overset{\mathscr{L}}{\implies} \cN\left(0, \left(Q^\psi\right)^{-1}V^\psi \left(Q^\psi\right)^{-1}\right) \,.
\end{align*}
and they are asymptotically uncorrelated.

\section{Proof of Theorem \ref{thm:binary}}
\label{sec:supp_classification}
In this section, we present the details of the binary response model, the assumptions, a roadmap of the proof and then finally prove Theorem \ref{thm:binary}.
\noindent 
\begin{assumption}
\label{as:distribution}
The below assumptions pertain to the parameter space and the distribution of $Q$:
\begin{enumerate}
\item  The parameter space $\Theta$ is a compact subset of $\mathbb{R}^p$. 
\item  The support of the distribution of $Q$ contains an open subset around origin of $\mathbb{R}^p$ and the distribution of $Q_1$ conditional on $\tilde{Q} = (Q_2, \dots, Q_p)$ has, almost surely, everywhere positive density with respect to Lebesgue measure. 
\end{enumerate}
\end{assumption}

\noindent 
For notational convenience, define the following: 
\begin{enumerate}
\item Define $f_{\psi} (\cdot | \tilde{Q})$ to the conditional density of $Q^{\top}\psi$ given $\tilde{Q}$ for $\theta \in \Theta$. Note that the following relation holds: $$f_{\theta}(\cdot |\tilde{Q}) = f_{Q_1}(\cdot - \tilde{\psi}^{\top}\tilde{Q} | \tilde{Q}) \,.$$ where we define $f_{Q_1}(\cdot | \tilde X)$ is the conditional density of $Q_1$ given $\tilde Q$. 
\item Define $f_0(\cdot | \tilde{Q}) = f_{\psi_0}(\cdot | \tilde{Q})$ where $\psi_0$ is the unique minimizer of the population score function $M(\psi)$. 
\item Define $f_{\tilde Q}(\cdot)$ to be the marginal density of $\tilde Q$. 
\end{enumerate}

\noindent
The rest of the assumptions are as follows: 
\begin{assumption}
\label{as:differentiability}
$f_0(y|\tilde{Q})$ is at-least once continuously differentiable almost surely for all $\tilde{Q}$. Also assume that there exists $\delta$ and $t$ such that $$\inf_{|y| \le \delta} f_0(y|\tilde{Q}) \ge t$$ for all $\tilde{Q}$ almost surely. 
\end{assumption}
This assumption can be relaxed in the sense that one can allow the lower bound $t$ to depend on $\tilde{Q}$, provided that some further assumptions are imposed on $\bbE(t(\tilde{Q}))$. As this does not add anything of significance to the import of this paper, we use Assumption \ref{as:differentiability} to simplify certain calculations.

\begin{assumption}
\label{as:density_bound}
Define $m\left(\tilde{Q}\right) = \sup_{t}f_{X_1}(t | \tilde{Q}) = \sup_{\theta} \sup_{t}f_{\theta}(t | \tilde{Q})$. Assume that $\bbE\left(m\left(\tilde{Q}\right)^2\right) < \infty$. 
\end{assumption}


\begin{assumption}
\label{as:derivative_bound}
Define $h(\tilde{Q}) = \sup_{t} f_0'(t | \tilde{Q})$. Assume that $\bbE\left(h^2\left(\tilde{Q}\right)\right) < \infty$.  
\end{assumption}
\begin{assumption}
\label{as:eigenval_bound}
Assume that $f_{\tilde{Q}}(0) > 0$ and also that the minimum eigenvalue of $\bbE\left(\tilde{Q}\tilde{Q}^{\top}f_0(0|\tilde{Q})\right) > 0$. 
\end{assumption}

\subsection{Sufficient conditions for above assumptions }
We now demonstrate some sufficient conditions for the above assumptions to hold. If the support of $Q$ is compact and both $f_1(\cdot | \tilde Q)$ and $f'_1(\cdot | \tilde Q)$ are uniformly bounded in $\tilde Q$, then Assumptions $(\ref{as:distribution}, \ \ref{as:differentiability}, \ \ref{as:density_bound},\ \ref{as:derivative_bound})$ follow immediately. The first part of Assumption \ref{as:eigenval_bound}, i.e. the assumption $f_{\tilde{Q}}(0) > 0$ is also fairly general and satisfied by many standard probability distributions. The second part of Assumption \ref{as:eigenval_bound} is satisfied when $f_0(0|\tilde{Q})$ has some lower bound independent of $\tilde{Q}$ and $\tilde{Q}$ has non-singular dispersion matrix. 

Below we state our main theorem. In the next section, we first provide a roadmap of our proof and then fill in the corresponding details. For the rest of the paper, \emph{we choose our bandwidth $\sigma_n$ to satisfy $\frac{\log{n}}{n \sigma_n} \rightarrow 0$}.

\noindent
\begin{remark}
As our procedure requires the weaker condition $(\log{n})/(n \sigma_n) \rightarrow 0$, it is easy to see from the above Theorem that the rate of convergence can be almost as fast as $n/\sqrt{\log{n}}$. 
\end{remark}
\begin{remark}
Our analysis remains valid in presence of an intercept term. Assume, without loss of generality, that the second co-ordinate of $Q$ is $1$ and let $\tilde{Q} = (Q_3, \dots, Q_p)$. It is not difficult to check that all our calculations go through under this new definition of $\tilde Q$. We, however, avoid this scenario for simplicity of exposition. 
\end{remark}
\vspace{0.2in}
\noindent
{\bf Proof sketch: }We now provide a roadmap of the proof of Theorem \ref{thm:binary} in this paragraph while the elaborate technical derivations in the later part. 
Define the following: $$T_n(\psi) = \nabla \bbM_n^s(\psi)= -\frac{1}{n\sigma_n}\sum_{i=1}^n (Y_i - \gamma)K'\left(\frac{Q_i^{\top}\psi}{\sigma_n}\right)\tilde{Q}_i$$ $$Q_n(\psi) = \nabla^2 \bbM_n^s(\psi) = -\frac{1}{n\sigma_n^2}\sum_{i=1}^n (Y_i - \gamma)K''\left(\frac{Q_i^{\top}\psi}{\sigma_n}\right)\tilde{Q}_i\tilde{Q}_i^{\top}$$ As $\hat{\psi}^s$ minimizes $\mathbb{M}^s_n(\psi)$ we have $T_n(\hat{\psi}^s) = 0$. Using one step Taylor expansion we have:
\allowdisplaybreaks 
\begin{align*}
T_n(\hat{\psi}^s) = T_n(\psi_0) + Q_n(\psi^*_n)\left(\hat{\psi}^s - \psi_0\right) = 0
\end{align*}
or: 
\begin{equation}
\label{eq:main_eq}  \sqrt{n/\sigma_n}\left(\hat{\psi}^s - \psi_0\right) = -\left(\sigma_nQ_n(\psi^*_n)\right)^{-1}\sqrt{n\sigma_n}T_n(\psi_0) 
\end{equation}
for some intermediate point $\psi^*_n$ between $\hat \psi^s$ and $\psi_0$. The following lemma establishes the asymptotic properties of $T_n(\psi_0)$: 
\begin{lemma}[Asymptotic Normality of $T_n$]
\label{asymp-normality}
\label{asymp-normality}
If $n\sigma_n^{3} \rightarrow \lambda$, then 
$$
\sqrt{n \sigma_n} T_n(\psi_0) \Rightarrow \mathcal{N}(\mu, \Sigma)
$$
where 
$$\mu = -\sqrt{\lambda}\frac{\beta_0 - \alpha_0}{2}\left[\int_{-1}^{1} K'\left(t\right)|t| \ dt \right] \int_{\mathbb{R}^{p-1}}\tilde{Q} f'(0 | \tilde{Q}) \ dP(\tilde{Q})
$$ 
and 
$$\Sigma = \left[a_1 \int_{-1}^{0} \left(K'\left(t\right)\right)^2  \ dt + a_2 \int_{0}^{1} \left(K'\left(t\right)\right)^2 \ dt \right]\int_{\mathbb{R}^{p-1}}\tilde{Q}\tilde{Q}^{\top} f(0|\tilde{Q})  \ dP(\tilde{Q}) \,.
$$ 
Here $a_1 = (1 - \gamma)^2 \alpha_0 + \gamma^2 (1-\alpha_0), a_2 = (1 - \gamma)^2 \beta_0 + \gamma^2 (1-\beta_0)$ and $\alpha_0, \beta_0, \gamma$ are model parameters defined around equation \eqref{eq:new_loss}. 
\end{lemma}
\noindent
In the case that $n \sigma_n^3 \rightarrow 0$, which, holds when $n\sigma_n \rightarrow 0$ as assumed prior to the statement of the theorem, $\lambda = 0$ and we have: 
$$\sqrt{n \sigma_n} T_n(\psi_0) \rightarrow \mathcal{N}(0, \Sigma) \,.$$ 
Next, we analyze the convergence of $Q_n(\psi^*_n)^{-1}$ which is stated in the following lemma: 
\begin{lemma}[Convergence in Probability of $Q_n$]
\label{conv-prob}
Under Assumptions (\ref{as:distribution} - \ref{as:eigenval_bound}), for any random sequence $\breve{\psi}_n$ such that $\|\breve{\psi}_n - \psi_0\|/\sigma_n \overset{P} \rightarrow 0$, 
$$
\sigma_n Q_n(\breve{\psi}_n) \overset{P} \rightarrow Q = \frac{\beta_0 - \alpha_0}{2}\left(\int_{-1}^{1} -K''\left(t \right)\s(t) \ dt\right) \ \bbE\left(\tilde{Q}\tilde{Q}^{\top} f(0 |\tilde{Q})\right) \,.
$$
\end{lemma}
It will be shown later that the condition $\|\breve{\psi}_n - \psi_0\|/\sigma_n \overset{P} \rightarrow 0$ needed in Lemma \ref{conv-prob} holds for the (random) sequence $\psi^*_n$. Then, combining Lemma \ref{asymp-normality} and Lemma \ref{conv-prob} we conclude from equation \ref{eq:main_eq} that: 
$$
\sqrt{n/\sigma_n} \left(\hat{\psi}^s - \psi_0\right) \Rightarrow N(0, Q^{-1}\Sigma Q^{-1}) \,.
$$ 
This concludes the proof of the our Theorem \ref{thm:binary} with $\Gamma = Q^{-1}\Sigma Q^{-1}$. 
\newline
\newline
Observe that, to show $\left\|\psi^*_n - \psi_0 \right\| = o_P(\sigma_n)$, it suffices to to prove that $\left\|\hat \psi^s - \psi_0 \right\| = o_P(\sigma_n)$. Towards that direction, we have following lemma: 

\begin{lemma}[Rate of convergence]
\label{lem:rate}
Under Assumptions (\ref{as:distribution} - \ref{as:eigenval_bound}), 
$$
n^{2/3}\sigma_n^{-1/3}  d^2_n\left(\hat \psi^s, \psi_0^s\right) = O_P(1) \,,
$$ 
where 
$$
d_n\left(\psi, \psi_0^s\right) = \sqrt{\left[\frac{\|\psi - \psi_0^s\|^2}{\sigma_n} \mathds{1}(\|\psi - \psi_0^s\| \le \cK\sigma_n) + \|\psi - \psi_0^s\|  \mathds{1}(\|\psi - \psi_0^s\| \ge \cK\sigma_n)\right]}
$$
for some specific constant $\cK$. (This constant will be mentioned precisely in the proof). 
\end{lemma}

\noindent
The lemma immediately leads to the following corollary: 

\begin{corollary}
\label{rate-cor}
If $n\sigma_n \rightarrow \infty$ then $\|\hat \psi^s - \psi_0^s\|/\sigma_n  \overset{P}  \longrightarrow 0$.
\end{corollary}

\noindent
Finally, to establish $\|\hat \psi^s - \psi_0\|/\sigma_n \overset{P} \rightarrow 0$, all we need is that $\|\psi_0^s - \psi_0\|/\sigma_n \rightarrow 0$ as demonstrated in the following lemma:

\begin{lemma}[Convergence of population minimizer]
\label{bandwidth}
For any sequence of $\sigma_n \rightarrow 0$, we have: $\|\psi_0^s - \psi_0\|/\sigma_n \rightarrow 0$. 
\end{lemma}

\noindent
Hence the final roadmap is the following: Using Lemma \ref{bandwidth} and Corollary \ref{rate-cor} we establish that $\|\hat \psi^s - \psi_0\|/\sigma_n \rightarrow 0$ if $n\sigma_n \rightarrow \infty$. This, in turn, enables us to prove that $\sigma_n Q_n(\psi^*_n) \overset{P} \rightarrow Q$,which, along with Lemma \ref{asymp-normality}, establishes the main theorem. 

\begin{remark}
\label{rem:gamma}
In the above analysis, we have assumed knowledge of $\gamma$ in between $(\alpha_0, \beta_0)$. However, all our calculations go through if we replace $\gamma$ by its estimate (say $\bar Y$) with more tedious book-keeping. One way to simplify the calculations is to split the data into two halves, estimate $\gamma$ (via $\bar Y$) from the first half and then use it as a proxy for $\gamma$ in the second half of the data to estimate $\psi_0$. As this procedure does not add anything of interest to the core idea of our proof, we refrain from doing so here.  
\end{remark}

\subsection{Variant of quadratic loss function}
\label{loss_func_eq}
In this sub-section we argue why the loss function in \eqref{eq:new_loss} is a variant of the quadratic loss function for any $\gamma \in (\alpha_0, \beta_0)$. Assume that we know $\alpha_0, \beta_0$ and seek to estimate $\psi_0$. We start with an expansion of the quadratic loss function: 
\begin{align*}
& \bbE\left(Y - \alpha_0\mathds{1}_{Q^{\top}\psi \le 0} - \beta_0 \mathds{1}_{Q^{\top}\psi > 0}\right)^2  \\
& = \bbE\left(\bbE\left(Y - \alpha_0\mathds{1}_{Q^{\top}\psi \le 0} - \beta_0 \mathds{1}_{Q^{\top}\psi > 0}\right)^2 \ | X\right) \\
& = \bbE_{Q}\left(\bbE\left( Y^2 \mid Q \right) \right) + \bbE_{Q}\left(\alpha_0\mathds{1}_{Q^{\top}\psi \le 0} + \beta_0 \mathds{1}_{Q^{\top}\psi > 0}\right)^2 \\
& \qquad \qquad \qquad -2 \bbE_{Q}\left(\left(\alpha_0\mathds{1}_{Q^{\top}\psi \le 0} + \beta_0 \mathds{1}_{Q^{\top}\psi > 0}\right) \bbE(Y \mid Q)\right) \\
& = \bbE_Q\left(\bbE\left( Y \mid Q \right) \right) + \bbE_Q\left(\alpha_0\mathds{1}_{Q^{\top}\psi \le 0} + \beta_0 \mathds{1}_{Q^{\top}\psi > 0}\right)^2 \\
& \qquad \qquad \qquad -2 \bbE_Q\left(\left(\alpha_0\mathds{1}_{Q^{\top}\psi \le 0} + \beta_0 \mathds{1}_{Q^{\top}\psi > 0}\right) \bbE(Y \mid Q)\right) \\
\end{align*}
Since the first summand is just $\bbE Y$, it is irrelevant to the minimization. A cursory inspection shows that it suffices to minimize
\begin{align}
& \bbE\left(\left(\alpha_0\mathds{1}_{Q^{\top}\psi \le 0} + \beta_0 \mathds{1}_{Q^{\top}\psi > 0}\right) - \bbE(Y \mid Q)\right)^2 \notag\\
\label{eq:lse_1} & = (\beta_0 - \alpha_0)^2 \P\left(\s(Q^{\top}\psi) \neq \s(Q^{\top}\psi_0)\right)
\end{align}
On the other hand the loss we are considering is $\bbE\left((Y - \gamma)\mathds{1}_{Q^{\top}\psi \le 0}\right)$: 
\begin{align}
\label{eq:lse_2} \bbE\left((Y - \gamma)\mathds{1}_{Q^{\top}\psi \le 0}\right) & = (\beta_0 - \gamma)\P(Q^{\top}\psi_0 > 0 , Q^{\top}\psi \le 0) \notag \\
& \hspace{10em}+ (\alpha_0 - \gamma)\P(Q^{\top}\psi_0 \le 0, Q^{\top}\psi \le 0)\,,
\end{align}
which can be rewritten as: 
\begin{align*}
& (\alpha_0 - \gamma)\P(X^{\top} \psi_0 \leq 0) + (\beta_0 - \gamma)\,\P(X^{\top} \psi_0 > 0, X^{\top} \psi \leq 0) \\
& \qquad \qquad \qquad + (\gamma - \alpha_0)\,P (X^{\top} \psi_0 \leq 0, X^{\top} \psi > 0) \,.
\end{align*}
By Assumption \ref{as:distribution}, for $\psi \neq \psi_0$, $\P\left(\s(Q^{\top}\psi) \neq \s(Q^{\top}\psi_0)\right) > 0$. As an easy consequence, equation \eqref{eq:lse_1} is uniquely minimized at $\psi = \psi_0$. To see that the same is true for \eqref{eq:lse_2} when $\gamma \in (\alpha_0, \beta_0)$, note that the first summand in the equation does not depend on $\psi$, that the second and third summands are both non-negative and that at least one of these must be positive under Assumption \ref{as:distribution}. 
\subsection{Linear curvature of the population score function}
Before going into the proofs of the Lemmas and the Theorem, we argue that the population score function $M(\psi)$ has linear curvature near $\psi_0$, which is useful in proving Lemma \ref{lem:rate}. We begin with the following observation: 
\begin{lemma}[Curvature of population risk]
\label{lem:linear_curvature}
Under Assumption \ref{as:differentiability} we have: $$u_- \|\psi - \psi_0\|_2 \le \bbM(\psi) - \bbM(\psi_0) \le u_+ \|\psi - \psi_0\|_2$$ for some constants $0 < u_- < u_+ < \infty$, for all $\psi \in \psi$. 
\end{lemma}
\begin{proof}
First, we show that 
$$
\bbM(\psi) - \bbM(\psi_0) = \frac{(\beta_0 - \alpha_0)}{2} \P(\s(Q^{\top}\psi) \neq X^{\top}(\psi_0))
$$ which follows from the calculation below:
\begin{align*}
& \bbM(\psi) - \bbM(\psi_0)  \\
& = \bbE\left((Y - \gamma)\mathds{1}(Q^{\top}\psi \le 0)\right) - \bbE\left((Y - \gamma)\mathds{1}(Q^{\top}\psi_0 \le 0)\right) \\
& = \frac{\beta_0 - \alpha_0}{2} \bbE\left(\left\{\mathds{1}(Q^{\top}\psi \le 0) - \mathds{1}(Q^{\top}\psi_0 \le 0)\right\}\left\{\mathds{1}(Q^{\top}\psi_0 \ge 0) - \mathds{1}(Q^{\top}\psi_0 \le 0)\right\}\right) \\
& = \frac{\beta_0 - \alpha_0}{2} \bbE\left(\left\{\mathds{1}(Q^{\top}\psi \le 0, Q^{\top}\psi_0 \ge 0) - \mathds{1}(Q^{\top}\psi \le 0, Q^{\top}\psi_0 \le 0) + \mathds{1}(Q^{\top}\psi_0 \le 0)\right\}\right) \\
& = \frac{\beta_0 - \alpha_0}{2} \bbE\left(\left\{\mathds{1}(Q^{\top}\psi \le 0, Q^{\top}\psi_0 \ge 0) + \mathds{1}(Q^{\top}\psi \ge 0, Q^{\top}\psi_0 \le 0)\right\}\right) \\
& = \frac{\beta_0 - \alpha_0}{2} \P(\s(Q^{\top}\psi) \neq \s(Q^{\top}\psi_0)) \,.
\end{align*}
We now analyze the probability of the wedge shaped region, the region between the two hyperplanes $Q^{\top}\psi = 0$ and $Q^{\top}\psi_0 = 0$. Note that, 
\allowdisplaybreaks
\begin{align}
& \P(Q^{\top}\psi > 0 > Q^{\top}\psi_0) \notag\\
& = \P(-\tilde{Q}^{\top}\tilde{\psi} < X_1 < -\tilde{Q}^{\top}\tilde{\psi}_0) \notag\\
\label{lin1} & = \bbE\left[\left(F_{X_1 | \tilde{Q}}\left(-\tilde{Q}^{\top}\tilde{\psi}_0\right) - F_{X_1 | \tilde{Q}}\left(-\tilde{Q}^{\top}\tilde{\psi}\right)\right)\mathds{1}\left(\tilde{Q}^{\top}\tilde{\psi}_0 \le \tilde{Q}^{\top}\tilde{\psi}\right)\right]
\end{align}
A similar calculation yields
\allowdisplaybreaks
\begin{align}
\label{lin2}  \P(Q^{\top}\psi < 0 < Q^{\top}\psi_0) & = \bbE\left[\left(F_{X_1 | \tilde{Q}}\left(-\tilde{Q}^{\top}\tilde{\psi}\right) - F_{X_1 | \tilde{Q}}\left(-\tilde{Q}^{\top}\tilde{\psi}_0\right)\right)\mathds{1}\left(\tilde{Q}^{\top}\tilde{\psi}_0 \ge \tilde{Q}^{\top}\tilde{\psi}\right)\right]
\end{align}
Adding both sides of equation \ref{lin1} and \ref{lin2} we get: 
\begin{equation}
\label{wedge_expression}
\P(\s(Q^{\top}\psi) \neq \s(Q^{\top}\psi_0)) =  \bbE\left[\left|F_{X_1 | \tilde{Q}}\left(-\tilde{Q}^{\top}\tilde{\psi}\right) - F_{X_1 | \tilde{Q}}\left(-\tilde{Q}^{\top}\tilde{\psi}_0\right)\right|\right]
\end{equation}
Define $\psi_{\max} = \sup_{\psi \in \psi}\|\psi\|$, which is finite by Assumption \ref{as:distribution}. Below, we establish the lower bound: 
\allowdisplaybreaks
\begin{align*}
& \P(\s(Q^{\top}\psi) \neq \s(Q^{\top}\psi_0)) \notag\\
& =  \bbE\left[\left|F_{X_1 | \tilde{Q}}\left(-\tilde{Q}^{\top}\tilde{\psi}\right) - F_{X_1 | \tilde{Q}}\left(-\tilde{Q}^{\top}\tilde{\psi}_0\right)\right|\right] \\
& \ge  \bbE\left[\left|F_{X_1 | \tilde{Q}}\left(-\tilde{Q}^{\top}\tilde{\psi}\right) - F_{X_1 | \tilde{Q}}\left(-\tilde{Q}^{\top}\tilde{\psi}_0\right)\right|\mathds{1}\left(\left|\tilde{Q}^{\top}\tilde{\psi}\right| \vee \left| \tilde{Q}^{\top}\tilde{\psi}_0\right| \le \delta\right)\right] \hspace{0.2in} [\delta \ \text{as in Assumption \ref{as:differentiability}}]\\
& \ge  \bbE\left[\left|F_{X_1 | \tilde{Q}}\left(-\tilde{Q}^{\top}\tilde{\psi}\right) - F_{X_1 | \tilde{Q}}\left(-\tilde{Q}^{\top}\tilde{\psi}_0\right)\right|\mathds{1}\left(\|\tilde{Q}\| \le \delta/\psi_{\max}\right)\right] \\
& \ge  t \bbE\left[\left| \tilde{Q}^{\top}(\psi - \psi_0)\right| \mathds{1}\left(\|\tilde{Q}\| \le \delta/\psi_{\max}\right)\right] \\
& = t \|\psi - \psi_0\| \,\bbE\left[\left| \tilde{Q}^{\top}\frac{(\psi - \psi_0)}{\|\psi - \psi_0\|}\right| \mathds{1}\left(\|\tilde{Q}\| \le \delta/\psi_{\max}\right)\right] \\
& \ge t\|\psi - \psi_0\| \inf_{\gamma \in S^{p-1}}\bbE\left[\left| \tilde{Q}^{\top}\gamma\right| \mathds{1}\left(\|\tilde{Q}\| \le \delta/\psi_{\max}\right)\right] \\
& = u_-\|\psi - \psi_0\| \,.
\end{align*} 
At the very end, we have used the fact that $$\inf_{\gamma \in S^{p-1}}\bbE\left[\left| \tilde{Q}^{\top}\gamma\right| \mathds{1}\left(\|\tilde{Q}\| \le \delta/\psi_{\max}\right)\right] > 0$$ To prove this, assume that the infimum is 0. Then, there exists $\gamma_0 \in S^{p-1}$ such that 
$$\bbE\left[\left| \tilde{Q}^{\top}\gamma_0\right| \mathds{1}\left(\|\tilde{Q}\| \le \delta/\psi_{\max}\right)\right] = 0 \,,$$ 
as the above function continuous in $\gamma$ and any continuous function on a compact set attains its infimum. Hence, $\left|\tilde{Q}^{\top}\gamma_0 \right| = 0$ for all $\|\tilde{Q}\| \le \delta/\psi_{\max}$, which implies that $\tilde{Q}$ does not have full support, violating Assumption  \ref{as:distribution} (2). This gives a contradiction.
\\\\
\noindent
Establishing the upper bound is relatively easier. Going back to equation \eqref{wedge_expression}, we have: 
\begin{align*}
& \P(\s(Q^{\top}\psi) \neq \s(Q^{\top}\psi_0)) \notag\\
& =  \bbE\left[\left|F_{Q_1 | \tilde{Q}}\left(-\tilde{Q}^{\top}\tilde{\psi}\right) - F_{Q_1 | \tilde{Q}}\left(-\tilde{Q}^{\top}\tilde{\psi}_0\right)\right|\right] \\
& \le  \bbE\left[m(\tilde Q) \, \|Q\| \,\|\psi- \psi_0\|\right] \hspace{0.2in} [m(\cdot) \ \text{is defined in Assumption \ref{as:density_bound}}]\\
& \le u_+ \|\psi - \psi_0\| \,,
\end{align*}
as $ \bbE\left[m(\tilde Q) \|Q\|\right] < \infty$  by Assumption \ref{as:density_bound} and the sub-Gaussianity of $\tilde X$. 
\end{proof}

\subsection{Proof of Lemma \ref{asymp-normality}}
\begin{proof}
We first prove that under our assumptions $\sigma_n^{-1} \bbE(T_n(\psi_0)) \overset{n \to \infty}\longrightarrow A$ where $$A = -\frac{\beta_0 - \alpha_0}{2!}\left[\int_{-\infty}^{\infty} K'\left(t\right)|t| \ dt \right] \int_{\mathbb{R}^{p-1}}\tilde{Q}f_0'(0 | \tilde{Q}) \ dP(\tilde{Q})$$ The proof is based on Taylor expansion of the conditional density: 
\allowdisplaybreaks
\begin{align*}
& \sigma_n^{-1} \bbE(T_n(\psi_0)) \\
& =  -\sigma_n^{-2}\bbE\left((Y - \gamma)K'\left(\frac{Q^{\top}\psi_0}{\sigma_n}\right)\tilde{Q}\right) \\
& =  -\frac{\beta_0 - \alpha_0}{2}\sigma_n^{-2}\bbE\left(K'\left(\frac{Q^{\top}\psi_0}{\sigma_n}\right)\tilde{Q}(\mathds{1}(Q^{\top}\psi_0 \ge 0) - \mathds{1}(Q^{\top}\psi_0 \le 0))\right) \\
& =  -\frac{\beta_0 - \alpha_0}{2}\sigma_n^{-2}\int_{\mathbb{R}^{p-1}}\tilde{Q}\left[\int_{0}^{\infty} K'\left(\frac{z}{\sigma_n}\right)f_0(z|\tilde{Q}) \ dz - \int_{-\infty}^{0} K'\left(\frac{z}{\sigma_n}\right)f_0(z|\tilde{Q}) \ dz \right] \ dP(\tilde{Q}) \\
& =  -\frac{\beta_0 - \alpha_0}{2}\sigma_n^{-1}\int_{\mathbb{R}^{p-1}}\tilde{Q}\left[\int_{0}^{\infty} K'\left(t\right)f_0(\sigma_n t|\tilde{Q}) \ dt - \int_{-\infty}^{0} K'\left(t\right)f_0(\sigma_n t |\tilde{Q}) \ dt \right] \ dP(\tilde{Q}) \\
& =  -\frac{\beta_0 - \alpha_0}{2}\sigma_n^{-1}\left[\int_{\mathbb{R}^{p-1}}\tilde{Q}\left[\int_{0}^{\infty} K'\left(t\right)f_0(0|\tilde{Q}) \ dt - \int_{-\infty}^{0} K'\left(t\right)f_0(0 |\tilde{Q}) \ dt \right] \ dP(\tilde{Q}) \right. \\ 
& \qquad \qquad \qquad + \left. \int_{\mathbb{R}^{p-1}}\sigma_n \left[\int_{0}^{\infty} K'\left(t\right)tf_0'(\lambda \sigma_n t|\tilde{Q}) \ dt - \int_{-\infty}^{0} K'\left(t\right) t f_0'(\lambda \sigma_n t |\tilde{Q}) \ dt \right] \ dP(\tilde{Q}) \right] \hspace{0.2in} [0 < \lambda < 1]\\ 
& = -\frac{\beta_0 - \alpha_0}{2}\int_{\mathbb{R}^{p-1}}\tilde{Q}\left[\int_{0}^{\infty} k\left(t\right)tf_0'(\lambda \sigma_n t|\tilde{Q}) \ dz - \int_{-\infty}^{0} k\left(t\right)tf_0'(\lambda \sigma_nt |\tilde{Q}) \ dz \right] \ dP(\tilde{Q})\\
& \underset{n \rightarrow \infty} \longrightarrow -\frac{\beta_0 - \alpha_0}{2}\left[\int_{-\infty}^{\infty} k\left(t\right)|t| \ dt \right] \int_{\mathbb{R}^{p-1}}\tilde{Q}f_0'(0 | \tilde{Q}) \ dP(\tilde{Q})
\end{align*}
Next, we prove that $\mbox{Var}\left(\sqrt{n\sigma_n}T_n(\psi_0)\right)\longrightarrow \Sigma$ as $n \rightarrow \infty$, where $\Sigma$ is as defined in Lemma \ref{asymp-normality}. Note that: 
\allowdisplaybreaks
\begin{align*}
\mbox{Var}\left(\sqrt{n\sigma_n}T_n(\psi_0)\right) & = \sigma_n \bbE\left((Y - \gamma)^2\left(K'\left(\frac{Q^{\top}\psi_0}{\sigma_n}\right)^2\frac{\tilde{Q}\tilde{Q}^{\top}}{\sigma_n^2}\right)\right) - \sigma_n \bbE(T_n(\psi_0))\bbE(T_n(\psi_0))^{\top}
\end{align*}
As $\sigma_n^{-1}\bbE(T_n(\psi_0)) \rightarrow A$, we can conclude that $\sigma_n \bbE(T_n(\psi_0))\bbE(T_n(\psi_0))^{\top} \rightarrow 0$. 
Define $a_1 = (1 - \gamma)^2 \alpha_0 + \gamma^2 (1-\alpha_0), a_2 = (1 - \gamma)^2 \beta_0 + \gamma^2 (1-\beta_0)$. For the first summand: 
\allowdisplaybreaks
\begin{align*}
& \sigma_n \bbE\left((Y - \gamma)^2\left(K^{'^2}\left(\frac{Q^{\top}\psi_0}{\sigma_n}\right)\frac{\tilde{Q}\tilde{Q}^{\top}}{\sigma_n^2}\right)\right) \\
& = \frac{1}{\sigma_n} \int_{\mathbb{R}^{p-1}}\tilde{Q}\tilde{Q}^{\top} \left[a_1 \int_{-\infty}^{0} K^{'^2}\left(\frac{z}{\sigma_n}\right) f(z|\tilde{Q}) \ dz \right. \notag \\ & \left. \qquad \qquad \qquad + a_2 \int_{0}^{\infty}K^{'^2}\left(\frac{z}{\sigma_n}\right) f(z|\tilde{Q}) \ dz \right] \ dP(\tilde{Q})\\
& =  \int_{\mathbb{R}^{p-1}}\tilde{Q}\tilde{Q}^{\top} \left[a_1 \int_{-\infty}^{0} K^{'^2}\left(t\right)f(\sigma_n t|\tilde{Q}) \ dt + a_2 \int_{0}^{\infty} K^{'^2}\left(t\right) f(\sigma_n t |\tilde{Q}) \ dt \right] \ dP(\tilde{Q}) \\
& =  \int_{\mathbb{R}^{p-1}}\tilde{Q}\tilde{Q}^{\top} \left[a_1 \int_{-\infty}^{0} K^{'^2}\left(t\right)f(\sigma_n t|\tilde{Q}) \ dt + a_2 \int_{0}^{\infty} K^{'^2}\left(t\right) f(\sigma_n t |\tilde{Q}) \ dt \right] \ dP(\tilde{Q}) \\
& \underset{n \rightarrow \infty} \longrightarrow \left[a_1 \int_{-\infty}^{0} K^{'^2}\left(t\right)  \ dt + a_2 \int_{0}^{\infty} K^{'^2}\left(t\right) \ dt \right]\int_{\mathbb{R}^{p-1}}\tilde{Q}\tilde{Q}^{\top} f(0|\tilde{Q})  \ dP(\tilde{Q}) \ \ \overset{\Delta} = \Sigma \, . 
\end{align*}
Finally, suppose $n \sigma_n^{3} \rightarrow \lambda$. Define $W_n = \sqrt{n\sigma_n}\left[T_n(\psi) - \bbE(T_n(\psi))\right]$. Using Lemma 6 of Horowitz \cite{horowitz1992smoothed}, it is easily established that $W_n \Rightarrow N(0, \Sigma)$. Also, we have: 
\allowdisplaybreaks 
\begin{align*}
\sqrt{n\sigma_n}\bbE(T_n(\psi_0)) = \sqrt{n\sigma_n^{3}}\sigma_n^{-1}\bbE(T_n(\psi_0) & \rightarrow \sqrt{\lambda}A = \mu
\end{align*}
As $\sqrt{n\sigma_n}T_n(\psi_0) = W_n + \sqrt{n\sigma_n}\bbE(T_n(\psi_0))$, we conclude that $\sqrt{n\sigma_n} T_n(\psi_0) \Rightarrow N(\mu, \Sigma)$.
\end{proof}
\subsection{Proof of Lemma \ref{conv-prob}}
\begin{proof}
Let $\epsilon_n \downarrow 0$ be a sequence such that $\P(\|\breve{\psi}_n - \psi_0\| \le \epsilon_n \sigma_n) \rightarrow 1$. Define $\Psi_n = \{\psi: \|\psi - \psi_0\| \le \epsilon_n \sigma_n\}$. We show that $$\sup_{\psi \in \psi_n} \|\sigma_n Q_n(\psi) - Q\|_F \overset{P} \to 0$$ where $\|\cdot\|_F$ denotes the Frobenius norm of a matrix. Sometimes, we omit the subscript $F$ when there is no ambiguity. Define $\mathcal{G}_n$ to be collection of functions: 
$$
\mathcal{G}_n= \left\{g_{\psi}(y, q) = -\frac{1}{\sigma_n}(y - \gamma)\tilde q\tilde q^{\top} \left(K''\left(\frac{q^{\top}\psi}{\sigma_n}\right) - K''\left(\frac{q^{\top}\psi_0}{\sigma_n}\right)\right), \psi \in \Psi_n \right\}
$$
That the function class $\cG_n$ has bounded uniform entropy integral (BUEI) is immediate from the fact that the function $Q \to Q^{\top}\psi$ has finite VC dimension (as the hyperplanes has finite VC dimension) and it does change upon constant scaling. Therefore $Q \mapsto Q^{\top}\psi/\sigma_n$ also has finite VC dimension which does not depend on n and hence BUEI. As composition with a monotone function and multiplication with constant (parameter free) functions or multiplication of two BUEI class of functions keeps BUEI property, we conclude that $\cG_n$ has BUEI. 
We first expand the expression in two terms:
\allowdisplaybreaks
\begin{align*}
\sup_{\psi \in \psi_n} \|\sigma_n Q_n(\psi) - Q\| & \le \sup_{\psi \in \psi_n} \|\sigma_n Q_n(\psi) - \bbE(\sigma_n Q_n(\psi))\| + \sup_{\psi \in \psi_n} \| \bbE(\sigma_n Q_n(\psi)) - Q\| \\ 
& = \|(\mathbb{P}_n - P)\|_{\mathcal{G}_n} + \sup_{\psi \in \psi_n}\| \bbE(\sigma_n Q_n(\psi)) - Q\| \\
& = T_{1,n} + T_{2,n} \hspace{0.3in} \,. [\text{Say}]
\end{align*}

\vspace{0.2in}
\noindent
That $T_{1,n} \overset{P} \to 0$ follows from uniform law of large number of a BUEI class (e.g. combining Theorem 2.4.1 and Theorem 2.6.7 of \cite{vdvw96}). 
For uniform convergence of the second summand $T_{n,2}$, define $\chi_n = \{\tilde{Q}: \|\tilde{Q}\| \le 1/\sqrt{\epsilon_n}\}$. Then $\chi_n \uparrow \mathbb{R}^{p-1}$. Also for any $\psi \in \Psi_n$, if we define $\gamma_n \equiv \gamma_n(\psi) = (\psi - \psi_0)/\sigma_n$, then $|\tilde \gamma_n^{\top}\tilde{Q}| \le \sqrt{\epsilon_n}$ for all $n$ and for all $\psi \in \Psi_n, \tilde{Q} \in \chi_n$. Now, 
\allowdisplaybreaks
\begin{align*}
& \sup_{\psi \in \psi_n}\| \bbE(\sigma_n Q_n(\psi)) - Q\| \notag \\
&\qquad \qquad  = \sup_{\psi \in \psi_n}\| (\bbE(\sigma_n Q_n(\psi)\mathds{1}(\chi_n))-Q_1) +  (\bbE(\sigma_n Q_n(\psi)\mathds{1}(\chi_n^c))-Q_2)\|
\end{align*}
where $$Q_1 =  \frac{\beta_0 - \alpha_0}{2}\left(\int_{-\infty}^{\infty} -K''\left(t \right)\s(t) \ dt\right) \ \bbE\left(\tilde{Q}\tilde{Q}^{\top} f_0(0 |\tilde{Q})\mathds{1}(\chi_n) \right)$$ $$Q_2 =  \frac{\beta_0 - \alpha_0}{2}\left(\int_{-\infty}^{\infty} -K''\left(t \right)\s(t) \ dt\right) \ \bbE\left(\tilde{Q}\tilde{Q}^{\top} f(0 |\tilde{Q})\mathds{1}(X_n^c) \right) \,.$$
Note that 
\allowdisplaybreaks
\begin{flalign}
& \|\bbE(\sigma_n Q_n(\psi)\mathds{1}(\chi_n)) - Q_1\| \notag\\
& =\left\| \frac{\beta_0 - \alpha_0}{2}\left[\int_{\chi_n} \tilde{Q}\tilde{Q}^{\top} \left[\int_{-\infty}^{\tilde{Q}^{\top}\gamma_n} K''\left(t \right) f_0(\sigma_n (t-\tilde{Q}^{\top}\gamma_n) |\tilde{Q}) \ dt \right. \right. \right. \notag \\
& \left. \left. \left. \qquad \qquad - \int_{\tilde{Q}^{\top}\gamma_n}^{\infty} K''\left(t\right) f_0(\sigma_n (t - \tilde{Q}^{\top}\gamma_n) | \tilde{Q}) \ dt \right]dP(\tilde{Q})\right]\right. \notag\\ & \left. \qquad \qquad \qquad - \frac{\beta_0 - \alpha_0}{2}\left[\int_{\chi_n} \tilde{Q}\tilde{Q}^{\top} f(0 |\tilde{Q})\left[\int_{-\infty}^{0} K''\left(t \right) \ dt - \int_{0}^{\infty} K''\left(t\right) \ dt \right]dP(\tilde{Q})\right] \right \|\notag\\
& =\left \| \frac{\beta_0 - \alpha_0}{2}\left[\int_{\chi_n} \tilde{Q}\tilde{Q}^{\top} \left[\int_{-\infty}^{\tilde{Q}^{\top}\gamma_n} K'''\left(t \right) (f_0(\sigma_n (t-\tilde{Q}^{\top}\gamma_n) |\tilde{Q})-f_0(0 | \tilde{Q})) \ dt \right. \right. \right.\notag\\& \qquad \qquad- \left. \left. \left. \int_{\tilde{Q}^{\top}\gamma_n}^{\infty} K''\left(t\right) (f_0(\sigma_n (t - \tilde{Q}^{\top}\gamma_n) | \tilde{Q}) - f_0(0 | \tilde{Q})) \ dt \right]dP(\tilde{Q})\right]\right. \notag\\ & \qquad \qquad \qquad + \left. \frac{\beta_0 - \alpha_0}{2}\left[\int_{\chi_n} \tilde{Q}\tilde{Q}^{\top} f_0(0 |\tilde{Q}) \left[\int_{-\infty}^{\tilde{Q}^{\top}\gamma_n} K''\left(t \right) \ dt - \int_{-\infty}^{0} K''\left(t \right) \ dt \right. \right. \right. \notag \\ 
& \qquad \qquad \qquad \qquad \left. \left. \left. + \int_{\tilde{Q}^{\top}\gamma_n}^{\infty} K''\left(t \right) \ dt - \int_{0}^{\infty} K''\left(t\right) \ dt \right]dP(\tilde{Q})\right] \right \|\notag\\
& \le \frac{\beta_0 - \alpha_0}{2}\sigma_n  \int_{\chi_n}\|\tilde{Q}\tilde{Q}^{\top}\|h(\tilde{Q})\int_{-\infty}^{\infty}|K''(t)||t - \gamma_n^{\top}\tilde{Q}| \ dt \ dP(\tilde{Q}) \notag\\ & \qquad \qquad + \frac{\beta_0 - \alpha_0}{2}  \int_{\chi_n}\|\tilde{Q}\tilde{Q}^{\top}\| f_0(0 | \tilde{Q}) \left[\left| \int_{-\infty}^{\tilde{Q}^{\top}\gamma_n} K''\left(t \right) \ dt - \int_{-\infty}^{0} K''\left(t \right) \ dt \right| \right. \notag \\ & \left. \qquad \qquad \qquad + \left| \int_{\tilde{Q}^{\top}\gamma_n}^{\infty} K''\left(t \right) \ dt - \int_{0}^{\infty} K''\left(t\right) \ dt \right|\right] \ dP(\tilde{Q})\notag\\
&  \le \frac{\beta_0 - \alpha_0}{2}\left[\sigma_n  \int_{\chi_n}\|\tilde{Q}\tilde{Q}^{\top}\|h(\tilde{Q})\int_{-\infty}^{\infty}|K''(t)||t - \gamma_n^{\top}\tilde{Q}| \ dt \ dP(\tilde{Q}) \right. \notag \\ 
& \left. \qquad \qquad \qquad  + 2\int_{\chi_n}\|\tilde{Q}\tilde{Q}^{\top}\| f_0(0 | \tilde{Q}) (K'(0) - K'(\gamma_n^{\top}\tilde{Q})) \ dP(\tilde{Q})\right]\notag \\
\label{cp1}&\rightarrow 0 \hspace{0.3in} [\text{As} \ n \rightarrow \infty] \,,
\end{flalign}
by DCT and Assumptions \ref{as:distribution} and \ref{as:derivative_bound}. For the second part: 
\allowdisplaybreaks
\begin{align}
& \|\bbE(\sigma_n Q_n(\psi)\mathds{1}(\chi_n^c)) - Q_2\|\notag\\
& =\left\| \frac{\beta_0 - \alpha_0}{2}\left[\int_{\chi_n^c} \tilde{Q}\tilde{Q}^{\top} \left[\int_{-\infty}^{\tilde{Q}^{\top}\gamma_n} K''\left(t \right) f_0(\sigma_n (t-\tilde{Q}^{\top}\gamma_n) |\tilde{Q}) \ dt \right. \right. \right. \notag \\ 
& \left. \left. \left. \qquad \qquad - \int_{\tilde{Q}^{\top}\gamma_n}^{\infty} K''\left(t\right) f_0(\sigma_n (t - \tilde{Q}^{\top}\gamma_n) | \tilde{Q}) \ dt \right]dP(\tilde{Q})\right]\right. \notag\\ &  \left. \qquad \qquad \qquad -\frac{\beta_0 - \alpha_0}{2}\left[\int_{\chi_n^c} \tilde{Q}\tilde{Q}^{\top} f_0(0 |\tilde{Q})\left[\int_{-\infty}^{0} K''\left(t \right) \ dt - \int_{0}^{\infty} K''\left(t\right) \ dt \right]dP(\tilde{Q})\right] \right \|\notag\\
& \le \frac{\beta_0 - \alpha_0}{2} \int_{\infty}^{\infty} |K''(t)| \ dt \int_{\chi_n^c} \|\tilde{Q}\tilde{Q}^{\top}\|(m(\tilde{Q}) + f_0(0|\tilde{Q})) \ dP(\tilde{Q}) \notag\\
\label{cp2} & \rightarrow 0 \hspace{0.3in} [\text{As} \ n \rightarrow \infty] \,,
\end{align}
again by DCT and Assumptions \ref{as:distribution} and \ref{as:density_bound}. Combining equations \ref{cp1} and \ref{cp2}, we conclude the proof. 
\end{proof}

\subsection{Proof of Lemma \ref{bandwidth}}
Here we prove that $\|\psi^s_0 - \psi_0\|/\sigma_n \rightarrow 0$ where $\psi^s_0$ is the minimizer of $\bbM^s(\psi)$ and $\psi_0$ is the minimizer of $M(\psi)$. 
\begin{proof}
Define $\eta = (\psi^s_0 -  \psi_0)/\sigma_n$. At first we show that, $\|\tilde \eta\|_2$ is $O(1)$, i.e. there exists some constant $\Omega_1$ such that $\|\tilde \eta\|_2 \le \Omega_1$ for all $n$: 
\begin{align*}
\|\psi^s_0 - \psi_0\|_2 & \le \frac{1}{u_-} \left(\bbM(\psi_n) - \bbM(\psi_0)\right)  \hspace{0.2in} [\text{Follows from Lemma} \ \ref{lem:linear_curvature}]\\
& \le \frac{1}{u_-} \left(\bbM(\psi_n) - \bbM^s(\psi_n) + \bbM^s(\psi_n) - \bbM^s(\psi_0) + \bbM^s(\psi_0) - \bbM(\psi_0)\right)  \\
& \le  \frac{1}{u_-} \left(\bbM(\psi_n) - \bbM^s(\psi_n) + \bbM^s(\psi_0) - M(\psi_0)\right)  \hspace{0.2in} [\because \bbM^s(\psi_n) - \bbM^s(\psi_0) \le 0]\\
& \le \frac{2K_1}{u_-}\sigma_n \hspace{0.2in} [\text{from equation} \ \eqref{eq:lin_bound_1}]
\end{align*}

\noindent
As $\psi^s_0$ minimizes $\bbM^s(\psi)$: 
$$\nabla \bbM^s(\psi^s_0) = -\bbE\left((Y-\gamma)\tilde{Q}K'\left(\frac{Q^{\top}\psi^0_s}{\sigma_n}\right)\right) = 0$$
Hence:
\begin{align*}
0 &= \bbE\left((Y-\gamma)\tilde{Q}K'\left(\frac{Q^{\top}\psi_0^s}{\sigma_n}\right)\right) \\
& = \frac{(\beta_0 - \alpha_0)}{2} \bbE\left(\tilde{Q}K'\left(\frac{Q^{\top}\psi_0^s}{\sigma_n}\right)\left\{\mathds{1}(Q^{\top}\psi_0 \ge 0) -\mathds{1}(Q^{\top}\psi_0 < 0)\right\}\right) \\
& = \frac{(\beta_0 - \alpha_0)}{2} \bbE\left(\tilde{Q}K'\left(\frac{Q^{\top}\psi_0}{\sigma_n} + \tilde{\eta}^{\top} \tilde{Q}\right)\left\{\mathds{1}(Q^{\top}\psi_0 \ge 0) -\mathds{1}(Q^{\top}\psi_0 < 0)\right\}\right) \\
& = \frac{(\beta_0 - \alpha_0)}{2} \left[\int_{\mathbb{R}^{p-1}}\tilde{Q} \int_0^{\infty} K'\left(\frac{z}{\sigma_n} + \tilde{\eta}^{\top} \tilde{Q}\right) \ f_0(z|\tilde{Q}) \ dz \ dP(\tilde{Q})\right. \\
& \qquad \qquad \qquad \qquad \qquad \left. - \int_{\mathbb{R}^{p-1}}\tilde{Q} \int_{-\infty}^0 K'\left(\frac{z}{\sigma_n} + \tilde{\eta}^{\top} \tilde{Q}\right) \ f_0(z|\tilde{Q}) \ dz \ dP(\tilde{Q})\right] \\
& =\sigma_n \frac{(\beta_0 - \alpha_0)}{2} \left[\int_{\mathbb{R}^{p-1}}\tilde{Q} \int_0^{\infty} K'\left(t + \tilde{\eta}^{\top} \tilde{Q}\right) \ f_0(\sigma_n t|\tilde{Q}) \ dt \ dP(\tilde{Q})\right. \\
& \qquad \qquad \qquad \qquad \qquad \left. - \int_{\mathbb{R}^{p-1}}\tilde{Q} \int_{-\infty}^0 K'\left(t + \tilde{\eta}^{\top} \tilde{Q}\right) \ f_0(\sigma_n t|\tilde{Q}) \ dz \ dP(\tilde{Q})\right] 
\end{align*}
As $\sigma_n\frac{(\beta_0 - \alpha_0)}{2} > 0$, we can forget about it and continue. Also, as we have proved $\|\tilde \eta\| = O(1)$, there exists a subsequence $\eta_{n_k}$ and a point $c \in \mathbb{R}^{p-1}$ such that $\eta_{n_k} \rightarrow c$. Along that sub-sequence we have:  
\begin{align*}
0 & = \left[\int_{\mathbb{R}^{p-1}}\tilde{Q} \int_0^{\infty} K'\left(t + \tilde{\eta}_{n_k}^{\top} \tilde{Q}\right) \ f_0(\sigma_{n_k} t|\tilde{Q}) \ dt \ dP(\tilde{Q})\right. \\
& \qquad \qquad \qquad \qquad \qquad \left. - \int_{\mathbb{R}^{p-1}}\tilde{Q} \int_{-\infty}^0 K'\left(t + \tilde{\eta}_{n_k}^{\top} \tilde{Q}\right) \ f_0(\sigma_{n_k} t|\tilde{Q}) \ dt \ dP(\tilde{Q})\right] 
\end{align*}
Taking limits on both sides and applying DCT (which is permissible by DCT) we conclude: 
\begin{align*}
0 & = \left[\int_{\mathbb{R}^{p-1}}\tilde{Q} \int_0^{\infty} K'\left(t +c^{\top} \tilde{Q}\right) \ f_0(0|\tilde{Q}) \ dt \ dP(\tilde{Q})\right. \\
& \qquad \qquad \qquad \qquad \qquad \left. - \int_{\mathbb{R}^{p-1}}\tilde{Q} \int_{-\infty}^0 K'\left(t + c^{\top} \tilde{Q}\right) \ f_0(0|\tilde{Q}) \ dt \ dP(\tilde{Q})\right]  \\
& = \left[\int_{\mathbb{R}^{p-1}}\tilde{Q}  \ f_0(0|\tilde{Q}) \int_{c^{\top} \tilde{Q}}^{\infty} K'\left(t\right) \ dt \ dP(\tilde{Q})\right. \\
& \qquad \qquad \qquad \qquad \qquad \left. - \int_{\mathbb{R}^{p-1}}\tilde{Q}\ f_0(0|\tilde{Q})  \int_{-\infty}^{c^{\top} \tilde{Q}} K'\left(t \right) \ dt \ dP(\tilde{Q})\right]  \\
& = \left[\int_{\mathbb{R}^{p-1}}\tilde{Q}  \ f_0(0|\tilde{Q}) \left[1 - K(c^{\top} \tilde{Q})\right] \ dt \ dP(\tilde{Q})\right. \\
&  \qquad \qquad \qquad \qquad \qquad\left. - \int_{\mathbb{R}^{p-1}}\tilde{Q}\ f_0(0|\tilde{Q}) K(c^{\top} \tilde{Q}) \ dt \ dP(\tilde{Q})\right]  \\
& =   \bbE\left(\tilde{Q} \left(2K(c^{\top} \tilde{Q}) - 1\right)f_0(0|\tilde{Q})\right) \,.
\end{align*}
Now, taking the inner-products of both sides with respect to $c$, we get: 
\begin{equation}
\label{eq:zero_eq}
\bbE\left(c^{\top}\tilde{Q} \left(2K(c^{\top} \tilde{Q}) - 1\right)f_0(0|\tilde{Q})\right) = 0 \,.
\end{equation}
By our assumption that $K$ is symmetric kernel and that $K(t) > 0$ for all $t \in (-1, 1)$, we easily conclude that $c^{\top}\tilde{Q} \left(2K(c^{\top} \tilde{Q}) - 1\right) \ge 0$ almost surely in $\tilde{Q}$ with equality iff $c^{\top}X = 0$, which is not possible unless $c = 0$. Hence we conclude that $c = 0$. This shows that any convergent subsequence of $\eta_n$ converges to $0$, which completes the proof. 
\end{proof}

\subsection{Proof of Lemma \ref{lem:rate}}
\begin{proof}
To obtain the rate of convergence of our kernel smoothed estimator we use Theorem 3.4.1 of \cite{vdvw96}: There are three key ingredients that one needs to take care of if in order to apply this theorem:   
\begin{enumerate}
\item Consistency of the estimator (otherwise the conditions of the theorem needs to be valid for all $\eta$). 
\item The curvature of the population score function near its minimizer.
\item A bound on the modulus of continuity in a vicinity of the minimizer of the population score function. 
\end{enumerate}
Below, we establish the curvature of the population score function (item 2 above) globally, thereby obviating the need to establish consistency separately. Recall that the population score function was defined as: 
$$
\bbM^s(\psi) = \bbE\left((Y - \gamma)\left(1 - K\left(\frac{Q^{\top}\psi}{\sigma_n}\right)\right)\right)
$$ 
and our estimator $\hat{\psi}_n$ is the argmin of the corresponding sample version. Consider the set of functions $\mathcal{H}_n = \left\{h_{\psi}: h_{\psi}(q,y) = (y - \gamma)\left(1 - K\left(\frac{q^{\top}\psi}{\sigma_n}\right)\right)\right\}$. Next, we argue that $\mathcal{H}_n$ is a VC class of functions with fixed VC dimension. We know that the function $\{(q,y) \mapsto q^{\top}\psi/\sigma_n: \psi \in \psi\}$ has fixed VC dimension (i.e. not depending on $n$). Now, as a finite dimensional VC class of functions composed with a fixed monotone function or multiplied by a fixed function still remains a finite dimensional VC class, we conclude that $\mathcal{H}_n$ is a fixed dimensional VC class of functions with bounded envelope (as the functions considered here are bounded by 1). 

Now, we establish a lower bound on the curvature of the population score function $\bbM^s(\psi)$ near its minimizer $\psi_n$: 
$$
\bbM^s(\psi) - \bbM^s(\psi_n) \gtrsim d^2_n(\psi, \psi_n)$$ where $$d_n(\psi, \psi_n) = \sqrt{\frac{\|\psi - \psi_n\|^2}{\sigma_n} \mathds{1}\left(\|\psi - \psi_n\| \le \cK\sigma_n\right) + \|\psi - \psi_n\|\mathds{1}\left(\|\psi - \psi_n\| > \cK\sigma_n\right)}
$$ for some constant $\cK > 0$. The intuition behind this compound structure is following: When $\psi$ is in $\sigma_n$ neighborhood of $\psi_n$, $\bbM^s(\psi)$ behaves like a smooth quadratic function, but when it is away from the truth, $\bbM^s(\psi)$ starts resembling $M(\psi)$ which induces the linear curvature. 
\\\\
\noindent
For the linear part, we first establish that $|\bbM(\psi) - \bbM^s(\psi)| = O(\sigma_n)$ uniformly for all $\psi$. Define $\eta = (\psi - \psi_0)/\sigma_n$:
\allowdisplaybreaks
\begin{align}
& |\bbM(\psi) - \bbM^s(\psi)| \notag \\
& \le \bbE\left(\left | \mathds{1}(Q^{\top}\psi \ge 0) - K\left(\frac{Q^{\top}\psi}{\sigma_n}\right)\right | \right) \notag\\
& = \bbE\left(\left | \mathds{1}\left(\frac{Q^{\top}\psi_0}{\sigma_n} + \eta^{\top}\tilde{Q} \ge 0\right) - K\left(\frac{Q^{\top}\psi_0}{\sigma_n} + \eta^{\top}\tilde{Q}\right)\right | \right) \notag \\
& = \sigma_n \int_{\mathbb{R}^{p-1}} \int_{-\infty}^{\infty} \left | \mathds{1}\left(t + \eta^{\top}\tilde{Q} \ge 0\right) - K\left(t + \eta^{\top}\tilde{Q}\right)\right | f_0(\sigma_n t | \tilde{Q}) \ dt \ dP(\tilde{Q}) \notag\\
& = \sigma_n \int_{\mathbb{R}^{p-1}} \int_{-\infty}^{\infty} \left | \mathds{1}\left(t  \ge 0\right) - K\left(t \right)\right | f_0(\sigma_n (t-\eta^{\top}\tilde{Q}) | \tilde{Q}) \ dt \ dP(\tilde{Q}) \notag \\
& = \sigma_n \int_{\mathbb{R}^{p-1}} m(\tilde{Q})\int_{-\infty}^{\infty} \left | \mathds{1}\left(t  \ge 0\right) - K\left(t \right)\right |  \ dt \ dP(\tilde{Q}) \notag\\
& = \sigma_n \bbE(m(\tilde{Q})) \int_{-\infty}^{\infty} \left | \mathds{1}\left(t  \ge 0\right) - K\left(t \right)\right |  \ dt \notag \\
\label{eq:lin_bound_1} & \le K_1 \sigma_n \bbE(m(\tilde{Q})) < \infty \hspace{0.3in}  [\text{by Assumption \ref{as:density_bound}}] \,.
\end{align}
Here, the constant $K_1$ is $\bbE(m(\tilde{Q}))  \left[\int_{-1}^{1}\left | \mathds{1}\left(t  \ge 0\right) - K\left(t \right)\right |  \ dt \right]$ which does not depend on $\psi$, hence the bound is uniform over $\psi$. Next: 
\begin{align*}
\bbM^s(\psi) - \bbM^s(\psi_0^s) & = \bbM^s(\psi) - \bbM(\psi) + \bbM(\psi) - \bbM(\psi_0) \\
& \qquad \qquad + \bbM(\psi_0) - \bbM(\psi_0^s) + \bbM(\psi_0^s) -\bbM^s(\psi_0^s) \\ 
& = T_1 + T_2 + T_3 + T_4
\end{align*}
\noindent
We bound each summand separately: 
\begin{enumerate}
\item $T_1 = \bbM^s(\psi) - \bbM(\psi) \ge -K_1 \sigma_n$ by equation \ref{eq:lin_bound_1}\, 
\item $T_2 = \bbM(\psi) - \bbM(\psi_0) \ge u_-\|\psi - \psi_0\|$ by Lemma \ref{lem:linear_curvature}\,
\item $T_3 = \bbM(\psi_0) - \bbM(\psi_0^s) \ge -u_+\|\psi_0^s - \psi_0\|  \ge  -\epsilon_1 \sigma_n$ where one can take $\epsilon_1$ as small as possible, as we have established $\|\psi_0^s - \psi_0\|/\sigma_n \rightarrow 0$. This follows by Lemma \ref{lem:linear_curvature} along with Lemma \ref{bandwidth}\,  
\item $T_4 = \bbM(\psi_0^s) -\bbM^s(\psi_0^s) \ge -K_1 \sigma_n$ by equation \ref{eq:lin_bound_1}. 
\end{enumerate}
Combining, we have 
\allowdisplaybreaks
\begin{align*}
\bbM^s(\psi) - \bbM^s(\psi_0^s) & \ge u_-\|\psi - \psi_0\| -(2K_1 + \epsilon_1) \sigma_n \\
& \ge ( u_-/2)\|\psi - \psi_0\| \hspace{0.2in} \left[\text{If} \ \|\psi - \psi_0\| \ge \frac{2(2K_1 + \epsilon_1)}{u_-}\sigma_n\right] \\
& \ge ( u_-/4)\|\psi - \psi_0^s\| 
\end{align*}
where the last inequality holds for all large $n$ as proved in Lemma \ref{bandwidth}. Using Lemma \ref{bandwidth} again, we conclude that for any pair of positive constants $(\epsilon_1, \epsilon_2)$:  
$$\|\psi - \psi_0^s\| \ge \left(\frac{2(2K_1 + \epsilon_1)}{u_-}+\epsilon_2\right)\sigma_n \Rightarrow \|\psi - \psi_0\| \ge \frac{2(2K_1 + \epsilon_1)}{u_-}\sigma_n$$ for all large $n$, which implies: 
\begin{align}
& \bbM^s(\psi) - \bbM^s(\psi_0^s) \notag \\
& \ge (u_-/4) \|\psi - \psi_0^s\| \mathds{1}\left(\|\psi - \psi_0^s\| \ge \left(\frac{2(2K_1 + \epsilon_1)}{u_-}+\epsilon_2\right)\sigma_n \right) \notag \\
\label{lb2} & \ge (u_-/4) \|\psi - \psi_0^s\| \mathds{1}\left(\frac{\|\psi - \psi_0^s\|}{\sigma_n} \ge \left(\frac{7K_1}{u_-}\right) \right)  \hspace{0.2in} [\text{for appropriate specifications of} \ \epsilon_1, \epsilon_2] \notag \\
& :=  (u_-/4) \|\psi - \psi_0^s\| \mathds{1}\left(\frac{\|\psi - \psi_0^s\|}{\sigma_n} \ge \cK \right)
\end{align}

\noindent
In the next part, we find the lower bound when $\|\psi - \psi^0_s\| \le \cK \sigma_n$. For the quadratic curvature, we perform a two step Taylor expansion: Define $\eta = (\psi - \psi_0)/\sigma_n$. We have: 
\allowdisplaybreaks 
\begin{align}
& \nabla^2\bbM^s(\psi) \notag\\
& = \frac{\beta_0 - \alpha_0}{2}\frac{1}{\sigma_n^2} \bbE\left(\tilde{Q}\tilde{Q}^{\top} K''\left(\frac{Q^{\top}\psi}{\sigma_n}\right)\left\{\mathds{1}(Q^{\top}\psi_0 \le 0) - \mathds{1}(Q^{\top}\psi_0 \ge 0)\right\}\right) \notag\\
& = \frac{\beta_0 - \alpha_0}{2}\frac{1}{\sigma_n^2} \bbE\left(\tilde{Q}\tilde{Q}^{\top} K''\left(\frac{Q^{\top}\psi_0}{\sigma_n} + \tilde{Q}^{\top}\tilde \eta \right)\left\{\mathds{1}(Q^{\top}\psi_0 \le 0) - \mathds{1}(Q^{\top}\psi_0 \ge 0)\right\}\right) \notag\\
& = \frac{\beta_0 - \alpha_0}{2}\frac{1}{\sigma_n^2} \bbE\left[\tilde{Q}\tilde{Q}^{\top} \left[\int_{-\infty}^{0} K''\left(\frac{z}{\sigma_n} + \tilde{Q}^{\top}\tilde \eta \right) f_0(z |\tilde{Q}) \ dz \right. \right. \notag \\ 
& \left. \left. \qquad \qquad \qquad \qquad -\int_{0}^{\infty} K''\left(\frac{z}{\sigma_n} + \tilde{Q}^{\top}\tilde \eta \right) f_0(z | \tilde{Q}) \ dz \right]\right] \notag\\
& = \frac{\beta_0 - \alpha_0}{2}\frac{1}{\sigma_n} \bbE\left[\tilde{Q}\tilde{Q}^{\top} \left[\int_{-\infty}^{0} K''\left(t+ \tilde{Q}^{\top}\tilde \eta \right) f_0(\sigma_n t |\tilde{Q}) \ dt \right. \right. \notag \\
& \left. \left. \qquad \qquad \qquad \qquad - \int_{0}^{\infty} K''\left(t + \tilde{Q}^{\top}\tilde \eta \right) f_0(\sigma_n t | \tilde{Q}) \ dt \right]\right] \notag\\
& = \frac{\beta_0 - \alpha_0}{2}\frac{1}{\sigma_n} \bbE\left[\tilde{Q}\tilde{Q}^{\top} f_0(0| \tilde{Q})\left[\int_{-\infty}^{0} K''\left(t+ \tilde{Q}^{\top}\tilde \eta \right) \ dt \right. \right. \notag \\
& \left. \left. \qquad \qquad \qquad \qquad - \int_{0}^{\infty} K''\left(t + \tilde{Q}^{\top}\tilde \eta \right) \ dt \right]\right] + R \notag\\
\label{eq:quad_eq_1} & =(\beta_0 - \alpha_0)\frac{1}{\sigma_n}\bbE\left[\tilde{Q}\tilde{Q}^{\top} f_0(0| \tilde{Q})K'(\tilde{Q}^{\top}\tilde \eta)\right] + R \,.
\end{align}
As we want a lower bound on the set $\|\psi - \psi^0_s\| \le \cK \sigma_n$, we have $\|\eta\| \le \cK$. For the rest of the analysis, define 
\begin{align*}
\Lambda: (v_1, v_2) \mapsto  \inf_{\|v_1\| = 1, \|v_2\| \le \cK} \bbE_{\tilde X}\left[|v_1^{\top}\tilde{Q}|^2 f(0|\tilde{Q})K'(\tilde{Q}^{\top}v_2) \right]
\end{align*}
Clearly $\Lambda \ge 0$ and continuous on a compact set, hence its infimum is attained. Suppose $\Lambda(v_1, v_2) = 0$ for some $v_1, v_2$. Then we have: 
\begin{align*}
\bbE\left[|v_1^{\top}\tilde{Q}|^2 f(0|\tilde{Q})K'(\tilde{Q}^{\top}v_2)  \right] = 0 \,,
\end{align*}
which further implies $|\tilde v_1^{\top}\tilde X| = 0$ almost surely and violates Assumption \ref{as:eigenval_bound}. Hence, our claim is demonstrated. On the other hand, for the remainder term of equation \eqref{eq:quad_eq_1}: 
fix $\nu \in S^{p-1}$. Then: 
\allowdisplaybreaks
\begin{align}
& \left| \nu^{\top} R \nu \right| \notag \\
& = \left|\frac{1}{\sigma_n} \bbE\left[\left(\nu^{\top}\tilde{Q}\right)^2 \left[\int_{-\infty}^{0} K''\left(t+ \tilde{Q}^{\top}\tilde \eta \right) (f_0(\sigma_n t |\tilde{Q}) - f_0(0|\tilde{Q})) \ dt  \right. \right. \right. \notag \\
& \qquad \qquad \qquad \qquad \left. \left. \left. - \int_{0}^{\infty} K''\left(t + \tilde{Q}^{\top}\tilde \eta \right) (f_0(\sigma_n t |\tilde{Q}) - f_0(0|\tilde{Q})) \ dt \right]\right]\right| \notag\\
& \le  \bbE \left[\left(\nu^{\top}\tilde{Q}\right)^2h(\tilde{Q}) \int_{-\infty}^{\infty} \left|K''\left(t+ \tilde{Q}^{\top}\tilde \eta \right)\right| |t| \ dt\right] \notag\\
& \le  \bbE \left[\left(\nu^{\top}\tilde{Q}\right)^2h(\tilde{Q}) \int_{-1}^{1} \left|K''\left(t\right)\right| |t - \tilde{Q}^{\top}\tilde \eta | \ dt\right] \notag\\
\label{eq:quad_eq_3} & \le  \bbE \left[\left(\nu^{\top}\tilde{Q}\right)^2h(\tilde{Q})(1+ \|\tilde{Q}\|/2\kappa) \int_{-1}^{1} \left|K''\left(t\right)\right| \ dt\right] = C_1 \hspace{0.2in} [\text{say}]
\end{align}
by Assumption \ref{as:distribution} and Assumption \ref{as:derivative_bound}. By a two-step Taylor expansion, we have: 
\begin{align*}
\bbM^s(\psi) - \bbM^s(\psi_0^s) & =  \frac12 (\psi - \psi_0^s)^{\top} \nabla^2\bbM^s(\psi^*_n) (\psi - \psi_0^s)  \\
& \ge \left(\min_{\|v_1\| = 1, \|v_2 \| \le \cK} \Lambda(v_1, v_2)\right) \frac{\|\psi - \psi_0^s\|^2}{2\sigma_n} - \frac{C_1\sigma_n}{2} \, \frac{\|\psi - \psi_0^s\|^2_2}{\sigma_n} \\
& \gtrsim \frac{\|\psi - \psi_0^s\|^2_2}{\sigma_n} \,
\end{align*}
This concludes the proof of the curvature. 
\\\\
\noindent 
Finally, we bound the modulus of continuity:
$$\bbE\left(\sup_{d_n(\psi, \psi_0^s) \le \delta} \left|(\mathbb{M}^s_n-\bbM^s)(\psi) - (\mathbb{M}^s_n-\bbM^s)(\psi_n)\right|\right) \,.$$ 
The proof is similar to that of Lemma \ref{lem:rate_smooth} and therefore we sketch the main steps briefly. Define the estimating function $f_\psi$ as: 
$$
f_\psi(Y, Q) = (Y - \gamma)\left(1 - K\left(\frac{Q^{\top}\psi}{\sigma_n}\right)\right) 
$$
and the collection of functions $\cF_\zeta = \{f_\psi - f_{\psi_0^n}: d_n(\psi, \psi_0^s) \le \delta\}$. That $\cF_\zeta$ has finite VC dimension follows from the same argument used to show $\cG_n$ has finite VC dimension in the proof of Lemma \ref{conv-prob}. Now to bound modulus of continuity, we use Lemma 2.14.1 of \cite{vdvw96}, which implies: 
$$
\sqrt{n}\bbE\left(\sup_{d_n(\psi, \psi_0^s) \le \delta} \left|(\mathbb{M}^s_n-\bbM^s)(\psi) - (\mathbb{M}^s_n-\bbM^s)(\psi_n)\right|\right)  \lesssim \cJ(1, \cF_\zeta) \sqrt{PF_\zeta^2}
$$
where $F_\zeta(Y, Q)$ is the envelope of $\cF_\zeta$ defined as: 
\begin{align*}
F_\zeta(Y, Q) & = \sup_{d_*(\psi, \psi_0^s) \le \zeta}\left|(Y - \gamma)\left(K\left(\frac{Q^{\top}\psi^s_0}{\sigma_n}\right)-K\left(\frac{Q^{\top}\psi}{\sigma_n}\right)\right)\right| \\
& = \left|(Y - \gamma)\right| \sup_{d_*(\psi, \psi_0^s) \le \zeta} \left|\left(K\left(\frac{Q^{\top}\psi^s_0}{\sigma_n}\right)-K\left(\frac{Q^{\top}\psi}{\sigma_n}\right)\right)\right|
\end{align*}
and $\cJ(1, \cF_\zeta)$ is the entropy integral which can be bounded above by a constant independent of $n$ as the class $\cF_\zeta$ has finite VC dimension. As in the proof of Lemma \ref{lem:rate_smooth}, we here consider two separate cases: (1) $\zeta \le \sqrt{\cK \sigma_n}$ and (2) $\zeta > \sqrt{\cK \sigma_n}$. In the first case, we have $\sup_{d_n(\psi, \psi_0^s) \le \zeta} \|\psi_ - \psi_0^s\| = \zeta \sqrt{\sigma_n}$. This further implies: 
\begin{align*}
    & \sup_{d_*(\psi, \psi_0^s) \le \zeta} \left|\left\{K\left(\frac{Q^{\top}\psi_0^s}{\sigma_n}\right) -  K\left(\frac{Q^{\top}\psi}{\sigma_n}\right)\right\}\right|^2 \\
    & \le \max\left\{\left|\left\{K\left(\frac{Q^{\top}\psi_0^s}{\sigma_n}\right) -  K\left(\frac{Q^{\top}\psi_0^s}{\sigma_n} + \|\tilde Q\|\frac{\zeta}{\sqrt{\sigma_n}}\right)\right\}\right|^2, \right. \\
    & \qquad \qquad \qquad \qquad \left. \left|\left\{K\left(\frac{Q^{\top}\psi_0^s}{\sigma_n}\right) -  K\left(\frac{Q^{\top}\psi_0^s}{\sigma_n} - \|\tilde Q\|\frac{\zeta}{\sqrt{\sigma_n}}\right)\right\}\right|^2\right\} \\
    & := \max\{T_1, T_2\} \,.
\end{align*}
Therefore to bound $\bbE[F_\zeta^2(Y, Q)]$ is equivalent to bounding both $\bbE[(Y- \gamma)^2 T_1]$ and $\bbE[(Y - \gamma)^2 T_2]$ separately, which, in turn equivalent to bound $\bbE[T_1]$ and $\bbE[T_2]$, as $|Y - \gamma| \le 1$. These bounds follows from similar calculation as of Lemma \ref{lem:rate_smooth}, hence skipped. Finally we have in this case, $$
\bbE[F_\zeta^2(Y, Q)] \lesssim \zeta \sqrt{\sigma_n} \,.
$$ 
The other case, when $\zeta > \sqrt{\cK \sigma_n}$ also follows by similar calculation of Lemma \ref{lem:rate_smooth}, which yields: 
$$
\bbE[F_\zeta^2(Y, Q)] \lesssim \zeta^2 \,.
$$

\noindent
Using this in the maximal inequality yields: 
\begin{align*}
\sqrt{n}\bbE\left(\sup_{d_n(\psi, \psi_0) \le \delta} \left|\mathbb{M}_n(\psi - \psi_n) - \bbM^s(\psi - \psi_n)\right|\right) & \lesssim \sqrt{\zeta}\sigma^{1/4}_n\mathds{1}_{\zeta \le \sqrt{\cK \sigma_n}} +  \zeta \mathds{1}_{\zeta > \sqrt{\cK \sigma_n}} \\
& := \phi_n(\zeta) \,
\end{align*}
This implies (following the same argument as of Lemma \ref{lem:rate_smooth}): 
$$
n^{2/3}\sigma_n^{-1/3}d^2(\hat \psi^s, \psi_0^s) = O_p(1) \,.
$$
Now as $n^{2/3}\sigma_n^{-1/3} \gg \sigma_n^{-1}$, we have: 
$$
\frac{1}{\sigma_n}d_n^2(\hat \psi^s, \psi_0^s) = o_p(1) \,.
$$
which further indicates
\begin{align}
\label{rate1} & n^{2/3}\sigma_n^{-1/3}\left[\frac{\|\hat \psi^s - \psi_0^s\|^2}{\sigma_n} \mathds{1}(\|\hat \psi^s - \psi_0^s\| \le \cK\sigma_n) \right. \notag \\
& \qquad \qquad \qquad \left. + \|\hat \psi^s - \psi_0^s\|  \mathds{1}(\|\hat \psi^s - \psi_0^s\|\ge \cK\sigma_n)\right] = O_P(1)
\end{align}
This implies: 
\begin{enumerate}
\item $\frac{n^{2/3}}{\sigma_n^{4/3}}\|\hat \psi^s - \psi_0^s\| \mathds{1}(\|\hat \psi^s - \psi_0^s\|\le \cK\sigma_n) = O_P(1)$
\item $\frac{n^{2/3}}{\sigma_n^{1/3}}\|\hat \psi^s - \psi_0^s\| \mathds{1}(\|\hat \psi^s - \psi_0^s\| \ge \cK\sigma_n) = O_P(1)$
\end{enumerate}
Therefore: 
\begin{align*}
& \frac{n^{2/3}}{\sigma_n^{4/3}}\|\hat \psi^s - \psi_0^s\| \mathds{1}(\|\hat \psi^s - \psi_0^s\| \le \cK\sigma_n) \\
& \qquad \qquad \qquad + \frac{n^{2/3}}{\sigma_n^{1/3}}\|\hat \psi^s - \psi_0^s\| \mathds{1}(\|\hat \psi^s - \psi_0^s\| \ge \cK\sigma_n) = O_p(1) \,.
\end{align*}
i.e. 
$$
\left(\frac{n^{2/3}}{\sigma_n^{4/3}} \wedge  \frac{n^{2/3}}{\sigma_n^{1/3}}\right)\|\hat \psi^s - \psi_0^s\| = O_p(1) \,.
$$
Now $(n^{2/3}/\sigma_n^{4/3} \gg 1/\sigma_n$ as long as $n^{2/3} \gg \sigma_n^{1/3}$ which is obviously true. On the other hand, $n^{2/3}/\sigma_n^{1/3} \gg 1/\sigma_n$ iff $n\sigma_n \gg 1$ which is also true as per our assumption. Therefore we have: 
$$
\frac{\|\hat \psi^s - \psi_0^s\|}{\sigma_n} = O_p(1) \,.
$$
This completes the proof. 

%
%
%
%
%
\end{proof}

\newpage

\bibliography{D_C_CP, mybib}{}
\bibliographystyle{plain}

\end{document}